\documentclass[11pt]{amsart}
\usepackage{amsmath}
\usepackage{amssymb}
\usepackage{amsthm}
\usepackage{latexsym}
\usepackage{graphicx}
\usepackage{hyperref}
\usepackage{enumerate}
\usepackage[all]{xy}

\newtheorem{thm}{Theorem}[section]
\newtheorem{cor}[thm]{Corollary}
\newtheorem{lem}[thm]{Lemma}
\newtheorem{prop}[thm]{Proposition}

\newtheorem{thmintro}{Theorem}
\newtheorem{conj}[thmintro]{Conjecture}

\newcommand{\N}{\mathbb N}
\newcommand{\Z}{\mathbb Z}
\newcommand{\Q}{\mathbb Q}
\newcommand{\R}{\mathbb R}
\newcommand{\C}{\mathbb C}
\newcommand{\mf}{\mathfrak}
\newcommand{\mc}{\mathcal}
\newcommand{\mb}{\mathbf}
\newcommand{\mh}{\mathbb}
\def\Irr{{\rm Irr}}
\newcommand{\mr}{\mathrm}
\newcommand{\enuma}[1]{\begin{enumerate}[\textup{(}a\textup{)}] {#1} \end{enumerate}}
\newcommand{\Fr}{\mathrm{Frob}}
\newcommand{\Sc}{\mathrm{sc}}
\newcommand{\ad}{\mathrm{ad}}
\newcommand{\cusp}{\mathrm{cusp}}
\newcommand{\nr}{\mathrm{nr}}
\newcommand{\Wr}{\mathrm{wr}}

\newcommand{\Rep}{\mathrm{Rep}}
\newcommand{\af}{\mathrm{aff}}
\def\tor{{\rm tor}}
\newcommand{\unip}{\mathrm{unip}}
\newcommand{\der}{\mathrm{der}}
\newcommand{\matje}[4]{\left(\begin{smallmatrix} #1 & #2 \\ 
#3 & #4 \end{smallmatrix}\right)}
\newcommand{\Mod}{\mathrm{Mod}}
\newcommand{\Hom}{\mathrm{Hom}}
\newcommand{\mi}{\text{\bf -}}
\newcommand{\sfa}{\mathsf{a}}
\newcommand{\sfb}{\mathsf{b}}
\newcommand{\sfq}{\mathsf{q}}

\begin{document}

\title[On unipotent representations of ramified $p$-adic groups]{On unipotent representations \\ 
of ramified $p$-adic groups}
\date{\today}
\thanks{The author was supported by a NWO Vidi grant "A Hecke algebra approach to the 
local Langlands correspondence" (nr. 639.032.528).}
\maketitle

\begin{center}
{\Large Maarten Solleveld} \\[1mm]
IMAPP, Radboud Universiteit \\
Heyendaalseweg 135, 6525AJ Nijmegen, the Netherlands \\
email: m.solleveld@science.ru.nl \\[3mm]
\end{center}

\

\begin{abstract}
Let $G$ be any connected reductive group over a non-archimedean local field. We analyse 
the unipotent representations of $G$, in particular in the cases where $G$ is ramified. We 
establish a local Langlands correspondence for this class of representations, and we show that 
it satisfies all the desiderata of Borel as well as the conjecture of Hiraga, Ichino and Ikeda 
about formal degrees.

This generalizes work of Lusztig and of Feng, Opdam and the author, to reductive groups that 
do not necessarily split over an unramified extension of the ground field.\\[3mm]
\end{abstract}

\tableofcontents

\section*{Introduction}

Let $F$ be a non-archimedean local field and let $\mc G$ be a connected reductive $F$-group.
We consider smooth, complex representations of the group $G = \mc G (F)$. An irreducible 
smooth $G$-representation $\pi$ is called unipotent if there exists a parahoric subgroup
$P_{\mf f} \subset G$ and an irreducible $P_{\mf f}$-representation $\sigma$, which is
inflated from a cuspidal unipotent representation of the finite reductive quotient of 
$P_{\mf f}$, such that $\pi |_{P_{\mf f}}$ contains $\sigma$.

The study of unipotent representations of $p$-adic groups was initiated by Morris 
\cite{Mor1,Mor2} and Lusztig \cite{LusUni1,LusUni2}. In a series of papers 
\cite{FeOp,FOS1,Feng,SolLLC,Opd18,FOS2} Yongqi Feng, Eric Opdam and the author
investigated various aspects of these representations: Hecke algebras, classification,
formal degrees, L-packets. This culminated in a proof of a local Langlands correspondence
for this class of representations.

However, all this was worked out under the assumption that $\mc G$ splits over the maximal 
unramified extension $F_\nr$ of $F$. In that case $\mc G$ already splits over a finite
unramified extension of $F$. In the opposite case, where $\mc G$ does not split over
$F_\nr$, we call $\mc G$ ramified (over $F$).

On the one hand, the assumption that $\mc G$ is $F_\nr$-split is reasonable: unipotent 
$G$-representations come from unipotent representations over the residue field $k_F$, 
and extensions of $k_F$ correspond naturally to unramified extensions of $F$. This enables 
one to regard a cuspidal unipotent representation of $\mc G (F)$ as a member of a family, 
indexed by the finite unramified extensions of $F$. Over a finite field an analogue is 
known from \cite{Lus-Che}.

On the other hand, in examples of ramified simple $F$-groups the unipotent representations
look the same as for $F$-groups that are not ramified, see \cite{Mor2}. 
While general depth zero representations of ramified $F$-groups may very well be more 
intricate, for unipotent representations it is not easy to spot what difficulties could be 
created by ramification of the ground field. In any case, many of the nice properties of 
unipotent representations were already expected to hold for all connected reductive 
$F$-groups. For instance, an enhanced L-parameter should correspond to a unipotent
representation if and only if it is unramified (that is, trivial on the inertia subgroup 
$\mb I_F$ of the Weil group $\mb W_F$). 

Summarising, the restriction to $F_\nr$-split groups in the study of unipotent 
representations seems to be made mainly for technical convenience. In the current paper we 
will prove the main results of \cite{LusUni1,LusUni2,FOS1,SolLLC,FOS2} for ramified simple 
$p$-adic groups, and then generalize them to arbitrary connected reductive $F$-groups.
Before we summarise our main conclusions below, we need to introduce some notations.

We denote the set of irreducible $G$-representations by $\Irr (G)$, and we often add a 
subscript ``unip" for unipotent and subscript ``cusp" to indicate cuspidality. 
Let ${}^L G = G^\vee \rtimes \mb W_F$ be the dual L-group of $G$. To a Langlands parameter
$\phi$ for $G$ we associate a finite group $\mc S_\phi$ as in \cite{Art,AMS1}.
An enhancement of $\phi$ is an irreducible representation $\rho$ of $\mc S_\phi$. We
denote the collection of $G$-relevant enhanced L-parameters (considered modulo 
$G^\vee$-conjugation) by $\Phi_e (G)$. Then $\Phi_{\nr,e}(G)$ denotes the subset of
$\Phi_e (G)$ given by the condition $\phi |_{\mb I_F} = \mr{id}_{\mb I_F}$. 

\begin{thmintro}\label{thm:1}
Let $\mc G$ be a connected reductive group over a non-archimedean local field $F$ and
write $G = \mc G(F)$. There exists a bijection
\[
\begin{array}{ccc}
\Irr_\unip (G) & \longrightarrow & \Phi_{\nr,e}(G) \\
\pi & \mapsto & (\phi_\pi, \rho_\pi) \\
\pi (\phi,\rho) & \text{\rotatebox[origin=c]{180}{$\mapsto$}} & (\phi,\rho)
\end{array} .
\]
We can construct such a bijection for every group $G$ of this kind, in a compatible
way. The resulting family of bijections has the following properties:
\begin{enumerate}[(a)]
\item Compatibility with direct products of reductive $F$-groups.
\item Equivariance with respect to the canonical actions of the group $X_\Wr (G)$ of
weakly unramified characters of $G$.
\item The central character of $\pi$ equals the character of $Z(G)$ determined by $\phi_\pi$.
\item $\pi$ is tempered if and only if $\phi_\pi$ is bounded.
\item $\pi$ is essentially square-integrable if and only if $\phi_\pi$ is discrete.
\item $\pi$ is supercuspidal if and only if $(\phi_\pi, \rho_\pi)$ is cuspidal.
\item The analogous bijections for the Levi subgroups of $G$ and the cuspidal
support maps form a commutative diagram
\[
\begin{array}{ccc}
\Irr_\unip (G) & \longrightarrow & \Phi_{\nr,e}(G) \\
\downarrow & & \downarrow \\
\bigsqcup_M \Irr_{\cusp,\unip}(M) \big/ N_G (M)  & 
\longrightarrow & \bigsqcup_M \Phi_{\nr,\cusp}(M) \big/ N_{G^\vee} ({}^L M) 
\end{array} .
\]
Here $M$ runs over a collection of representatives for the conjugacy classes of Levi subgroups of $G$.
\item Suppose that $P = M U$ is a parabolic subgroup of $G$ and that 
$(\phi^M,\rho^M) \in \Phi_{\nr,e}(M)$ is bounded. Let $\phi$ be $\phi^M$ considered as L-parameter
for $G$. Then the normalized parabolically induced 
representation $I_P^G \pi (\phi^M,\rho^M)$ is a direct sum of representations $\pi (\phi,\rho)$, 
with multiplicities $[\rho^M : \rho]_{\mc S_{\phi}^M}$.
\item Compatibility with the Langlands classification for representations of reductive groups
and the Langlands classification for enhanced L-parameters \cite{SiZi}.
\item Compatibility with restriction of scalars of reductive groups over 
non-archimedean local fields.
\item Let $\eta : \tilde{\mc G} \to \mc G$ be a homomorphism of connected reductive $F$-groups, 
such that the kernel of $\textup{d}\eta : \mr{Lie}(\tilde{\mc G}) \to \mr{Lie}(\mc G)$ is central 
and the cokernel of $\eta$ is a commutative $F$-group. Let ${}^L \eta : {}^L \tilde G \to {}^L G$ 
be a dual homomorphism and let $\phi \in \Phi_\nr (G)$. 

Then the L-packet $\Pi_{{}^L \eta \circ \phi}(\tilde G)$ consists precisely of the constituents 
of the completely reducible $\tilde G$-representations $\eta^* (\pi)$ with $\pi \in \Pi_\phi (G)$.
\item The HII conjecture \cite{HII} holds for tempered unipotent $G$-representations.
\end{enumerate}
Moreover the properties (a), (c), (k) and (l) uniquely determine the surjection 
\[
\Irr_\unip (G) \to \Phi_\nr (G): \pi \mapsto \phi_\pi ,
\]
up to twisting by weakly unramified characters of $G$ that are trivial on $Z(G)$.
\end{thmintro}

We regard Theorem \ref{thm:1} as a local Langlands correspondence (LLC) for unipotent 
representations. We note that parts (b), (c), (d), (e) and (k) are precisely the desiderata
formulated by Borel \cite[\S 10]{Bor}. For the unexplained notions in the other parts we refer 
to \cite{SolLLC}. 

Let us phrase part (l) about Plancherel densities more precisely. We fix an additive 
character $\psi : F \to \C^\times$ of level zero (by \cite{HII} that can be done without
loss of generality). As in \cite{HII} that gives rise to a Haar measure $\mu_{G,\psi}$ on $G$, 
which we however normalize as in \cite[(A.25)]{FOS1}. 

Let $\mc P = \mc M \mc U$ be a parabolic $F$-subgroup of $\mc G$,
with Levi factor $\mc M$ and unipotent radical $\mc U$. Let $\pi \in \Irr (M)$
be square-integrable modulo centre and let $X_{\mr{unr}} (M)$ be the group of unitary
unramified characters of $M$. Let $\mc O = X_{\mr{unr}} (M) \pi \subset \Irr (M)$
be the orbit in $\Irr (M)$ of $\pi$, under twists by $X_{\mr{unr}} (M)$. We define
a Haar measure d$\mc O$ on $\mc O$ as in 
\cite[p. 239 and 302]{Wal}. This also provides a Haar measure on the family
of (finite length) $G$-representations $I_P^G (\pi')$ with $\pi' \in \mc O$.

Let $Z(\mc G)_s$ be the maximal $F$-split central torus of $\mc G$, with dual group
$Z(G^\vee)^{\mb W_F,\circ}$. We denote the adjoint representation of ${}^L M$ on 
Lie$\big( G^\vee) / \mr{Lie}(Z(M^\vee)^{\mb W_K} \big)$ by $\mr{Ad}_{G^\vee,M^\vee}$. 
We compute $\gamma$-factors with respect to the 
Haar measure on $F$ that gives the ring of integers $\mf o_F$ volume 1.

\begin{conj}\label{conj:HII}
\textup{\cite[\S 1.5]{HII}} \\
Suppose that the enhanced L-parameter of $\pi$ is $(\phi_\pi,\rho_\pi) \in \Phi_e (M)$. \\
Then the Plancherel density at $I_P^G (\pi) \in \mr{Rep}(G)$ is
\[
c_M \dim (\rho_\pi) \, \big| \pi_0 \big( Z_{(G / Z(G)_s)^\vee} (\phi_\pi) \big) \big|^{-1} | 
\gamma (0,\mr{Ad}_{G^\vee,M^\vee} \circ \phi_\pi,\psi)| \, \textup{d}\mc O (\pi) ,
\]
for some constant $c_M \in \R_{>0}$ independent of $F$ and $\mc O$.

Moreover, with the above normalizations of Haar measures $c_M$ equals 1.
\end{conj}

It is also interesting to consider Theorem \ref{thm:1} for all inner twists of a given
quasi-split group simultaneously. That is done best with the rigid inner twists from
\cite{Kal1,Dil}. In that setting we replace $\mc S_\phi$ by a slightly different component
group $\mc S_\phi^+$ and we write
\[
\Phi^+ ({}^L G) = \big\{ (\phi,\rho^+) : \phi \in \Phi (G), \rho^+ \in \Irr (\mc S_\phi^+) \big\}.
\]
For a rigid inner twist $(\mc G^z,z)$ of $\mc G$, we also replace $\Phi_e (G^z)$ by a 
slightly different set $\Phi^+ (G^z,z)$ of relevant enhanced L-parameters. The (disjoint) union
of the sets $\Phi^+ (G^z,z)$, over all $z$ in a set $H^1 (\mc E,Z(\mc G_\der) \to \mc G)$
parametrizing the equivalence classes of rigid inner twists of $\mc G$, is precisely 
$\Phi^+ ({}^L G)$.

We check (in Section \ref{sec:rigid}) that the new setup is essentially equivalent to the setup 
used so far, with the bonus that it is a bit more canonical. It follows that Theorem \ref{thm:1} 
is also valid in terms of rigid inner twists and the associated enhancements of L-parameters.

\begin{thmintro}\label{thm:3}
\textup{(see Theorem \ref{thm:7.3})} \\
The union of the instances of Theorem \ref{thm:1} for all rigid inner twists of a quasi-split
connected reductive $F$-group $\mc G$ gives a bijection
\[
\Phi_\nr^+ ({}^L G) \longrightarrow 
\bigsqcup\nolimits_{z \in H^1 (\mc E,Z(\mc G_\der) \to \mc G)} \Irr_\unip (G^z) .
\]
\end{thmintro}

It is believed (or hoped) that in the local Langlands program enhanced L-pa\-ra\-me\-ters are
in bijection with the irreducible representations of all rigid inner twists of a given reductive
$p$-adic group. Theorem \ref{thm:3} beautifully confirms this for unramified L-parameters
and unipotent representations.\\

Let us explain our strategy to prove Theorem \ref{thm:1}. The papers 
\cite{FeOp,FOS1,Opd18,FOS2} all use reduction to the case of simple (adjoint) $F$-groups,
so that is where we start. Like in \cite{Mor1,Mor2,LusUni1,LusUni2} we want to analyse
the parahoric subgroups $P_{\mf f}$ of $G$, their (cuspidal) unipotent representations
$\sigma$ and the Hecke algebras determined by a type of the form $(P_{\mf f}, \sigma)$.
The main trick stems from a remark of Lusztig \cite[\S 10.13]{LusUni2}: for every ramified
simple $F$-group $\mc G$ there exists a $F_\nr$-split simple ``companion group"
$\mc G'$. which has the same local index and the same relative local Dynkin diagram as
$\mc G$ (up to the direction of some arrows in these diagrams). That determines $\mc G'$
up to isogeny, and we fix it by requiring that $(Z(G'^\vee)^{\mb I_F})_\Fr \cong 
(Z(G^\vee)^{\mb I_F})_\Fr$. We will construct a LLC for $\Irr (G)_\unip$ via
$\Irr (G')_\unip$ (for which it is known already).

In Section \ref{sec:list} we provide an overview of all possible $\mc G$ and $\mc G'$.
It turns out that, although $\mc G'$ is connected when $\mc G$ is adjoint, sometimes
$\mc G' = \mc G^{'\circ} \times \{\pm 1\}$. 

This setup provides a bijection between the $G$-orbits of facets in the Bruhat--Tits
building of $\mc G(F)$ and the analogous set for $G' = \mc G' (F)$, say $\mf f \mapsto \mf f'$.
We call a representation of the parahoric subgroup $P_{\mf f}$ unipotent (resp. cuspidal)
if it arises by inflation from a unipotent (resp. cuspidal) representation of the finite
reductive quotient of $P_{\mf f}$. We show in Theorem \ref{thm:2.3} that the relation
between the ramified simple $F$-group $\mc G$ and its companion group $\mc G'$ gives rise
to a bijection
\begin{equation}\label{eq:1} 
\Irr (P_{\mf f})_\unip \longleftrightarrow \Irr (P_{\mf f'})_\unip : \sigma \mapsto \sigma' .
\end{equation}
Notice that this actually is a statement about finite reductive groups. Let $\hat P_{\mf f}$ 
be the pointwise stabilizer of $\mf f$ in $G$. Then \eqref{eq:1} can be extended to a bijection
\begin{equation}\label{eq:2}
\Irr (\hat P_{\mf f})_\unip \longleftrightarrow \Irr (\hat P_{\mf f'}) : 
\hat \sigma \mapsto \hat \sigma' .
\end{equation}
For cuspidal representations \eqref{eq:2} induces to a bijection
\begin{equation}\label{eq:3}
\Irr (G)_{\cusp,\unip} \longleftrightarrow \Irr (G')_{\cusp,\unip}
\end{equation}
which is almost canonical (Corollary \ref{cor:2.4}).

In Section \ref{sec:Hecke} we compare the non-cuspidal unipotent representations of
$G$ and $G'$. Let $\hat \sigma \in \Irr (\hat P_{\mf f})_{\unip,\cusp}$, so that 
$(\hat P_{\mf f}, \hat \sigma)$ is a type for a Bernstein component of unipotent
$G$-representations \cite{Mor3}. The Bernstein block Rep$(G)_{(\hat P_{\mf f},\hat \sigma)}$
is equivalent with the module category of the Hecke algebra $\mc H (G,\hat P_{\mf f},
\hat \sigma)$. We prove in Theorem \ref{thm:3.1} that \eqref{eq:2} canonically induces
an algebra isomorphism
\begin{equation}\label{eq:4}
\mc H (G',\hat P_{\mf f'}, \hat \sigma') \to \mc H (G,\hat P_{\mf f},\hat \sigma) .
\end{equation}
These Hecke algebras are essential for everything in the non-cuspidal cases.
By \cite[\S 1]{LusUni1} there are equivalences of categories
\begin{equation}\label{eq:5}
\Rep (G)_\unip = \prod_{\{ (\hat P_{\mf f}, \hat \sigma) \} / G\text{-conjugation}}
\hspace{-11mm} \Rep (G)_{(\hat P_{\mf f}, \hat \sigma)} \cong
\prod_{\{ (\hat P_{\mf f}, \hat \sigma) \} / G\text{-conjugation}} 
\hspace{-11mm} \mr{Mod} \big( \mc H (G,\hat P_{\mf f}, \hat \sigma) \big) .
\end{equation}
Combining that with \eqref{eq:4} for all possible $(\hat P_{\mf f}, \hat \sigma)$
yields an equivalence of categories
\begin{equation}\label{eq:6}
\Rep (G)_\unip \longrightarrow \Rep (G')_\unip .
\end{equation}
Although \eqref{eq:6} is not entirely canonical, we do show that it preserves several
properties of representations. 

Now that the situation for unipotent representations of simple $F$-groups is under control,
we turn to the complex dual groups and L-parameters for $G$ and $G'$. The most 
important observation (checked case-by-case with the list from Section \ref{sec:list})
is Lemma \ref{lem:4.1}: there exists a canonical isomorphism 
$G^{'\vee} \to (G^\vee)^{\mb I_F}$. This induces a canonical bijection
\begin{equation}\label{eq:7}
\Phi_{\nr,e}(G') \longrightarrow \Phi_{\nr,e}(G) ,
\end{equation}
which preserves relevant properties of enhanced L-parameters 
(Proposition \ref{prop:4.4} and Lemma \ref{lem:4.7}). From \eqref{eq:6}, \eqref{eq:7}
and Theorem \ref{thm:1} for $G'$ \cite{SolLLC,FOS2} we deduce Theorem \ref{thm:1}
for ramified simple groups. More precisely, we establish some properties of the bijection
\begin{equation}\label{eq:8}
\Irr (G)_\unip \longrightarrow \Phi_{\nr,e}(G) ,
\end{equation}
not yet all. In particular equations \eqref{eq:4}--\eqref{eq:8} mean that the main
results of \cite{LusUni1,LusUni2} are now available for all simple $F$-groups.\\

With the case of simple $F$-groups settled, we embark on the study of supercuspidal
unipotent representations of connected reductive $F$-groups (Section \ref{sec:supercusp}).
For a $F_\nr$-split group $\mc G$, Theorem \ref{thm:1} was proven for 
$\Irr (G)_{\cusp,\unip}$ in \cite{FOS1} (again with most but not yet all properties).
We aim to generalize the arguments from \cite{FOS1} to possibly ramified connected
reductive $F$-groups. It is only at this stage that the differences caused by ramification
of field extensions force substantial modifications of previous strategies. 

Assume for the moment that the centre of $\mc G$ is $F$-anisotropic. When $\mc G$ is in
addition $F_\nr$-split, the derived group $G_\der$ has the same supercuspidal unipotent 
representations as $G$ \cite[\S 15]{FOS1}. That is not true for ramified $F$-groups. 
Related to that $(G^\vee)^{\mb I_F}$ need not be connected.
Let $\sfq : \mc G \to \mc G_\ad$ be the quotient map to the adjoint group and let
$\sfq^\vee : {G_\ad}^\vee \to G^\vee$ be the dual homomorphism. The set 
$\sfq^\vee (\Phi_\nr (G_\ad))$ constists precisely of the $\phi \in \Phi_\nr (G)$ with
$\Phi (\Fr) \in G^{\vee,\mb I_F,\circ}$. 

Similarly, the natural map $X_\Wr (G_\ad) \to X_\Wr (G)$ need not be surjective for 
ramified groups. We denote its image by $X_\Wr (G_\ad,G)$. In Lemmas \ref{lem:5.3}
and \ref{lem:5.4} we show that there are natural bijections
\begin{equation}\label{eq:9}
\begin{aligned}
& X_\Wr (G) \underset{X_\Wr (G_\ad,G)}{\times} \Irr (G / Z(G))_\unip \longrightarrow
\Irr (G)_\unip , \\
& X_\Wr (G) \underset{X_\Wr (G_\ad,G)}{\times} \sfq^\vee (\Phi_\nr (G_\ad))
\longrightarrow \Phi_\nr (G). 
\end{aligned}
\end{equation}
Using \eqref{eq:9} and the case of adjoint groups, the proof of Theorem \ref{thm:1}
for supercuspidal unipotent representations of $F_\nr$-split semisimple $F$-groups 
in \cite{FOS1} generalizes readily to connected reductive $F$-groups with anisotropic 
centre. The step from there to arbitrary connected reductive $F$-groups is easy anyway. 
This completes the proof of Theorem \ref{thm:1} for supercuspidal unipotent 
representations, except for the properties (d), (e), (g), (h), (i) and (k).

In Section \ref{sec:LLC} we set out to generalize the local Langlands correspondence
for $\Irr (G)_\unip$ in \cite{SolLLC} from $F_\nr$-split to arbitrary connected
reductive $F$-groups. With the above results on the adjoint and the cuspidal cases,
that is straightforward. The arguments from \cite{SolLLC} yield Theorem \ref{thm:1}
for $\Irr (G)_\unip$, except for the properties (k) and (l).

Property (k), about the behaviour of unipotent representations upon pullback along
certain homomorphisms of reductive groups, is an instance of the main results of
\cite{SolFunct}. We only have to verify that the $F_\nr$-split assumption made in
\cite[\S 7]{SolFunct} can be lifted. That requires a few remarks about the small
modifications in the ramified case. We formulate a more precise version of 
property (k) in Theorem \ref{thm:6.2}.

Finally we deal with the essential uniqueness of our LLC and with property (l), the 
HII conjecture \ref{conj:HII}. For $F_\nr$-split groups the latter is the main result 
of \cite{FOS2}. We check that the arguments from
\cite{FOS2} can be generalized to possibly ramified connected reductive $F$-groups.\\[2mm]

\textbf{Acknowledgment.}
We express our gratitude to the referee for the helpful comments and the detailed report.

\section{List of ramified simple groups}
\label{sec:list}

Let $F$ be a non-archimedean local field with ring of integers $\mf o_F$ and
a uniformizer $\varpi_F$. Let $k_F = \mf o_F / \varpi_F \mf o_F$ be its residue
field, of cardinality $q_F$. We fix a separable closure $F_s$ and assume that all
separable extensions of $F$ are realized in $F_s$. Let $F_\nr$ be the maximal 
unramified extension of $F$. Let $\mb W_F \subset \mr{Gal}(F_s / F)$ be the Weil
group of $F$ and let $\Fr$ be a geometric Frobenius element. 
Let $\mb I_F = \mr{Gal}(F_s / F_\nr) \subset \mb W_F$ be the inertia subgroup, 
so that $\mb W_F / \mb I_F \cong \Z$ is generated by $\Fr$.

Let $\mc G$ be a connected reductive $F$-group and pick a maximal $F$-split torus  
$\mc S$ in $\mc G$. Let $\mc T_\nr$ be a maximal $F_\nr$-split torus in 
$Z_{\mc G}(\mc S)$ defined over $F$ -- such a torus exists by
\cite[\S 1.10]{Tit2}. Then $\mc T := Z_{\mc G}(\mc T_\nr)$ is a maximal torus
of $\mc G$, defined over $F$ and containing $\mc T_\nr$ and $\mc S$.
Let $\Phi (\mc G, \mc T)$ be the associated root system. We also fix a Borel 
subgroup $\mc B$ of $\mc G$ containing $\mc T$ and defined over $\mc F_\nr$, 
which determines bases $\Delta_{\mc T}$ of $\Phi (\mc G, \mc T)$, $\Delta_\nr$ of
$\Phi (\mc G,\mc T_\nr)$ and $\Delta$ of $\Phi (\mc G, \mc S)$.

We call $G = \mc G (F)$:
\begin{itemize}
\item \emph{unramified} if $\mc G$ is quasi-split and splits over $F_\nr$;
\item \emph{ramified} if $\mc G$ does not split over $F_\nr$.
\end{itemize}
Unfortunately this common terminology does not exhaust the possibilities: some 
$F_\nr$-split groups are neither ramified nor unramified. Earlier work on unipotent 
representations of reductive $p$-adic groups applied to groups that are not ramified.
In this section we present the list of simple ramified $F$-groups of adjoint type,
obtained from \cite{Tit1,Tit2}. For each such group we provide some useful data, which 
we describe next. We follow the conventions and terminology from \cite{Tit2,SolLLC}.

The Bruhat--Tits building $\mc B (\mc G,F)$ has an apartment $\mh A_{\mc S} =
X_* (\mc S) \otimes_\Z \R$ associated to $\mc S$. The walls of $\mh A_S$ determine
an affine root system $\Sigma$, which naturally projects onto the finite root system
$\Phi (\mc G, \mc S)$. Similarly the Bruhat--Tits building $\mc B (\mc G,F_\nr)$ has 
an apartment $\mh A_\nr = X_* (\mc T_\nr) \otimes_\Z \R$ associated to $\mc T_\nr$. 
The walls of $\mh A_\nr$ determine an affine root system $\Sigma_\nr$, which naturally 
projects onto $\Phi (\mc G, \mc T_\nr)$. We recall from \cite[2.6.1]{Tit2} that
\begin{equation}\label{eq:1.1}
\mc B (\mc G,F) =  \mc B (\mc G,F_\nr)^{\mathrm{Gal}(F_\nr / F)} = 
\mc B (\mc G,F_\nr)^\Fr \quad \text{and} \quad \mh A_{\mc S} = \mh A_\nr^\Fr . 
\end{equation}
Let $C_\nr$ be a $\Fr$-stable chamber in $\mh A_\nr$ whose closure contains 0 and 
which (as far as possible) lies in the positive Weyl chamber determined by $\mc B$. 
The walls of $C_\nr$ provide a basis $\Delta_{\nr,\af}$ of $\Sigma_\nr$, which 
naturally surjects to $\Delta_\nr$. The group $\mr{Gal}(F_s / F_\nr)$ acts naturally
on $C_\nr$ and hence on $\Delta_{\nr,\af}$. The Dynkin diagram of 
$(\Sigma_\nr,\Delta_{\nr,\af})$,
together with the action of $\Fr$, is called the \emph{local index} of $\mc G (F)$.

By \eqref{eq:1.1} there exists a unique chamber $C_0$ in $\mh A_{\mc S}$ containing
$C_\nr \cap \mh A_{\mc S}$. The walls of $C_0$ yield a basis $\Delta_\af$ of 
$\Sigma$ which projects onto $\Delta$. By construction $\Delta_\af$ consists of the
restrictions of $\Delta_{\nr,\af}$ to $\mh A_{\mc S}$. As $\mc G$ is simple, $|\Delta_\af| 
= |\Delta| + 1$ and $|\Delta_{\nr,\af}| = |\Delta_\nr| + 1$. The relative local
Dynkin diagram of $\mc G (F)$ is defined as the Dynkin diagram of $(\Sigma,\Delta_\af)$.

We will also need a group called $\Omega$ or $\Omega_G$, which can be described 
in several equivalent ways \cite[Appendix]{PaRa}:
\begin{itemize}
\item $\Irr \big( (Z(G^\vee)^{\mb I_F})_\Fr \big)$, where $G^\vee$ is the complex
dual group of $G$;
\item $G$ modulo the kernel of the Kottwitz homomorphism 
$G \to \Irr \big( (Z(G^\vee)^{\mb I_F})_\Fr \big)$;
\item $G$ modulo the subgroup generated by all parahoric subgroups of $G$;
\item the stabilizer of $C_0$ in the group $N_G (S) / (Z_G (S) \cap P_{C_0})$, where 
$P_{C_0} \subset G$ denotes the Iwahori subgroup associated to $C_0$;
\item $\big( \big( X_* (\mc T) / \Z \Phi^\vee (\mc G,\mc T) \big)_{\mb I_F} \big)^\Fr$.
\end{itemize}
The group $\Omega_G$ acts naturally on the relative local Dynkin diagram of $\mc G (F)$.
We say that a character of $G$ is weakly unramified if it is trivial on every 
parahoric subgroup of $G$. By the above, the group $X_\Wr (G)$ of all such characters
is naturally isomorphic with $\Irr (\Omega_G)$ and with $(Z(G^\vee)^{\mb I_F})_\Fr$.

We say that $\mc G$ is simple if it is simple as $F_s$-group. If it is merely simple
as $F$-group, we call it $F$-simple.
For every ramified simple $F$-group $\mc G$ we give a $F_\nr$-split
``companion" $F$-group $\mc G'$. It is determined by the following requirements:
\begin{itemize}
\item There exists a $\Fr$-equivariant bijection between $\Delta_{\nr,\af}$ for
$\mc G$ and $\mc G'$, which preserves the number of bonds in the Dynkin diagram(s) 
of $(\Sigma_\nr, \Delta_{\nr,\af})$. Thus the local index of $\mc G'$ is the same
as that of $\mc G$, except that the directions of some arrows may differ.
In particular this gives a bijection from the relative local Dynkin diagram for
$\mc G' (F)$ to that for $\mc G (F)$.
\item There is an isomorphism $\Omega_{G'} \cong \Omega_G$, which renders the 
bijection between the relative local Dynkin diagrams $\Omega_G$-equivariant.
\end{itemize}
We specify a bijection between $\Delta_{\nr,\af}$ for $\mc G$ and $\mc G'$
by marking one special vertex on both sides. In most cases 0 is a special vertex,
then we pick that one. This also determines one marked vertex of $\Delta_\af$
(and one of $\Delta'_\af$). The remainder of the relative local Dynkin diagram 
$\Delta_\af$ is canonically in bijection with $\Delta$, so the bijection
$\Delta_\af \longleftrightarrow \Delta'_\af$ induces a bijection 
$\Delta \longleftrightarrow \Delta'$.

These relations between the groups $\mc G$ and $\mc G'$ lead to many similarities.
For instance, their parabolic $F$-subgroups can be compared. Namely,
it is well-known that the $G$-conjugacy classes of parabolic $F$-subgroups of $\mc G$ 
are naturally in bijection with the power set of $\Delta$ \cite[Theorem 15.4.6]{Spr}. 
The same holds for $\mc G'$. Hence the above bijection $\Delta \longleftrightarrow \Delta'$
induces a bijection from the set of conjugacy classes
of parabolic $F$-subgroups of $\mc G$ to the analogous set for $\mc G'$. Furthermore
every conjugacy class of parabolic subgroups of $\mc G$ contains a unique standard 
parabolic subgroup (with respect to $\mc B$). We will denote the resulting bijection 
between standard parabolic $F$-subgroups by $\mc P \mapsto \mc P'$.\\

Apart from the above data, the group $\mc G$ depends on the choice of a suitable
field extension of $F$. Below $F^{(2)}$ is the unique unramified quadratic
extension of $F$ and $E$ (resp. $E^{(2)}$) denotes a ramified separable quadratic 
extension of $F$ (resp. of $F^{(2)}$). 

We use the names for the local indices from \cite{Tit1,Tit2}. For each name, we start
with the adjoint group $\mc G$ of that type, and its companion group. After that, we
list the groups isogenous to $\mc G$. These have the same local index and relative
local Dynkin diagram as $\mc G$ and their companion groups are isogenous to $\mc G'$, 
but they have a smaller group $\Omega_G$. 

\subsection{$\mathbf{B\mi C_n}$} \
\label{par:B-Cn}

$\mc G = PU_{2n}$, quasi-split over $F$, split over $E$

Local index and relative local Dynkin diagram: \
\raisebox{-5mm}{\includegraphics[width=4cm]{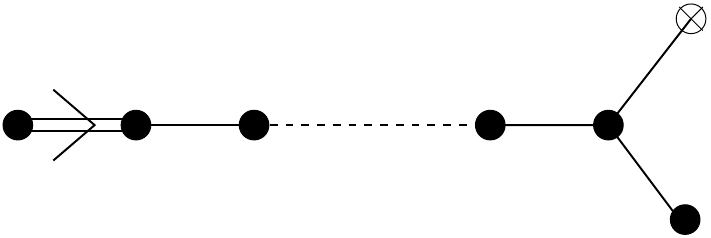}}

Trivial $\Fr$-action

$\Omega_G$ has two elements, it exchanges the two legs on the right hand side\\

$\mc G' = SO_{2n+1}$, $F$-split

Local index and relative local Dynkin diagram: \
\raisebox{-5mm}{\includegraphics[width=4cm]{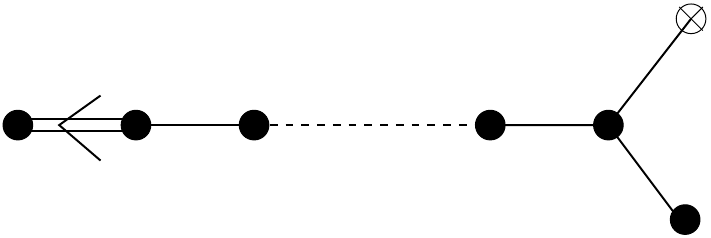}}

Groups isogenous to $PU_{2n}$ fit in a sequence $SU_{2n} \to \mc G \to PU_{2n}$.

Such a group is determined by the order of its schematic centre, call that $d$.
\[
\mc G' = \left\{\begin{array}{ll}
SO_{2n+1} & \text{if } d \text{ is odd} \\
Spin_{2n+1} \times \{\pm 1\} & \text{if } d \text{ is even and } 2n/d \text{ is even} \\
Spin_{2n+1} & \text{if } d \text{ is even and } 2n/d \text{ is odd} 
\end{array}
\right.
\]

In the first two cases $\Omega_G$ has order two and it acts on the diagram as for $PU_{2n}$, 

in the second case $|\Omega_G| = 2$ and it acts trivially on the diagram, while in the 

third case $|\Omega_G| = 1$.

\subsection{$\mathbf{C\mi BC_n}$} \
\label{par:C-BCn}

$\mc G = PU_{2n+1}$, quasi-split over $F$, split over $E$

Local index and relative local Dynkin diagram: \
\raisebox{-1mm}{\includegraphics[width=4cm]{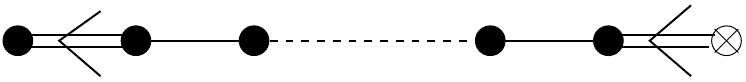}}

Trivial $\Fr$-action

$\Omega_G$ has one element\\

$\mc G' = Sp_{2n}$, $F$-split

Local index and relative local Dynkin diagram: \
\raisebox{-1mm}{\includegraphics[width=4cm]{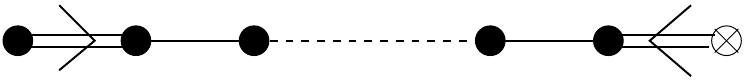}}\\

Groups isogenous to $PU_{2n+1}$ fit in a sequence 
$SU_{2n+1} \to \mc G \to PU_{2n+1}$. 

Then $|\Omega_G| = 1$ and $\mc G' = Sp_{2n}$.

\subsection{$\mathbf{C\mi B_n}$} \
\label{par:C-Bn}

$\mc G = PSO^*_{2n+2}$, quasi-split over $F$, split over $E$

Local index and relative local Dynkin diagram: \
\raisebox{-1mm}{\includegraphics[width=4cm]{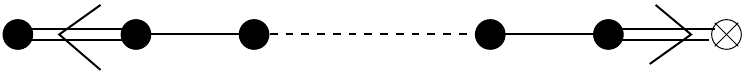}}

Trivial $\Fr$-action

$\Omega_G$ has two elements, it reflects in the middle of the diagram\\

$\mc G' = PSp_{2n}$, $F$-split

Local index and relative local Dynkin diagram: \
\raisebox{-1mm}{\includegraphics[width=4cm]{Cn.pdf}}\\

Isogenous group $\mc G = Spin_{2n + 2}^*$: $\Omega_G = 1$ and $\mc G' = Sp_{2n}$.

The isogenous group $\mc G = SO_{2n+2}^*$ has $\Omega_G$ of order 2, but acting trivially 

on the diagram. We take $\mc G' = Sp_{2n} \times \{\pm 1\}$.

\subsection{$\mathbf{ {}^2 B\mi C_n}$} \
\label{par:2B-Cn}

$\mc G = PU_{2n}$, not quasi-split over $F$, quasi-split over $F^{(2)}$, split over $E^{(2)}$

Local index: \
\raisebox{-5mm}{\includegraphics[width=4cm]{BCn.pdf}}

$\Fr$ exchanges the two legs on the right hand side

Relative local Dynkin diagram: 
\raisebox{-1mm}{\includegraphics[width=4cm]{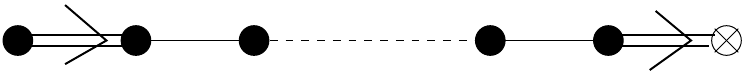}}

$\Omega_G$ has two elements, it acts trivially on the diagram\\

$\mc G' = SO_{2n+1}$, not split over $F$, split over $F^{(2)}$

Local index: \
\raisebox{-5mm}{\includegraphics[width=4cm]{Bn.pdf}} 

Relative local Dynkin diagram: 
\raisebox{-1mm}{\includegraphics[width=4cm]{CBn.pdf}}\\

For isogenous groups $\mc G$, the situation is as for $\mathbf{B\mi C_n}$, 

except that $\mc G'$ splits over $F^{(2)}$ but not over $F$.

\subsection{$\mathbf{ {}^2 C\mi B_{2n}}$} \
\label{par:2C-Beven}

$\mc G = PSO^*_{4n}$, not quasi-split over $F$, quasi-split over $F^{(2)}$, split over $E^{(2)}$

Local index: \
\raisebox{-5mm}{\includegraphics[width=4cm]{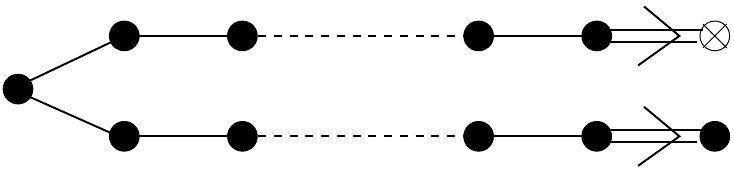}}

$\Fr$ exchanges the upper and the lower row

Relative local Dynkin diagram: 
\raisebox{-1mm}{\includegraphics[width=4cm]{2BCn.pdf}}

$\Omega_G$ has two elements, it acts trivially on the diagram\\

$\mc G' = PSp_{4n-2}$, not split over $F$, split over $F^{(2)}$

Local index: \
\raisebox{-5mm}{\includegraphics[width=4cm]{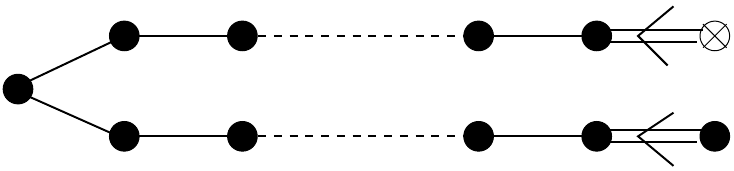}}

Relative local Dynkin diagram: 
\raisebox{-1mm}{\includegraphics[width=4cm]{Cn.pdf}}\\

Isogenous group $\mc G = Spin_{4n}^*$: $|\Omega_G| = 1$ and $\mc G' = Sp_{4n-2}$.

Isogenous group $\mc G = SO_{4n}^*$: $|\Omega_G| = 2$ and $\mc G' = Sp_{4n-2} \times \{\pm 1\}$.

\subsection{$\mathbf{ {}^2 C\mi B_{2n+1}}$} \
\label{par:2C-Bodd}

$\mc G = PSO^*_{4n+2}$, not quasi-split over $F$, quasi-split over $F^{(2)}$, split over $E^{(2)}$

Local index: \
\raisebox{-5mm}{\includegraphics[width=4cm]{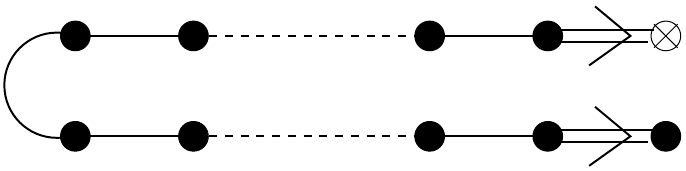}}

$\Fr$ exchanges the upper and the lower row

Relative local Dynkin diagram: 
\raisebox{-1mm}{\includegraphics[width=4cm]{CBn.pdf}}

$\Omega_G$ has two elements, it acts trivially on the diagram\\

$\mc G' = PSp_{4n}$, not split over $F$, split over $F^{(2)}$

Local index: \
\raisebox{-5mm}{\includegraphics[width=4cm]{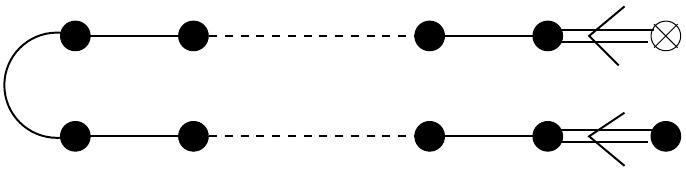}}

Relative local Dynkin diagram: 
\raisebox{-1mm}{\includegraphics[width=4cm]{CBCn.pdf}}\\

Isogenous group $\mc G = Spin_{4n+2}^*$: $|\Omega_G| = 1$ and $\mc G' = Sp_{4n}$.

Isogenous group $\mc G = SO_{4n+2}^*$: $|\Omega_G| = 2$ and $\mc G' = Sp_{4n} \times \{\pm 1\}$.

\subsection{$\mathbf{F_4^I}$} \
\label{par:F4I}

$\mc G = {}^2 E_{6,\ad}$, quasi-split over $F$, split over $E$

Local index and relative local Dynkin diagram: \
\raisebox{-1mm}{\includegraphics[width=3cm]{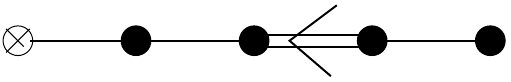}}

Trivial $\Fr$-action

$\Omega_G$ has one element\\

$\mc G' = F_4$, $F$-split

Local index and relative local Dynkin diagram: \
\raisebox{-1mm}{\includegraphics[width=3cm]{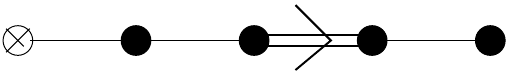}}\\

Isogenous group $\mc G = {}^2 E_{6,\Sc}$: $|\Omega_G| = 1$ and $\mc G' = F_4$.

\subsection{$\mathbf{G_2^I}$} \
\label{par:G2I}

$\mc G = {}^r D_{4,\ad}$ with $r = 3$ or $r =6$, quasi-split over $F$, split over a Galois extension 

$E' / F$ of degree $r$ such that the unique degree 3 subextension of $F$ is ramified

Local index and relative local Dynkin diagram: \
\raisebox{-1mm}{\includegraphics[width=16mm]{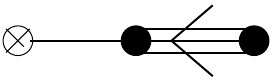}}

Trivial $\Fr$-action

$\Omega_G$ has one element\\

$\mc G' = G_2$, $F$-split

Local index and relative local Dynkin diagram: \
\raisebox{-1mm}{\includegraphics[width=16mm]{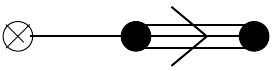}} \\

Isogenous group $\mc G = {}^r D_{4,\Sc}$: $|\Omega_G| = 1$ and $\mc G' = G_2$. \\

To fulfill the requirement $\Omega_G \cong \Omega_{G'}$ we sometimes needed a disconnected 
group $\mc G' = \mc G^{'\circ} \times \{\pm 1\}$. All standard operations for connected 
reductive groups extend naturally such $\mc G'$. For instance, the Bruhat--Tits building of
$(\mc G',F)$ is that of $(\mc G^{'\circ},F)$, with $\{\pm 1\}$ acting trivially. 
In particular a parahoric subgroup of $\mc G' (F)$ is a parahoric subgroup of 
$\mc G^{'\circ}(F)$. The complex dual group of $\mc G^{'\circ} \times \{\pm 1\}$ is defined 
to be $(\mc G^{'\circ})^\vee \times \{\pm 1\}$. In Lemma \ref{lem:4.1} we will see that 
this fits well, which motivates our choice of $\mc G'$.

\section{Matching of unipotent representations}
\label{sec:unirep}

By construction there is a canonical bijection between the set of faces of $C_\nr$
and the collection of proper subsets of $\Delta_{\nr,\af}$. Explicitly, it
associates to a face $\mf f$ the set $J_{\mf f}$ of simple affine roots of
$\Sigma_\nr$ that vanish on $\mf f$. With \eqref{eq:1.1} and 
$\Delta_\af = \Delta_{\nr,\af} / \mr{Gal}(F_\nr / F)$ this leads to canonical
bijections between the following sets:
\begin{itemize}
\item proper subsets of $\Delta_\af$;
\item $\Fr$-stable proper subsets of $\Delta_{\nr,\af}$;
\item $\Fr$-stable faces of $C_\nr$;
\item faces of $C_0$.
\end{itemize}
Let $\mf f$ be a $\Fr$-stable face of $C_\nr$, identified with a face of $C_0$.
Bruhat and Tits \cite{BrTi2} associated to $\mf f$ an $\mf o_F$-group $\mc G_{\mf f}$,
such that $\mc G_{\mf f}^\circ (\mf o_F)$ equals the parahoric subgroup $P_{\mf f}
\subset G$ associated to $\mf f$ and $\mc G_{\mf f}(\mf o_F) = N_G (P_{\mf f})$ equals
the $G$-stabilizer of $\mf f$.

Let $\overline{\mc G_{\mf f}}$ be the maximal reductive quotient of $\mc G_{\mf f}$
as $k_F$-group. Then $\overline{\mc G_{\mf f}^\circ}(k_F) = P_{\mf f} / P_{\mf f}^+$, 
where $P_{\mf f}^+$ denotes the pro-unipotent radical of $P_{\mf f}$.

Let $\Omega_{G,\mf f}$ be the stabilizer of $\mf f$ in $\Omega_G$ (with respect to
its action on $\Delta_\af$). Then
\begin{equation}\label{eq:2.1}
\overline{\mc G_{\mf f}} / \overline{\mc G_{\mf f}^\circ} \cong N_G (P_{\mf f}) / P_{\mf f}
\cong \Omega_{G,\mf f} . 
\end{equation}
The algebraic group $\overline{\mc G_{\mf f}^\circ}$ splits over $k_{F_\nr}$ (an
algebraic closure of $k_F$) and its Dynkin diagram is the subdiagram of $\Delta_{\nr,\af}$
formed by the vertices in $J_{\mf f}$. Further, the isogeny class of the $k_F$-group
$\overline{\mc G_{\mf f}^\circ}$ is determined by the action of $\Fr$ on $J_{\mf f}$,
so it only depends on $\mf f$ and the local index of $\mc G$.

\begin{prop}\label{prop:2.1}
Let $\mc G$ be a ramified simple $F$-group and let $\mc G'$ be its $F_\nr$-split 
companion group, as in Section \ref{sec:list}. Let $\mf f$ be a $\Fr$-stable
face of $C_\nr$ and let $\mf f'$ be the face of $C'_\nr$ corresponding to it via the
bijection between the local indices of $\mc G$ and $\mc G'$.
Then the $k_F$-groups $\overline{\mc G_{\mf f}}$ and $\overline{\mc G'_{\mf f'}}$ have the
following in common:
\begin{itemize}
\item their Lie type, up to changing the direction of some arrows in the Dynkin diagram;
\item their dimension;
\item $|\overline{\mc G^\circ_{\mf f}}(k_F)| = |\overline{\mc G^{'\circ}_{\mf f'}}(k_F)|$;
\item $\Omega_{G,\mf f} \cong \Omega_{G',\mf f'}$.
\end{itemize}
\end{prop}
\begin{proof}
The setup from Section \ref{sec:list} provides a bijection between $J_{\mf f}$ and $J_{\mf f'}$,
which preserves the number of bonds in the Dynkin diagrams of $\overline{\mc G_{\mf f}}$ 
and $\overline{\mc G'_{\mf f'}}$. Decomposing into connected components, we get
$J_{\mf f} = \bigsqcup_i J_{\mf f}^i$ and $J_{\mf f'} = \bigsqcup_i J_{\mf f'}^i$ where the 
connected Dynkin diagram $J_{\mf f}^i$ is isomorphic to $J_{\mf f'}^i$ or to its dual.

Write $\overline{\mc G_{\mf f}^\circ}$ as an almost direct product of simple groups
$\overline{\mc G_{\mf f}^i}$, and similarly for $\overline{\mc G^{'\circ}_{\mf f'}}$.
Then $\overline{\mc G_{\mf f}^i}$ is isogenous to $\overline{\mc G^{'i}_{\mf f'}}$ or to
the dual group of $\overline{\mc G^{'i}_{\mf f'}}$ over $k_F$. Consequently
\[
\dim \overline{\mc G_{\mf f}^\circ} = \sum\nolimits_i \dim \overline{\mc G_{\mf f}^i} =
\sum\nolimits_i \dim \overline{\mc G^{'i}_{\mf f'}} = \dim \overline{\mc G^{'\circ}_{\mf f'}} .
\]
The number of elements of a connected reductive group over a finite field only depends
on the group up to isogeny \cite[Proposition 1.4.13.c]{GeMa}. It also does not change
if we replace a simple group by its dual group, by Chevalley's counting formulas
\cite[Theorem 1.6.7]{GeMa}. We deduce
\[
\big| \overline{\mc G^\circ_{\mf f}}(k_F) \big| = 
\prod\nolimits_i \big| \overline{\mc G^i_{\mf f}}(k_F) \big| = 
\prod\nolimits_i \big| \overline{\mc G^{'i}_{\mf f'}}(k_F) \big| = 
\big| \overline{\mc G^{'\circ}_{\mf f'}}(k_F) \big| .
\]
The claim about $\Omega_{G,\mf f}$ follows from the $\Omega_G$-equivariance of the bijection
between $\Delta_\af$ and $\Delta'_\af$.
\end{proof}

It is well-known that the conjugacy classes of parabolic $k_F$-subgroups of 
$\overline{\mc G^\circ_{\mf f}}$ are naturally in bijection with the subsets of the 
Dynkin diagram $J_{\mf f}$. Every conjugacy class of parabolic $k_F$-subgroups of 
$\overline{\mc G^\circ_{\mf f}}$ contains a unique parabolic subgroup $\mc P$ which is standard 
(with respect to the image of $\mc S$ and the basis $\Delta$ of $\Phi (\mc G,\mc S)$).
We denote the unique standard $k_F$-Levi factor of $\mc P$ by $\mc M$ and its unipotent
radical by $\mc U$.

The same holds for $\overline{\mc G^{'\circ}_{\mf f'}}$, and we have a bijection
between $J_{\mf f}$ and $J_{\mf f'}$. Thus $(\mc P,\mc M)$ determines a unique
standard parabolic pair $(\mc P',\mc M')$ for $\overline{\mc G^{'\circ}_{\mf f'}}$, defined
over $k_F$.

Let $\Irr_\unip (H)$ denote the collection of irreducible unipotent representations of a
group $H$. 

\begin{prop}\label{prop:2.2}
Use the notations from Proposition \ref{prop:2.1} and let $\mc P = \mc M \mc U$
be a standard parabolic $k_F$-subgroup $\mc P = \mc M \mc U$ of 
$\overline{\mc G^\circ_{\mf f}}(k_F)$.
\enuma{
\item There exists a natural bijection
\[ 
\Irr_{\unip,\cusp} (\mc M (k_F)) \longleftrightarrow \Irr_{\unip,\cusp} (\mc M' (k_F)) .
\]
\item There exists a bijection
\[ 
\Irr_\unip (\mc M (k_F)) \longleftrightarrow \Irr_\unip (\mc M' (k_F)) ,
\]
which preserves dimensions.
\item Extend part (b) to a bijection $\Rep_\unip (\mc M (k_F)) \longleftrightarrow 
\Rep_\unip (\mc M' (k_F))$ by making it additive.
The system of such bijections, with $\mc P$ running over all standard
parabolic $k_F$-subgroups of $\overline{\mc G^\circ_{\mf f}}(k_F)$, is compatible
with parabolic induction and Jacquet restriction.
}
\end{prop}
\begin{proof}
(a) By \cite[Proposition 3.15]{Lus-Che} $\Irr_\unip (\overline{\mc G^\circ_{\mf f}}(k_F))$
depends only on $\overline{\mc G^\circ_{\mf f}}$ up to isogeny. Using the constructions
from the proof of Proposition \ref{prop:2.1}, we obtain a natural bijection
\begin{equation}\label{eq:2.2}
\Irr_\unip (\overline{\mc G^\circ_{\mf f}}(k_F)) \longleftrightarrow
\prod\nolimits_i \Irr_\unip (\overline{\mc G^i_{\mf f}}(k_F)) ,
\end{equation}
and similarly for $\overline{\mc G^{'\circ}_{\mf f'}}(k_F)$.
The unipotent representations of $\overline{\mc G^i_{\mf f}}(k_F)$ are built from
the cuspidal unipotent representations of Levi factors $\mc M_i (k_F)$ of parabolic
subgroups $\mc P_i (k_F)$ and from Hecke algebras, see \cite[Theorem 3.26]{Lus-Che}.
Since $\overline{\mc G^i_{\mf f}}$ and $\overline{\mc G^{'i}_{\mf f'}}$ have the same
Dynkin diagram (up to the direction of arrows), the conjugacy class of 
$\mc P_i = \mc M_i \mc U_i$ corresponds to a unique conjugacy class of parabolic subgroups 
$\mc P'_i = \mc M'_i \mc U'_i$ of $\overline{\mc G^{'i}_{\mf f'}}$. Then $\mc M_i$ and
$\mc M'_i$ also have the same Dynkin diagram (up to the direction of arrows).

The classification of cuspidal unipotent representations in \cite[\S 3]{Lus-Che} shows
that for each such $\mc P_i$ there is a canonical bijection
\begin{equation}\label{eq:2.3}
\Irr_{\unip,\cusp}(\mc M_i (k_F)) \longleftrightarrow \Irr_{\unip,\cusp}(\mc M'_i (k_F)) ,
\end{equation}
which preserves dimensions. Moreover, if $\psi$ is any $\Fr$-equivariant automorphism of 
the Dynkin diagram of $\mc M_i$ (and hence also for $\mc M'_i$) and we lift it to automorphisms
of $\mc M_i (k_F)$ and of $\mc M'_i (k_F)$, then \eqref{eq:2.3} is $\psi$-equivariant.
These claims can be checked case-by-case from \cite{Lus-Che}. To make that easier, one may note 
that the list in Section \ref{sec:list} shows that no factors of Lie type $E_n$ are involved.
We conclude that the bijection \eqref{eq:2.3} is natural.

Now we combine \eqref{eq:2.3} with \eqref{eq:2.2} for $\Irr_{\cusp,\unip}(\mc M (k_F))$ and for
$\Irr_{\cusp,\unip}(\mc M' (k_F))$. \\
(b) Fix $\mc P_i = \mc M_i \mc U_i$ and $\rho \in \Irr_{\unip,\cusp}(\mc M_i (k_F))$, and let
$\rho'_i \in \Irr_{\unip,\cusp}(\mc M'_i (k_F))$ be its image under \eqref{eq:2.3}.
From \cite[Table II]{Lus-Che} we see that the Hecke algebra 
$\mr{End}_{\overline{\mc G^i_{\mf f}}(k_F)} \big( 
\mr{ind}^{\overline{\mc G^i_{\mf f}}(k_F)}_{\mc P_i (k_F)} \rho_i \big)$ for
$(\overline{\mc G^i_{\mf f}}, \mc P_i, \rho_i)$ is isomorphic to the Hecke algebra of
$(\overline{\mc G^{'i}_{\mf f'}}, \mc P'_i, \rho'_i)$. This works for all $(\mc P_i,\rho_i)$ 
and, as described in \cite[\S 3.25]{Lus-Che}, it gives rise to a bijection
\begin{equation}\label{eq:2.4}
\Irr_\unip \big( \overline{\mc G^i_{\mf f}}(k_F) \big) \longleftrightarrow 
\Irr_\unip \big( \overline{\mc G^{'i}_{\mf f'}}(k_F) \big) .
\end{equation}
By \cite[(3.26.1)]{Lus-Che} and Proposition \ref{prop:2.1}, \eqref{eq:2.4} preserves 
dimensions. Combine this with \eqref{eq:2.2} to get part (b) for 
$\overline{\mc G^i_{\mf f}}(k_F)$. For $\mc M(k_F)$ it can be shown in the same way.\\
(c) In the constructions for part (b) everything is obtained by parabolic induction from
the cuspidal level, followed by selecting suitable subrepresentations by means of Hecke
algebras. In view of the transitivity of parabolic induction, this setup entails that
the system of bijections $\Rep_\unip (\mc M (k_F)) \longleftrightarrow 
\Rep_\unip (\mc M' (k_F))$ is compatible with parabolic induction. Since Jacquet
restriction is the adjoint functor of parabolic induction, the system of bijections is
also compatible with that.
\end{proof}

By \eqref{eq:2.1} the group $\Omega_{G,\mf f}$ acts naturally on
$\Irr_\unip (\overline{\mc G^\circ_{\mf f}}(k_F))$.

\begin{thm}\label{thm:2.3}
\enuma{
\item The bijection 
\[
\Irr_\unip (\overline{\mc G^\circ_{\mf f}}(k_F)) \longleftrightarrow
\Irr_\unip (\overline{\mc G^{'\circ}_{\mf f'}}(k_F)) ,
\]
constructed in the proof of Proposition \ref{prop:2.2}.b is $\Omega_{G,\mf f}$-equivariant.
\item It extends to a bijection 
\[
\Irr_\unip (\overline{\mc G_{\mf f}}(k_F)) \longleftrightarrow 
\Irr_\unip (\overline{\mc G'_{\mf f'}}(k_F)) ,
\]
which preserves dimensions and cuspidality.
}
\end{thm}
\begin{proof}
Recall from Section \ref{sec:list} that $|\Omega_G| \leq 2$. 
When $|\Omega_{G,\mf f}| = 1$, \eqref{eq:2.1} shows that there is nothing to prove. 
So we may assume that $\Omega_{G,\mf f} = \{1,\omega\} \cong \Omega_{G',\mf f'}$,
where we have picked representatives for $\omega$ in $\overline{\mc G_{\mf f}}(k_F)$
and in $\overline{\mc G'_{\mf f'}}(k_F)$.\\
(a) The bijections \eqref{eq:2.3} combine to a $\Omega_{G,\mf f}$-equivariant 
bijection between
\begin{equation}\label{eq:2.5}
\left\{ \begin{array}{c}
(\mc P = \mc M \mc U, \rho) : \mc P \text{ parabolic } k_F\text{-subgroup of }
\overline{\mc G^\circ_{\mf f}} \\ 
\text{ with Levi factor } \mc M, \rho \in \Irr_{\unip,\cusp}(\mc M (k_F)) 
\end{array} \right\} \Big/ \overline{\mc G^\circ_{\mf f}}(k_F) \text{-conjugacy}
\end{equation}
to its analogue for $\overline{\mc G^{'\circ}_{\mf f'}}(k_F)$.

When $\omega$ does not stabilize the $\overline{\mc G^\circ_{\mf f}}(k_F)$-orbit of 
$(\mc P,\rho)$, the subsets of $\Irr_\unip (\overline{\mc G^\circ_{\mf f}}(k_F))$ 
associated to $(\mc P,\rho)$ and to $(\omega \mc P \omega^{-1}, \omega \cdot \rho)$ 
are disjoint \cite[\S 3.25]{Lus-Che}. In particular $\omega$ does not stabilize any 
representation in such a set.

Suppose now that $(\omega \mc P \omega^{-1}, \omega \cdot \rho)$ is 
$\overline{\mc G^\circ_{\mf f}}(k_F)$-conjugate to $(\mc P,\rho)$. Choosing another
representative for $\omega$ in $\overline{\mc G_{\mf f}}(k_F)$, we may assume that
$(\omega \mc P \omega^{-1}, \omega \cdot \rho) = (\mc P,\rho)$. Since the parabolic
subgroup $\mc P$ is its own normalizer in $\overline{\mc G_{\mf f}}$, this choice of
$\omega$ is unique up to inner automorphisms of $\mc P (k_F)$.
Now $\rho$ extends to a representation $\tilde \rho$ of 
\[
\langle \mc P (k_F),\omega \rangle = \mc P (k_F) \cup \omega \mc P (k_F).
\]
In fact there are two choices for $\tilde \rho$, differing by a quadratic character, but 
which one does not matter because we only need conjugation by $\tilde \rho (\omega)$.

Then $\Pi := \mr{ind}^{\overline{\mc G^\circ_{\mf f}}(k_F)}_{\mc P (k_F)} \rho$ 
extends to the $\overline{\mc G_{\mf f}}(k_F)$-representation $\tilde \Pi := 
\mr{ind}^{\overline{\mc G_{\mf f}}(k_F)}_{\langle \mc P (k_F), \omega \rangle} \tilde \rho$.
Conjugation by $\tilde \Pi (\omega)$ provides an automorphism $\psi_\omega$ of the 
Hecke algebra 
\[
\mc H := \mr{End}_{\overline{\mc G^\circ_{\mf f}}(k_F)} (\Pi) = 
\mr{End}_{\overline{\mc G^\circ_{\mf f}}(k_F)} (\tilde \Pi) .
\]
A $\pi \in \Irr_\unip (\overline{\mc G^\circ_{\mf f}}(k_F))$ associated to $(\mc P,\rho)$
corresponds to the irreducible $\mc H$-module $\mr{Hom}_{\overline{\mc G^\circ_{\mf f}}
(k_F)}(\Pi,\pi)$. Conversely, any $\pi_{\mc H} \in \Irr (\mc H)$ gives rise to
$E \otimes_{\mc H} \pi_{\mc H} \in \Irr_\unip (\overline{\mc G^\circ_{\mf f}}(k_F))$.
Under this correspondence the action of $\omega$ on 
$\Irr_\unip (\overline{\mc G^\circ_{\mf f}}(k_F))$ translates to the action of
$\psi_\omega$ on $\Irr (\mc H)$.

Given that $\omega$ stabilizes $(\mc P,\omega)$, the entire setup is canonical up to
inner automorphisms. The $\Omega_{G,\mf f}$-action can be described entirely with
data coming from the cuspidal level. Of course the same applies to 
$\overline{\mc G^{'\circ}_{\mf f'}}(k_F)$. Together with the $\Omega_{G,\mf f}$-equivariance
of the bijection involving \eqref{eq:2.5}, we deduce $\Omega_{G,\mf f}$-equivariance
in the desired generality.\\
(b) To extend the bijection from Proposition \ref{prop:2.2}.b to $\overline{\mc G_{\mf f}}(k_F)$ 
and $\overline{\mc G'_{\mf f'}}(k_F)$, we need Clifford theory with respect to the action
of $\Omega_{G,\mf f} \cong \Omega_{G',\mf f'}$ on 
$\Irr_\unip (\overline{\mc G^\circ_{\mf f}}(k_F))$ and 
$\Irr_\unip (\overline{\mc G^{'\circ}_{\mf f'}}(k_F))$.

When $\omega$ does not stabilize $\pi \in \Irr_\unip (\overline{\mc G^\circ_{\mf f}}(k_F))$,
the $\overline{\mc G^\circ_{\mf f}}(k_F)$-representation $\pi \oplus \omega \cdot \pi$
extends to an irreducible $\overline{\mc G_{\mf f}}(k_F)$-representation $\tilde \pi$.
In this way the pair $\{\pi,\omega \cdot \pi\}$ accounts for one element of
$\Irr_\unip (\overline{\mc G_{\mf f}}(k_F))$. 

When $\omega$ stabilizes $\pi \in \Irr_\unip (\overline{\mc G^\circ_{\mf f}}(k_F))$,
$\pi$ extends in precisely two ways to an irreducible representation of 
$\overline{\mc G_{\mf f}}(k_F)$. The two extensions $\pi_+, \pi_-$ are related by 
$\pi_- (\omega) = -\pi_+ (\omega)$. Thus $\pi$ gives rise to a pair $\{\pi_+,\pi_-\}$
in $\Irr_\unip (\overline{\mc G_{\mf f}}(k_F))$. Clifford theory tells us that every
element of $\Irr_\unip (\overline{\mc G_{\mf f}}(k_F))$ arises in a unique way from
one of these two constructions.

In view of part (a), Clifford theory 
works in the exactly same way for $\overline{\mc G'_{\mf f'}}(k_F)$. Denoting the
bijection from Proposition \ref{prop:2.2}.b by $\pi \mapsto \pi'$, we can extend it to 
\[
\Irr_\unip (\overline{\mc G_{\mf f}}(k_F)) \longleftrightarrow 
\Irr_\unip (\overline{\mc G'_{\mf f'}}(k_F)) 
\]
by sending either sending $\tilde \pi$ to $\tilde \pi'$ or sending $\pi_+$ to 
$\pi'_+$ and $\pi_-$ to $\pi'_-$. (Notice that this is not canonical, for we could 
just as well exchange $\pi'_+$ and $\pi'_-$.)
As dimensions and cuspidality are preserved in Proposition \ref{prop:2.2}, they
are preserved here as well.
\end{proof}

It will be handy to know how the bijections in Theorem \ref{thm:2.3} behave with
respect to outer automorphisms of $\mc G$ and $\mc G'$. From the list in Section 
\ref{sec:list} we see that $\mc G'$ has Lie type $B_n, C_n, F_4$ or $G_2$, so all its
automorphisms are inner. On the other hand, the group $\mc G$ does admit outer 
automorphisms. Requiring that they fix a pinning, outer automorphisms can be classified 
in terms of $\mb W_F$-equivariant automorphisms of the (absolute) Dynkin diagram of 
$(\mc G,\mc T)$, see \cite[Corollary 3.3]{SolFunct}. Then we call them \emph{diagram
automorphisms} of $\mc G (F)$.

In paragraphs \ref{par:B-Cn}--\ref{par:F4I} the absolute root datum of $\mc G$ 
admits exactly one nontrivial automorphism $\tau$. In paragraph \ref{par:G2I} there 
is no such automorphism for ${}^6 D_4$, and there are two for ${}^3 D_4$, say $\tau$ 
and $\tau^2$. In all these cases $\tau$ lifts to an automorphism of $\mc G (F)$ because 
$\mc G$ is either quasi-split or the unique inner twist of its quasi-split form. 

Recall that the $F$-points of $\mc G$ can be obtained from its $F_s$-points by
taking the invariants with respect to the $\mb W_F$-action that defines the structure
as $F$-group. This $\mb W_F$-action is a combination of the natural Galois action on
matrix coefficients and algebraic group automorphisms. In the cases under consideration,
one element of $\mb W_F$ acts as $g \mapsto \overline{\tau (g)}$, where the overline
indicates a field automorphism. It follows that on $\mc G (F)$, $\tau$ works out as
\begin{itemize}
\item the nontrivial field automorphism of $E/F$, applied to matrix coefficients,
for $\mathbf{B\mi C_n , C\mi BC_n , C\mi B_n, F_4^I}$;
\item the nontrivial field automorphism of $E^{(2)}/F^{(2)}$, applied to matrix 
coefficients, for $\mathbf{ {}^2 B\mi C_n , {}^2 C\mi B_{2n}, 
{}^2 C\mi B_{2n+1}}$;
\item one of the two nontrivial field automorphisms of $E'/F$, applied to matrix 
coefficients, for $\mathbf{G_2^I}$ with $\mc G$ of type ${}^3 D_4$. 
\end{itemize}

\begin{lem}\label{lem:2.5} 
The diagram automorphism $\tau$ of $\mc G (F)$ stabilizes the groups $P_{\mf f},
N_G (P_{\mf f}),\\ \overline{\mc G^\circ_{\mf f}}(k_F)$ and
$\overline{\mc G_{\mf f}}(k_F)$, as well as all their unipotent representations.
\end{lem}
\begin{proof}
The local index of $\mc G$ is obtained from the Dynkin diagram of $(\mc G,\mc T)$
by di\-vi\-ding out the $\mb I_F$-action. Here $\mb I_F$ acts via powers of $\tau$, so
$\tau \in \mr{Aut}(\mc G(F))$ acts trivially on the local index of $\mc G$. It
follows that $\tau$ stabilizes every face of $C_\nr$, and hence acts on
the four indicated groups.

The absolute Dynkin diagram of 
$\overline{\mc G^\circ_{\mf f}}$ is a subdiagram of $\Delta_{\nr,\af}$, and $\tau$ 
fixes that pointwise. Consequently $\tau$ acts on $\overline{\mc G^\circ_{\mf f}}$ by
an inner $k_F$-automorphism, that is, as conjugation by an element of the adjoint
group $(\overline{\mc G^\circ_{\mf f}})_\ad (k_F)$. It is known from 
\cite[Proposition 3.15]{Lus-Che} that every unipotent representation $\pi$ of 
$\overline{\mc G^\circ_{\mf f}}(k_F)$ extends to a representation of 
$(\overline{\mc G^\circ_{\mf f}})_\ad (k_F)$. This shows that $\tau$ stabilizes all 
unipotent representations of $\overline{\mc G^\circ_{\mf f}}(k_F)$ and of $P_{\mf f}$.

In the proof of Theorem \ref{thm:2.3}.b we saw how Clifford theory produces irreducible
unipotent representations of $\overline{\mc G_{\mf f}}(k_F)$ from those of 
$\overline{\mc G^\circ_{\mf f}}(k_F)$. The constructions over there work just as well
when we consider $\pi$ as $(\overline{\mc G^\circ_{\mf f}})_\ad (k_F)$-representation.
The extension $\overline{\mc G_{\mf f}}$ of $\overline{\mc G^\circ_{\mf f}}$ by
$\Omega_{G,\mf f}$ naturally induces an extension $(\overline{\mc G_{\mf f}})_\ad$ of
$(\overline{\mc G^\circ_{\mf f}})_\ad$ by $\Omega_{G,\mf f}$.
It follows that $\tilde \pi$, $\pi_+$ and $\pi_-$ are also representations of 
$(\overline{\mc G_{\mf f}})_\ad (k_F)$. In particular $\tau$ acts on them via an 
element of $(\overline{\mc G_{\mf f}})_\ad (k_F)$, so these representations are 
stabilized by $\tau$. Clifford theory tells us that these account for all irreducible
unipotent representations of $\overline{\mc G_{\mf f}}(k_F)$ and of $N_G (P_{\mf f})$.
\end{proof}

Let $P_{\mf f}$ be a maximal parahoric subgroup of $G$ and let 
$\sigma \in \Irr (P_{\mf f})$ be inflated from a cuspidal unipotent representation of 
$\overline{\mc G_{\mf f}^\circ}(k_F) = P_{\mf f} / P_{\mf f}^+$. As noted for instance
in \cite{LusUni1,Mor1,Mor2,MoPr2}, $\mr{ind}_{P_{\mf f}}^G \sigma$ is a direct sum of finitely
many supercuspidal $G$-representations. 

For a more precise description we choose an extension $\sigma^N$ of $\sigma$ to 
$N_G (P_{\mf f})$. That is always possible \cite[Proposition 4.6]{Mor2}, and any two such 
extensions differ by a character of $N_G (P_{\mf f}) / P_{\mf f} \cong \Omega_{G,\mf f}$:
\begin{equation}\label{eq:2.13}
\mr{ind}_{P_{\mf f}}^{N_G (P_{\mf f})} (\sigma ) = 
\bigoplus\nolimits_{\chi \in \Irr (\Omega_{G,\mf f})} \sigma^N \otimes \chi .
\end{equation}
Every supercuspidal unipotent $G$-representation is of the form
\begin{equation}\label{eq:2.6}
\mr{ind}_{N_G (P_{\mf f})}^G (\sigma^N ) . 
\end{equation}
Given a supercuspidal unipotent $G$-representation, the pair $(N_G (P_{\mf f}),\sigma^N)$ 
is unique up to conjugation.

Let $G^\vee$ be the complex dual group of $\mc G$, endowed with an action of Gal$(F_s/ F)$ 
(by pinned automorphisms) coming from the $F$-structure of $\mc G$. Then $\mf g^\vee = 
\mr{Lie}(G^\vee)$ is a representation of Gal$(F_s/F)$ and of $\mb W_F$. We denote its Artin 
conductor by $\mb a (\mf g^\vee)$. We note that by \cite[(18) and \S 3.4]{GrRe} this equals the
Artin conductor of the motive of $\mc G$. For $F_\nr$-split groups $\mb a (\mf g^\vee) = 0$, 
while for ramified groups $\mb a (\mf g^\vee) \in \Z_{>0}$.

Let $|\omega_G|$ be the canonical Haar measure on $G$ from \cite[\S 5]{GaGr}. 
Let $\psi : F \to \C^\times$ be an additive character. Recall that the order of $\psi$ is the 
largest $n \in \Z$ such that $\psi (f) = 1$ for all $f \in F$ of valuation $\geq -n$. 
Following \cite[(A.25)]{FOS1} we normalize the Haar measure on $G$ as
\begin{equation}\label{eq:2.11}
\mu_{G,\psi} = q_F^{-(\mb a (\mf g^\vee) + \mr{ord}(\psi) \dim (\mc G / Z(\mc G)_s) ) / 2} |\omega_G |.
\end{equation}
Actually this is a correction on \cite[(A.25)]{FOS1}, because there it said dim$(\mc G)$ instead of
dim$(\mc G / Z(\mc G)_s)$. In \cite{FOS1} the authors took the reduction from $G$ to $G / Z(G)_s$
(a standard way to handle formal degrees when the centre of $G$ is not compact) for granted, so that
\cite[(A.25)]{FOS1} should only be used when $Z(G)$ is compact. For \cite[Lemmas A.5 and A.6]{FOS1}
this does not matter, they work equally well with \eqref{eq:2.11} and with \cite[(A.25)]{FOS1}.
Apart from that, \cite{FOS1} only uses these normalized Haar measures in Proposition A.7. The proof
of that, especially the part about reduction to the case where $\psi$ has order 0, is actually 
based on the normalization \eqref{eq:2.11}.

Unless explicitly mentioned otherwise, we assume that $\psi$ has order 0. 
For $F_\nr$-split groups \eqref{eq:2.11} agrees with the normalizations in
\cite{Gro,GaGr,HII}, while for ramified groups the correction term $q_F^{-\mb a (\mf g^\vee)/2}$ 
is needed to relate formal degrees to adjoint $\gamma$-factors as in \cite{HII}. 

The computation of the volume of the Iwahori subgroup of $G$ in \cite[(4.11)]{Gro} gives:
\begin{equation}\label{eq:2.9}
\mr{vol}(P_{\mf f}) = \big| \overline{\mc G_{\mf f}^\circ} (k_F) \big| \, q_F^{-\big( 
\mb a (\mf g^\vee) + \dim \overline{\mc G_{\mf f}^\circ} + \dim (G^\vee)^{\mb I_F} \big) / 2} .
\end{equation}
By \cite[\S 5.1]{DeRe} this formula actually holds for every facet $\mf f$ and every 
connected reductive $F$-group.

For a ramified simple group, we will see in Lemma \ref{lem:4.1} that
\begin{equation}\label{eq:2.12}
\dim (G^\vee)^{\mb I_F} = \dim (G'^\vee)^{\mb I_F} = \dim G'^\vee = \dim \mc G' . 
\end{equation}
With \eqref{eq:2.9}, \eqref{eq:2.12} and Proposition \ref{prop:2.1} we can compare the 
Haar measures on $G$ and $G'$:
\begin{equation}\label{eq:2.10}
\mr{vol}(P_{\mf f'}) = \big| \overline{\mc G_{\mf f'}^{'\circ}} (k_F) \big| \,
q_F^{-\big( \dim \overline{\mc G_{\mf f'}^{'\circ}} + \dim (G'^\vee)^{\mb I_F} \big) / 2} =  
q_F^{\mb a (\mf g^\vee) / 2} \mr{vol}(P_{\mf f}) .
\end{equation}
By \eqref{eq:2.1} the formal degree of \eqref{eq:2.6} is
\begin{equation}\label{eq:2.8}
\mr{fdeg}\big( \mr{ind}_{N_G (P_{\mf f})}^G \sigma^N \big) = 
\frac{\dim (\sigma^N)}{\mr{vol}(N_G (P_{\mf f}))} = 
\frac{\dim (\sigma) q_F^{\big( \mb a (\mf g^\vee) + \dim \overline{\mc G_{\mf f}^\circ} + 
\dim (G^\vee)^{\mb I_F} \big) / 2}}{|\Omega_{G,\mf f}| \, | \overline{\mc G_{\mf f}^\circ} (k_F)|} . 
\end{equation}

\begin{cor}\label{cor:2.4}
Every diagram automorphism of $\mc G (F)$ or $\mc G' (F)$ stabilizes every irreducible
supercuspidal unipotent representation of that group.

The bijection from Theorem \ref{thm:2.3}.b induces a bijection
\[
\begin{array}{ccc}
\Irr_{\unip,\cusp}(G) & \longleftrightarrow & \Irr_{\unip,\cusp}(G') \\
\pi & \leftrightarrow & \pi'
\end{array} 
\]
which relates formal degrees as
\[
\mr{fdeg}(\pi') = q_F^{- \mb a (\mf g^\vee) / 2} \mr{fdeg}(\pi)
\]
This bijection is canonical up to choosing extensions of cuspidal unipotent representations
of $P_{\mf f}$ to $N_G (P_{\mf f})$ (or equivalently: from 
$\overline{\mc G_{\mf f}^\circ}(k_F)$ to $\overline{\mc G_{\mf f}}(k_F)$).
\end{cor}
\begin{proof}
The first claim follows from Lemma \ref{lem:2.5} and the discussion preceding it.
The bijectivity is a consequence of Theorem \ref{thm:2.3} and the bijection between 
$\Delta_{\nr,\af}$ and $\Delta'_{\nr,\af}$. The indicated canonicity comes from 
Proposition \ref{prop:2.2}. The relation between the formal degrees follows \eqref{eq:2.8}, 
\eqref{eq:2.10} and the dimension preservation in Theorem \ref{thm:2.3}.b.
\end{proof}

\section{Matching of Hecke algebras}
\label{sec:Hecke}

To analyse the non-supercuspidal unipotent $G$-representations, we need types and 
Hecke algebras, following \cite{BuKu}. This was worked out for general depth zero
representations in \cite{Mor1,Mor3}, and for representations of $F_\nr$-split
simple groups of adjoint type in \cite[\S 1]{LusUni1}. Fortunately the arguments 
from \cite[\S 1]{LusUni1} also apply to ramified simple groups, see \cite[\S 3]{SolLLC}.
We recall the main points, in the notation from \cite[\S 3]{SolLLC}. 

Let $\hat P_{\mf f}$ be the pointwise stabilizer of $\mf f$ in $G$, so
$P_{\mf f}^+ \subset P_{\mf f} \subset \hat P_{\mf f} \subset N_G (P_{\mf f})$. Then
\[
\hat P_{\mf f} / P_{\mf f} \cong \Omega_{G,\mf f,\tor} ,
\]
where the right hand side denotes the pointwise stabilizer of $\mf f$ in 
$\Omega_G$ (or equivalently the pointwise stabilizer in $\Omega_G$ of all vertices
of $\mf f$). As $\ker (\Omega_G \to \Omega_{G_\ad})$ acts trivially on the relative
local Dynkin diagram of $\mc G$, it is contained in $\Omega_{G,\mf f,\tor}$.

Let $\hat \sigma$ be an extension of a cuspidal unipotent representation
$\sigma$ of $P_{\mf f}$ to $\hat P_{\mf f}$. Then $(\hat P_{\mf f}, \hat \sigma)$
is a type for a single Bernstein block of $G$, say $\Rep (G)_{(\hat P_{\mf f}, \hat \sigma)}$.
We denote the associated Hecke algebra by 
\begin{equation}\label{eq:3.9}
\mc H (G,\hat P_{\mf f}, \hat \sigma) = 
\mr{End}_G \big( \mr{ind}_{P_{\mf f}}^G \hat \sigma \big).
\end{equation}
There is an equivalence of categories
\begin{equation}\label{eq:3.1}
\begin{array}{ccc}
\Rep (G)_{(\hat P_{\mf f}, \hat \sigma)} & \longrightarrow & 
\Mod ( \mc H (G,\hat P_{\mf f}, \hat \sigma) ) \\
\pi & \mapsto & \Hom_{\hat P_{\mf f}} (\hat \sigma, \pi) 
\end{array}.
\end{equation}
Let $J_{\mf f} \subset \Delta_\af$ be the set of simple affine roots that vanish on $\mf f$.
If $|J_{\mf f}| = |\Delta_\af| - 1$, then $\mr{ind}_{\hat P_{\mf f}}^G (\hat \sigma)$ is
irreducible, supercuspidal and $\mc H (G,\hat{P_{\mf f}}, \hat \sigma) \cong \C$.

Henceforth we assume that $|J_{\mf f}| < |\Delta_\af| - 1$, so that $\mf f$ is not a vertex of
$\mc B (\mc G,F)$ and $P_{\mf f}$ is not a maximal parahoric subgroup of $G$. 
The set $\Delta_{\mf f,\af} := \Delta_\af \setminus J_{\mf f}$ indexes a set of generators
$S_{\mf f,\af}$ for an affine Weyl group $W_\af (J_{\mf f},\sigma)$ contained in
$N_G (S) / (N_G (S) \cap P_{C_0})$. Let $\ell$ be the length function of the
Coxeter system $(W_\af (J_{\mf f},\sigma), S_{\mf f,\af})$. Together with a parameter function
$q^{\mc N} : \Delta_{\mf f,\af} \to \R_{>0}$ this gives rise to an Iwahori--Hecke algebra 
$\mc H (W_\af (J_{\mf f},\sigma),q^{\mc N})$. As a $\C$-vector space it has a basis
$\{ N_w : w \in W_\af (J_{\mf f},\sigma) \}$, every generator $N_s$ (with $s \in S_{\mf f,\af}$)
satisfies a quadratic relation
\begin{equation}\label{eq:3.2}
\big( N_s - q^{\mc N (s)/2} \big) \big( N_s + q^{\mc N (s)/2} \big) = 0 
\end{equation}
and there are braid relations
\begin{equation}\label{eq:3.5}
N_w N_v = N_{wv} \quad \text{whenever} \quad \ell (w) + \ell (v) = \ell (wv). 
\end{equation}
Moreover the relations \eqref{eq:3.2} and \eqref{eq:3.5} provide a presentation of 
$\mc H (W_\af (J_{\mf f},\sigma),q^{\mc N})$.

The group $\Omega_{G,\mf f} / \Omega_{G,\mf f,\tor}$ (which in our setting has order one
or two) acts naturally on $(W_\af (J_{\mf f},\sigma), S_{\mf f,\af})$ and on 
$\mc H (W_\af (J_{\mf f},\sigma),q^{\mc N})$. With these notations there is an algebra isomorphism
\begin{equation}\label{eq:3.3}
\mc H (G, \hat P_{\mf f}, \hat \sigma) \cong
\mc H (W_\af (J_{\mf f},\sigma),q^{\mc N}) \rtimes \Omega_{G,\mf f} / \Omega_{G,\mf f,\tor} .
\end{equation}
When $\Omega_{G,\mf f} / \Omega_{G,\mf f,\tor}$ is represented by 
$\{1,\omega\} \subset N_G (P_{\mf f})$, the basis element $N_\omega$ of \eqref{eq:3.3} 
acts on $\mr{ind}_{\hat P_{\mf f}}^G \hat \sigma$ by
\begin{equation}\label{eq:3.7}
(N_\omega f) (g) = \sigma^N (\omega) f (g \omega) \qquad 
f \in \mr{ind}_{\hat P_{\mf f}}^G \hat \sigma ,
\end{equation}
where $\sigma^N \in \Irr (N_G (P_{\mf f}))$ is an extension of $\hat \sigma$.
Hence \eqref{eq:3.3} is canonical up choosing such an extension, 
or equivalently up to a character of $N_G (P_{\mf f}) / \hat P_{\mf f}$.

As Lusztig noted in \cite[\S 10.13]{LusUni2}, all these constructions depend only on the local
index of $G$ and on the action of $\Omega_G$ on the relative local Dynkin diagram.\\

Let $\mc G'$ be the $F_\nr$-split companion group of $\mc G$, as in Section \ref{sec:list}.
Applying the proof of Theorem \ref{thm:2.3}.b to Theorem \ref{thm:2.3}.a with 
$\Omega_{G,\mf f,\tor}$ instead of $\Omega_{G,\mf f}$, we obtain a bijection
\begin{equation}\label{eq:3.4}
\Irr_\unip (\hat P_{\mf f}) \longleftrightarrow \Irr_\unip (\hat P_{\mf f'}) , 
\end{equation}
which preserves dimensions and cuspidality. In fact, as $|\Omega_{G,\mf f,\tor}| \leq 2$,
we can take for \eqref{eq:3.4} just an instance of Theorem \ref{thm:2.3}.a if 
$\Omega_{G,\mf f,\tor} = 1$ and an instance of Theorem \ref{thm:2.3}.b otherwise.

\begin{thm}\label{thm:3.1}
Let $\hat \sigma \in \Irr_{\unip,\cusp}(\hat P_{\mf f})$ and let $\hat \sigma' \in
\Irr_{\unip,\cusp} (\hat P_{\mf f'})$ be its image under \eqref{eq:3.4}. 
\enuma{
\item The bijection $\Delta'_{\nr,\af} \longleftrightarrow \Delta_{\nr,\af}$ 
induces an isomorphism of Coxeter systems
\[
\begin{array}{ccc}
(W_\af (J_{\mf f'},\sigma'), S_{\mf f',\af}) & \longrightarrow & 
(W_\af (J_{\mf f},\sigma), S_{\mf f,\af}) \\
(w',s') & \mapsto & (w,s)  
\end{array} . 
\]
\item The linear map
\[
\begin{array}{ccc}
\mc H (W_\af (J_{\mf f'},\sigma'), q^{\mc N'}) & \longrightarrow & 
\mc H (W_\af (J_{\mf f},\sigma), q^{\mc N}) \\
N_{w'} & \mapsto & N_w 
\end{array}
\]
is an algebra isomorphism. 
\item Part (b) and Theorem \ref{thm:2.3}.b induce an algebra isomorphism
\[
\mc H (G',\hat P_{\mf f'},\hat \sigma') \longrightarrow \mc H (G,\hat P_{\mf f},\hat \sigma) .
\]
}
\end{thm}
\begin{proof}
(a) Recall from Section \ref{sec:list} that the local indices of $G$ and $G'$ are isomorphic
up to changing some arrows. From \cite[\S 1.15 and \S 2.28--2.30]{LusUni1} we see that
$(W_\af (J_{\mf f},\sigma), S_{\mf f,\af})$ depends only on $\mf f$ and on the local index
of $G$, and that it does not change if we reverse some arrows in $\Delta_{\nr,\af}$. This
gives the isomorphism of Coxeter systems.\\
(b) Similarly, \cite[\S 1.18]{LusUni1} and \cite[Table II]{Lus-Che} show that 
$q^{\mc N} : S_{\mf f,\af} \to \R_{>0}$ depends on $\mf f$ and the local index of $G$
(modulo changing the direction of arrows). From the relations \eqref{eq:3.2} and \eqref{eq:3.5} 
we see that part (a) extends linearly to an isomorphism of Iwahori--Hecke algebras.\\
(c) Choose an extension $\sigma^N$ of $\hat \sigma$ to $N_G (P_{\mf f})$ and use it to get
\eqref{eq:3.3}. Analogously, we use the image of $\sigma^N$ under Theorem \ref{thm:2.3}.b
to construct \eqref{eq:3.3} for $G'$. Then the group isomorphism 
\begin{equation}\label{eq:3.6}
\Omega_{G,\mf f} / \Omega_{G,\mf f,\tor} \cong \Omega_{G',\mf f'} / \Omega_{G',\mf f',\tor} 
\end{equation}
extends the isomorphism from part (b) to the indicated affine Hecke algebras.

We still need to check that this isomorphism does not depend on the choice of $\sigma^N$.
The only other possible extension of $\hat \sigma$ is $\sigma^N \otimes \chi_-$, where
$\chi_-$ denotes the unique nontrivial character of $N_G (P_{\mf f}) / \hat P_{\mf f}$.
Notice that the latter group is naturally isomorphic with \eqref{eq:3.6} and with
$N_{G'} (P_{\mf f'}) / \hat P_{\mf f'}$. Then the image in $\Irr (N_{G'}(P_{\mf f'}))$ is also
adjusted by tensoring with $\chi_-$, and \eqref{eq:3.7} shows that $\sigma^N \otimes \chi_-$ 
leads to the same isomorphism of affine Hecke algebras as $\sigma^N$.
\end{proof}

\begin{cor}\label{cor:3.2}
There are equivalences of categories
\[
\begin{array}{ccccccc}
\!\! \Rep (G)_{(\hat P_{\mf f}, \hat \sigma)} & \rightarrow & 
\Mod ( \mc H (G,\hat P_{\mf f}, \hat \sigma) ) & \rightarrow & 
\Mod ( \mc H (G',\hat P_{\mf f'}, \hat \sigma') ) & \leftarrow &
\!\! \Rep (G')_{(\hat P_{\mf f'}, \hat \sigma')} \\
\pi & \mapsto & \Hom_{\hat P_{\mf f}} (\hat \sigma, \pi) & \mapsto &
\Hom_{\hat P_{\mf f'}} (\hat \sigma', \pi') & \text{\reflectbox{$\mapsto$}} & \!\! \pi' 
\end{array} \!\!.
\]
With \eqref{eq:5} these combine to an equivalence between the categories 
$\Rep (G)_\unip$ and $\Rep (G')_\unip$.
\end{cor}
\begin{proof}
The equivalences with the Bernstein block $\Rep (G)_{(\hat P_{\mf f}, \hat \sigma)}$ are a 
consequence of Theorem \ref{thm:3.1} and \eqref{eq:3.1}. By Theorem \ref{thm:2.3} and the bijection
$\Delta_{\nr,\af} \leftrightarrow \Delta'_{\nr,\af}$ the indexing set in \eqref{eq:5} is
in bijection with $\{ (\hat P_{\mf f'}, \hat \sigma') \} / G'$-conjugation. Hence the above
equivalence of categories for one Bernstein block combine, in the same way for $G$ and $G'$, 
to all unipotent Bernstein blocks.
\end{proof}

We aim to show that Corollary \ref{cor:3.2} preserves many relevant properties. Let 
$\mc L_{\mf f}$ be the standard $F$-Levi subgroup of $\mc G$ such that $\Phi (\mc L_{\mf f},
\mc S)$ consists precisely of the roots in $\Phi (\mc G,\mc S)$ that are constant on $\mf f$.
Then the cuspidal supports of elements of $\Irr (G)_{(\hat P_{\mf f}, \hat \sigma)}$ are
contained in a Bernstein component of $\Rep (L_{\mf f})$. We define $\mc L'_{\mf f'} \subset
\mc G'$ in the same way.

By definition, a character of $G$ is weakly unramified if it is trivial on every parahoric 
subgroup of $G$. Via the Kottwitz map $G \to \Omega_G$, these characters can be identified
with the characters of $\Omega_G$. In particular Section \ref{sec:list} provides a 
canonical bijection between the weakly unramified characters of $G$ and of $G'$.

\begin{lem}\label{lem:3.7}
The equivalence between the categories $\Rep (G)_\unip$ and $\Rep (G')_\unip$ from
Corollary \ref{cor:3.2} is compatible with twisting by weakly unramified characters.
\end{lem}
\begin{proof}
When $\Omega_G = 1$, also $\Omega_{G'} = 1$, all weakly unramified characters are trivial
and there is nothing to prove. 
Otherwise $|\Omega_G| = |\Omega_{G'}| = 2$. Then we identify the nontrivial weakly unramified 
character of $G$ with that of $G'$ and we call it $\chi$. There are three cases to consider:
\begin{itemize}
\item When $|\Omega_{G,\mf f} / \Omega_{G,\mf f,\tor}| = 2$, tensoring by $\chi$ stabilizes
the four categories in the first part of Corollary \ref{cor:3.2}. It is clear from \eqref{eq:3.7}
that its effect is compatible with the equivalences between these four categories.
\item When $|\Omega_{G,\mf f}| = |\Omega_{G,\mf f,\tor}| = 2$, tensoring by $\chi$ identifies
$\mc H (G,\hat P_{\mf f}, \hat \sigma)$ with $\mc H (G,\hat P_{\mf f}, \chi \otimes \hat \sigma)$
and $\mc H (G',\hat P_{\mf f'}, \hat \sigma')$ with $\mc H (G',\hat P_{\mf f'}, 
\chi \otimes \hat \sigma')$. If $\pi$ is mapped to $\pi'$, then by the complete analogy on both
sides ($G$ and $G'$), $\chi \otimes \pi$ is mapped to $\chi \otimes \pi'$.
\item When $|\Omega_{G,\mf f}| = |\Omega_{G,\mf f,\tor}| = 1$, the representations
$\mr{ind}_{\hat P_{\mf f}}^G (\hat \sigma)$ and $\mr{ind}_{\hat P_{\mf f'}}^{G'} (\hat \sigma')$
are unaffected by tensoring with $\chi$. These are progenerators of the categories
$\Rep (G)_{(\hat P_{\mf f}, \hat \sigma)}$ and $\Rep (G')_{(\hat P_{\mf f'}, \hat \sigma')}$,
so all elements of those categories are stable under tensoring by $\chi$. \qedhere
\end{itemize}

\end{proof}

Recall that any $\pi \in \Rep (G)$ is called essentially square-integrable if
its restriction to the derived group $G_\der$ is square-integrable. In particular this
forces $\pi$ to be admissible. For the definitions of various kinds of representations of
affine Hecke algebras we refer to \cite{SolComp}.

\begin{lem}\label{lem:3.3}
\enuma{
\item The equivalences of categories in Corollary \ref{cor:3.2} preserve temperedness
of representations.
\item The equivalence between the categories $\Rep (G)_{(\hat P_{\mf f}, \hat \sigma)}$ and
$\Rep (G')_{(\hat P_{\mf f'}, \hat \sigma')}$ preserves essential square-integrability.
}
\end{lem}
\begin{proof}
(a) The isomorphism from Theorem \ref{thm:3.1}.c comes from isomorphisms between all the data 
used to construct these affine Hecke algebras, so it extends to an isomorphism between their 
respective Schwartz completions. By definition \cite[\S 1]{SolComp}, this means that the middle 
map in Corollary \ref{cor:3.2} preserves temperedness.

For the two outer maps in Corollary \ref{cor:3.2} the statement is a consequence of 
\cite[Theorem 3.12 and Corollary 4.4]{SolComp}.\\
(b) Suppose that $\Rep (G)_{(\hat P_{\mf f}, \hat \sigma)}$ contains an essentially 
square-integrable representation $\pi$, necessarily of finite length. Then \cite[Proposition 
3.10.a and Corollary 4.4]{SolComp} tell us that the root systems for $(G,L_{\mf f})$ 
and for $\mc H (G,\hat P_{\mf f}, \hat \sigma)$ have the same rank. By isomorphism, the root
system underlying $\mc H (G',\hat P_{\mf f'}, \hat \sigma')$ also has that rank. The rank
of the root system of $(G,L_{\mf f})$ is simply $|\Delta_\af \setminus J_{\mf f}| - 1 =
|\Delta'_\af \setminus J_{\mf f'}| - 1$, so equal to the rank of the root system of
$(G',L'_{\mf f'})$. 

This argument works just as well from the other side: if
$\Rep (G')_{(\hat P_{\mf f'}, \hat \sigma')}$ contains an essentially square-integrable
representation, then the root systems underlying the four terms in Corollary \ref{cor:3.2}
all have the same rank. Knowing that, \cite[Theorem 3.9 and Corollary 4.4]{SolComp}
prove the statement.
\end{proof}

Let us investigate the effect of Corollary \ref{cor:3.2} on formal degrees of square-integrable
unipotent $G$-representations. Recall that we normalized the Haar measure on $G$ in 
\eqref{eq:2.9}. We endow the affine Hecke algebra $\mc H (G,\hat P_{\mf f}, \hat \sigma)$ with
the unique trace such that
\[
\mr{tr}(N_w) = \left\{ \begin{array}{ll}
N_e (1) = \dim (\hat \sigma) \mr{vol}(\hat P_{\mf f})^{-1} & w = e \\
0 & w \neq e
\end{array} \right. .
\]
The Plancherel decomposition of this trace \cite{Opd-Sp} determines a density
on the set of irreducible tempered $\mc H (G,\hat P_{\mf f}, \hat \sigma)$-representations,
and in particular provides a normalization of formal degrees.
Similarly we normalize the trace on $\mc H (G',\hat P_{\mf f'}, \hat \sigma')$ by
tr$(N_e) = \dim (\hat{\sigma}') \mr{vol}(\hat P_{\mf f'})^{-1}$, and we use the 
Plancherel density derived from that.

\begin{lem}\label{lem:3.5}
The equivalences of categories in Corollary \ref{cor:3.2} relate formal degrees 
of square-integrable representations as
\[
\mr{fdeg}(\pi') = q_F^{- \mb a (\mf g^\vee) / 2} \mr{fdeg}(\pi) .
\]
It multiplies Plancherel densities of irreducible tempered representations by the same
factor $q_F^{- \mb a (\mf g^\vee) / 2}$.
\end{lem}
\begin{proof}
By \cite{BHK} the two outer maps in Corollary \ref{cor:3.2}, with the 
indicated normalizations, preserve formal degrees.
By Theorem \ref{thm:2.3}.b and \eqref{eq:2.10} the Hecke algebra isomorphism from 
Theorem \ref{thm:3.1}.c multiplies the traces by a factor $q_F^{- \mb a (\mf g^\vee) / 2}$. 
Hence it adjusts formal degrees by the same factor.

The same argument applies to Plancherel densities.
\end{proof}

Finally we consider the diagram automorphism $\tau$ of $\mc G (F)$ from page
\pageref{lem:2.5}. By Lemma \ref{lem:2.5} it stabilizes $\hat P_{\mf f}$ and $\hat \sigma$,
so it acts canonically on $\mc H (G,\hat P_{\mf f},\hat \sigma)$ by an algebra
automorphism.

\begin{lem}\label{lem:3.6}
For every type $(\hat P_{\mf f},\hat \sigma)$ as above, the action of $\tau$ on
$\mc H (G,\hat P_{\mf f},\hat \sigma)$ is the identity. Hence $\tau$ stabilizes all
unipotent representations of $G$.
\end{lem}
\begin{proof}
By Lemma \ref{lem:2.5} $\tau$ fixes the identity element $N_e$ of 
$\mc H (G,\hat P_{\mf f},\hat \sigma)$, and it fixes $N_G (P_{\mf f}) / P_{\mf f} \cong
\Omega_{G,\mf f}$ pointwise. Further, we observed in the proof of Lemma \ref{lem:2.5}
that $\tau$ acts trivially on the local index of $G$. Together with $\mf f$, these objects
determine $W_\af (J_{\mf f},\sigma) \rtimes \Omega_{G,\mf f}$. Hence $\tau$ stabilizes 
the double coset $\hat P_{\mf f} w \hat P_{\mf f}$, for every 
$w \in W_\af (J_{\mf f},\sigma)\rtimes \Omega_{G,\mf f} / \Omega_{G,\mf f,\tor}$.

In particular $\tau (N_s) \in \C N_s$ for every $s \in S_{\mf f,\af}$. The quadratic
relation \eqref{eq:3.2} and $\tau (N_e) = N_e$ force $\tau (N_s) = N_s$. Hence $\tau$
is the identity on $\mc H (W_\af (J_{\mf f},\sigma), q^{\mc N})$.

Suppose that $N_G (P_{\mf f}) / P_{\mf f} = \{1, \omega\}$ and $\sigma \in 
\Irr_\unip (P_{\mf f})$. Then $\sigma$ can be extended in two ways to $N_G (P_{\mf f})$,
differing by a character of $N_G (P_{\mf f}) / P_{\mf f}$. In the proof of Lemma \ref{lem:2.5}
we saw that $\tau$ stabilizes the two extensions $\sigma_+,\sigma_-$. On the other hand,
if $\tau (N_\omega) = - N_\omega$, then $\tau$ would exchange $\sigma_+$ and $\sigma_-$.
As $N_\omega^2 = N_e$, we conclude that $\tau (N_\omega) = N_\omega$. Now we see from 
\eqref{eq:3.3} that $\tau$ fixes $\mc H (G,\hat P_{\mf f},\hat \sigma)$ entirely.

Then \eqref{eq:3.1} implies that $\tau$ stabilizes all elements of $\Rep (G)_{(\hat P_{\mf f}, 
\hat \sigma)}$. This holds for all types $(\hat P_{\mf f}, \hat \sigma)$, so by
\eqref{eq:5} for the whole of $\Rep (G)_\unip$.
\end{proof}

Let $\mc P$ be a standard parabolic $F$-subgroup of $\mc G$ and let $\mc P'$ be the 
associated standard parabolic $F$-subgroup of $\mc G'$, as explained in Section 
\ref{sec:list}. Let $\mc M$ and $\mc M'$ be their respective standard Levi factors. 
Then $\mc M$ and $\mc M'$ stand in the same relation to each other as $\mc G$ and $\mc G'$,
except that they need not be simple or adjoint.

When $\mc M$ contains $\mc L_{\mf f}$, \cite[Theorem 2.1]{Mor3} says that
\[
(\hat P_{\mf f} \cap L_{\mf f}) / (P_{\mf f} \cap L_{\mf f}) \cong
(\hat P_{\mf f} \cap M) / (P_{\mf f} \cap M) \cong \hat P_{\mf f} / P_{\mf f} .
\]
Then $\hat \sigma$ can also be considered as a representation of 
$\hat P_{\mf f} \cap L_{\mf f}$ or of $\hat P_{M,\mf f} := \hat P_{\mf f} \cap M$. 
Thus Theorem \ref{thm:2.3} induces a bijection
\begin{equation}\label{eq:3.10}
\Irr (\hat P_{M,\mf f})_{\cusp,\unip} \to \Irr (\hat P_{M',\mf f'})_{\cusp,\unip} :
\hat \sigma \mapsto \hat \sigma' .
\end{equation}
We construct an equivalence of categories
\[
\Rep (M)_{(\hat P_{M,\mf f}, \hat \sigma)} \longleftrightarrow
\Rep (M')_{(\hat P_{M',\mf f'}, \hat \sigma')} 
\]
as in Corollary \ref{cor:3.2}. Let 
\[
\mr{pr}_{(\hat P_{M,\mf f}, \hat \sigma)} : \Rep (M) \to \Rep (M)_{(\hat P_{M,\mf f}, \hat \sigma)} 
\]
be the natural projection coming from the Bernstein decomposition. We denote the normalized
parabolic induction functor and the normalized Jacquet restriction functor associated to $P$ by
\[
I_P^G : \Rep (M) \to \Rep (G) \quad \text{and} \quad J_P^G : \Rep (G) \to \Rep (M).
\]
Let $\overline{\mc P}$ be the parabolic $F$-subgroup of $\mc G$ which is opposite to $\mc P$
with respect to $\mc M$.

\begin{lem}\label{lem:3.4}
The following diagrams commute:
\[
\begin{array}{ccc@{\hspace{6mm}}ccc}
\Rep (G)_{(\hat P_{\mf f}, \hat \sigma)} & \!\!\! \longleftrightarrow \!\!\! &
\Rep (G')_{(\hat P_{\mf f'}, \hat \sigma')} & \Rep (G)_{(\hat P_{\mf f}, \hat \sigma)} & 
\!\!\! \longleftrightarrow \!\!\! & \Rep (G')_{(\hat P_{\mf f'}, \hat \sigma')} \\
\uparrow I_P^G & & \uparrow I_{P'}^{G'} & 
\downarrow \mr{pr}_{(\hat P_{M,\mf f}, \hat \sigma)} \circ J^G_{\overline P} & & 
\downarrow \mr{pr}_{(\hat P_{M',\mf f'}, \hat \sigma')} \circ J^{G'}_{\overline{P'}} \\
\!\! \Rep (M)_{(\hat P_{M,\mf f}, \hat \sigma)} & \!\!\! \longleftrightarrow \!\!\! &
\Rep (M')_{(\hat P_{M',\mf f'}, \hat \sigma')} & \Rep (M)_{(\hat P_{M,\mf f}, \hat \sigma)} & 
\!\!\! \longleftrightarrow \!\!\! & \Rep (M')_{(\hat P_{M',\mf f'}, \hat \sigma')} 
\end{array}
\]
\end{lem}
\begin{proof}
By \cite[Corollary 3.10]{Mor3} and \cite[Proposition 8.5]{BuKu} the type  
$(\hat P_{\mf f}, \hat \sigma)$ is a cover of $(\hat P_{M,\mf f}, \hat \sigma)$. In this setting
\cite[Corollary 7.12]{BuKu} gives a canonical algebra monomorphism
\[
t_{\overline P} : \mc H (M,\hat P_{M,\mf f},\hat \sigma) \to \mc H (G,\hat P_{\mf f}, \hat \sigma) ,
\]
which implements unnormalized Jacquet restriction with respect to $\overline P$ 
\cite[Corollary 8.4]{BuKu}. We adjust it by the square root of a modular character as in 
the proof of \cite[Lemma 4.1]{SolComp}, and call the result $\lambda_{MG}$. In terms of 
the presentation from \eqref{eq:3.2} and \eqref{eq:3.5}, this works out as
\[
\lambda_{MG} (N_w) = N_w \qquad \text{for all } w \in W_\af (J_{\mf f},\sigma) \rtimes
\Omega_{G,\mf f} / \Omega_{G,\mf f,\tor} \text{ with } \hat P_{\mf f} w \hat P_{\mf f} \cap M
\neq \emptyset .
\]
Via $\lambda_{MG}$ we regard $\mc H (M,\hat P_{M,\mf f},\hat \sigma)$ as a subalgebra of 
$\mc H (G,\hat P_{\mf f}, \hat \sigma)$. Then \cite[Condition 3.1 and Lemma 4.1]{SolComp}
say that restriction of representations from $\mc H (G,\hat P_{\mf f}, \hat \sigma)$ to
$\mc H (M,\hat P_{M,\mf f}, \hat \sigma)$ fits in a commutative diagram 
\begin{equation}\label{eq:3.8}
\begin{array}{ccc}
\Rep (G)_{(\hat P_{\mf f}, \hat \sigma)} & 
\longleftrightarrow & \Mod ( \mc H (G,\hat P_{\mf f}, \hat \sigma) ) \\ 
\downarrow \mr{pr}_{(\hat P_{M,\mf f}, \hat \sigma)} \circ J^G_{\overline P} & & 
\downarrow \mr{Res} \\
\Rep (M)_{(\hat P_{M,\mf f}, \hat \sigma)} & 
\longleftrightarrow & \Mod ( \mc H (M,\hat P_{M,\mf f}, \hat \sigma) ) 
\end{array}.
\end{equation}
The same holds if the vertical arrows are replaced by $I_P^G$ and by induction from 
$\mc H (M,\hat P_{M,\mf f}, \hat \sigma)$ to $\mc H (G,\hat P_{\mf f}, \hat \sigma)$.

Of course that applies equally well to $G'$ and $M'$. Clearly the Hecke algebra isomorphism 
from Theorem \ref{thm:3.1}.c transfers $\lambda_{MG}$ to $\lambda_{M' G'}$. Hence the diagram
\[
\begin{array}{ccc}
\Mod ( \mc H (G,\hat P_{\mf f}, \hat \sigma) ) & 
\longleftrightarrow & \Mod ( \mc H (G',\hat P_{\mf f'}, \hat \sigma') ) \\ 
\downarrow \mr{Res} & & 
\downarrow \mr{Res} \\
\Mod ( \mc H (M,\hat P_{M,\mf f}, \hat \sigma) ) & 
\longleftrightarrow & \Mod ( \mc H (M',\hat P_{M',\mf f'}, \hat \sigma') ) 
\end{array}
\]
commutes, and similarly with the vertical arrows replaced by induction functors.
\end{proof}

In particular Lemma \ref{lem:3.4} shows that Corollary \ref{cor:3.2} respects supercuspidality
-- which we knew already from Corollary \ref{cor:2.4}.

\section{Comparison of Langlands parameters}
\label{sec:Lpar}

For the moment, $\mc G$ denotes any connected reductive $F$-group, and $G = \mc G (F)$.
Let $G^\vee$ be the complex dual group of $\mc G$ and let ${}^L G = G^\vee \rtimes \mb W_F$
be a Langlands dual group. Recall \cite{Bor} that a Langlands parameter for $G$ is a
homomorphism $\phi : \mb W_F \times SL_2 (\C) \to {}^L G$ satisfying certain conditions.
We denote the set of $G^\vee$-equivalence classes of Langlands parameters for $G$ by $\Phi (G)$.
We call $\phi$:
\begin{itemize}
\item bounded if $\phi (\Fr) = (s,\Fr)$ with $s$ in a compact subgroup of $G^\vee$;
\item discrete if the image of $\phi$ is not contained in ${}^L M$ for any proper $F$-Levi
subgroup $\mc M$ of $\mc G$;
\item unramified if $\phi (i) = (1,i)$ for all $i \in \mb I_F$.
\end{itemize}
We denote the corresponding subsets of $\Phi (G)$ by, respectively, $\Phi_{\mr{bdd}}(G),
\Phi^2 (G)$ and $\Phi_\nr (G)$. We note that an unramified L-parameter is determined up to
$G^\vee$-conjugacy by the semisimple element $\phi (\Fr) \in {}^L G$ and the unipotent 
element
\[
u_\phi := \phi \big( 1, \matje{1}{1}{0}{1} \big) \in G^\vee . 
\]
Let ${G^\vee}_\Sc$ be the simply connected cover of the derived group ${G^\vee}_\der$ of
$G^\vee$. The image of $Z_{G^\vee}(\phi)$ in ${G^\vee}_\der$ is 
$Z_{G^\vee}(\phi) Z(G^\vee) / Z(G^\vee)$. Let $Z^1_{{G^\vee}_\Sc}(\phi)$ be the preimage
of that in ${G^\vee}_\Sc$ and define
\[
\mc S_\phi := \pi_0 \big( Z^1_{{G^\vee}_\Sc}(\phi) \big) .
\]
This is the component group of $\phi$ used in \cite{Art,AMS1}. An enhancement of $\phi$
is an irreducible representation $\rho$ of $\mc S_\phi$. The group $G^\vee$ acts naturally
on the set of enhanced L-parameters by
\[
g \cdot (\phi,\rho) = (g \phi g^{-1}, g \cdot \rho) \qquad
(g \cdot \rho)(h) = \rho (g^{-1} h g) .
\]
Via the canonical map $Z({G^\vee}_\Sc) \to Z(\mc S_\phi)$, every enhancement $\rho$
determines a character $\chi_\rho$ of $Z({G^\vee}_\Sc)$. On the other hand, $\mc G$ is an
inner twist of a unique quasi-split $F$-group $\mc G^*$. The parametrization of equivalence
classes of inner twists of $\mc G^*$ by 
\[
H^1 (\mb W_F, \mc G^*_\ad) \cong \Irr \big( Z ({G^\vee}_\Sc)^{\mb W_F} \big)
\]
provides a character $\chi_{\mc G}$ of $Z ({G^\vee}_\Sc)^{\mb W_F}$. We choose an extension 
to a character $\chi_{\mc G}^e$ of $Z({G^\vee}_\Sc)$. (Such an extension is related to an
explicit construction of $\mc G$ as inner twist of $\mc G^*$, compare with Section 
\ref{sec:rigid}.) Then we say that
$\rho$ or $(\phi,\rho)$ is $G$-relevant if $\chi_\rho = \chi_{\mc G}^e$. We denote the
collection of $G^\vee$-orbits of $G$-relevant enhanced L-parameters by $\Phi_e (G)$.

When $\mc G$ is ramified $F$-simple and $\mc G' = \mc G'^\circ \times \{\pm 1\}$, we have
$Z(G'^{\vee,\circ}) = 1$ and $Z(G'^\vee) = \{\pm 1\}$. Then we define 
$\chi_{\mc G'} \in \Irr (Z(G'^\vee))$ to
be trivial if $\mc G$ is quasi-split over $F$ and nontrivial otherwise.

\begin{lem}\label{lem:4.1}
Let $\mc G$ be a ramified simple $F$-group and $\mc G'$ be its companion group
from Section \ref{sec:list}. There exists a $\mb W_F$-equivariant isomorphism
\[
\lambda^\vee_{G' G} : G'^\vee \to (G^\vee)^{\mb I_F}, 
\]
which is unique up to inner automorphisms.
\end{lem}
\begin{proof}
This boils down to one quick check for every entry in the list in Section \ref{sec:list}. In all
cases $\mb W_F$ acts trivially on $(G^\vee)^{\mb I_F}$ (and on $G'^\vee$) because $\mc G$
is an inner twist of the quasi-split $F$-group given by $\mc G (F_\nr)$.
\begin{itemize}
\item $\mathbf{B\mi C_n}$. Let $A^{-T}$ be the inverse transpose of an invertible matrix
and let $J$ be an anti-diagonal square matrix whose nonzero entries are alternatingly 
1 and -1. Then $\mb I_F$ acts on $G^\vee = SL_{2n}(\C)$ via $A \mapsto J A^{-T} J^{-1}$ and
\[
(G^\vee)^{\mb I_F} = SL_{2n}(\C)^{\mb I_F} = Sp_{2n}(\C) = G'^\vee . 
\]
The same $\mb I_F$-action is well-defined on any group $G^\vee$ isogenous to $SL_{2n}(\C)$. 
Such a group is determined by the order $d^\vee$ of $Z(G^\vee)$. Then $2n / d^\vee$ is the
order of the schematic centre of $\mc G$. We find
\[
(G^\vee)^{\mb I_F} = G'^\vee = \left\{ \begin{array}{ll}
Sp_{2n}(\C) & \text{if } 2n / d^\vee \text{ is odd} \\
PSp_{2n}(\C) \times \{\pm 1\}& \text{if } 2n / d^\vee \text{ is even and } d^\vee \text{ is even} \\
PSp_{2n}(\C) & \text{if } 2n / d^\vee \text{ is even and } d^\vee \text{ is odd} 
\end{array}
\right.
\]
\item $\mathbf{C\mi BC_n}$. Similarly to the previous case, for every group $G^\vee$ isogenous
to $SL_{2n+1}(\C)$:
\[
(G^\vee)^{\mb I_F} = SO_{2n+1}(\C) = G'^\vee .
\]
\item $\mathbf{C\mi B_n}$. We endow $\C^{2n+2}$ with the symmetric bilinear form given by\\
$\langle e_i, e_j \rangle = \delta_{j,2n+3-i}$.
We let $\mb I_F$ act on $G^\vee = Spin_{2n+2}(\C)$ via conjugation by
$I_n \oplus \matje{0}{1}{1}{0} \oplus I_n \in O_{2n+2}(\C)$. There are three cases:
\[
\begin{array}{ccccccc}
(PSO_{2n+2}^{*,\vee})^{\mb I_F} & = & Spin_{2n+2}(\C)^{\mb I_F} & = & Spin_{2n+1}(\C) & = & G'^\vee ,\\
(SO_{2n+2}^{*,\vee})^{\mb I_F} & = & SO_{2n+2}(\C)^{\mb I_F} & = & O_{2n+1}(\C) & = & G'^\vee ,  \\
(Spin_{2n+2}^{*,\vee})^{\mb I_F} & = & PSO_{2n+2}(\C)^{\mb I_F} & = & SO_{2n+1}(\C) & = & G'^\vee .
\end{array}
\]
\item $\mathbf{{}^2 B\mi C_n}$. These cases are the same as for $\mathbf{B\mi C_n}$.

\item $\mathbf{{}^2 C\mi B_{2n}}$. Here $\mb I_F$ acts on $G^\vee$ as in $\mathbf{C\mi B_n}$. 
There are three cases:
\[
\begin{array}{ccccccc}
(PSO_{4n}^{*,\vee})^{\mb I_F} & = & Spin_{4n}(\C)^{\mb I_F} & = & Spin_{4n-1}(\C) & = & G'^\vee ,\\
(SO_{4n}^{*,\vee})^{\mb I_F} & = & SO_{4n}(\C)^{\mb I_F} & = & O_{4n-1}(\C) & = & G'^\vee ,  \\
(Spin_{4n}^{*,\vee})^{\mb I_F} & = & PSO_{4n}(\C)^{\mb I_F} & = & SO_{4n-1}(\C) & = & G'^\vee .
\end{array}
\]
\item $\mathbf{{}^2 C\mi B_{2n+1}}$. Again $\mb I_F$ acts on $G^\vee$ as in $\mathbf{C\mi B_n}$, and
\[
\begin{array}{ccccccc}
(PSO_{4n+2}^{*,\vee})^{\mb I_F} & = & Spin_{4n+2}(\C)^{\mb I_F} & = & Spin_{4n+1}(\C) & = & G'^\vee ,\\
(SO_{4n+2}^{*,\vee})^{\mb I_F} & = & SO_{4n+2}(\C)^{\mb I_F} & = & O_{4n+1}(\C) & = & G'^\vee ,  \\
(Spin_{4n+2}^{*,\vee})^{\mb I_F} & = & PSO_{4n+2}(\C)^{\mb I_F} & = & SO_{4n+1}(\C) & = & G'^\vee .
\end{array}
\]
\item $\mathbf{F_4^I}$. The group $\mb I_F$ acts on $G^\vee = E_{6,\Sc}(\C)$ via an outer
automorphism which stabilizes a pinning, and
\[
(G^\vee)^{\mb I_F} = E_{6,\Sc}(\C)^{\mb I_F} = F_4(\C) = G'^\vee .  
\]
The same holds with $\mc G = E_{6,\Sc}$ and $G^\vee = E_{6,\ad}(\C)$.
\item $\mathbf{G_2^I}$. In this case some elements of $\mb I_F$ act on $G^\vee = Spin_8 (\C)$
via an automorphism $\tau$ of order three which stabilizes a pinning, and maybe some other
elements of $\mb I_F$ act via an outer automorphism which stabilizes the same pinning. The 
$\mb I_F$-invariants are already determined by $\tau$:
\[
(G^\vee)^{\mb I_F} = Spin_8 (\C)^\tau = G_2 (\C) = G'^\vee .  
\]
The same holds with $\mc G = {}^r D_{4,\Sc}$ and $G^\vee (\C) = PSO_8 (\C)$.
\end{itemize}
Any two isomorphisms $G'^\vee \to (G^\vee)^{\mb I_F}$ differ by an automorphism of
$G'^\vee$. As $G'^\vee$ has type $B_n, C_n, F_4$ or $G_2$, all its automorphisms are inner.
\end{proof}

Let us compare the unramified L-parameters for $G$ and $G'$.

\begin{lem}\label{lem:4.2}
Lemma \ref{lem:4.1} induces a canonical bijection
$\lambda^\Phi_{G' G} : \Phi_\nr (G') \to \Phi_\nr (G)$.
\end{lem}
\begin{proof}
For $G'$ the group $\mb I_F$ acts trivially on $G'^\vee$. Hence the data for an unramified 
L-parameter are simple: a group homomorphism
\[
\phi' : \mb W_F / \mb I_F \times SL_2 (\C) \to G'^\vee \rtimes \mb W_F / \mb I_F 
\]
which is algebraic on $SL_2 (\C)$ and with $\phi' (\Fr) \in G'^\vee \Fr$ semisimple. To get
$\Phi_\nr (G')$, we consider such $\phi'$ up to conjugation by $G'^\vee$.

In an unramified L-parameter $\phi$ for $G$, $\phi (i) = (1,i)$ for all $i \in \mb I_F$.
The semisimple element $\phi (\Fr) = (s,\Fr) \in {}^L G$ must satisfy
\begin{multline*}
(1, \Fr \,i\, \Fr^{-1}) = \phi (\Fr \,i\, \Fr^{-1}) \\
= \phi (\Fr) \phi (i) \phi (\Fr)^{-1} = (s, \Fr) (1,i) (s,\Fr)^{-1} \\
= (s, \Fr \,i\, \Fr^{-1})(s^{-1},1) = (s (\Fr \,i\, \Fr^{-1})(s^{-1}), \Fr \,i\, \Fr^{-1}) .
\end{multline*}
Hence $(\Fr \,i\, \Fr^{-1})(s)$ must equal $s$ for $i \in \mb I_F$, which says that
$s \in (G^\vee)^{\mb I_F}$. The remaining content of $\phi$ is an algebraic group homomorphism
\[
SL_2 (\C) \to Z_{(G^\vee)^{\mb I_F}}(\phi (\Fr)).
\]
In principle unramified L-parameters for $G$ are considered up to conjugation by elements of
$G^\vee$. But if they must stay unramified, we may only conjugate by elements that 
centralize $(1,i)$ for all $i \in \mb I_F$, that is, by elements of $(G^\vee)^{\mb I_F}$. 

In view of the above, Lemma \ref{lem:4.1} provides a bijection 
\begin{equation}\label{eq:4.2}
\Phi_\nr (G') \to \Phi_\nr (G) : \phi' \mapsto \lambda^\vee_{G' G} \circ \phi' ,
\end{equation}
where $\lambda^\vee_{G' G}$ is extended to $G'^\vee \rtimes \mb W_F \to (G^\vee)^{\mb I_F}
\rtimes \mb W_F$ by making it the identity on $\mb W_F$. By definition inner automorphisms
have no effect on $\Phi (G')$, they are already divided out. Hence the unicity property in 
Lemma \ref{lem:4.1} entails that the bijection \eqref{eq:4.2} is canonical. 
\end{proof}

Enhancements of unramified L-parameters for $G$ and $G'$ can be compared in the same way.

\begin{lem}\label{lem:4.3}
Let $\phi' \in \Phi_\nr (G')$ and let $\phi \in \Phi_\nr (G)$ be its image under Lemma 
\ref{lem:4.2}.
\enuma{
\item Lemma \ref{lem:4.1} induces an isomorphism $Z_{G'^\vee}(\phi') \to Z_{G^\vee}(\phi)$,
which sends $Z(G'^\vee)^{\mb W_F}$ to $Z(G^\vee)^{\mb W_F}$.
\item Part (a) gives a canonical bijection from the set of $G'$-relevant enhancements
of $\phi'$ to the set of $G$-relevant enhancements of $\phi$.
}
\end{lem}
\begin{proof}
(a) As $Z_{G^\vee}(\phi) \subset (G^\vee)^{\mb I_F}$, this is a direct consequence of
Lemmas \ref{lem:4.1} and \ref{lem:4.2}.\\
(b) We find it easiest to proceed by classification.
For $\mathbf{B\mi C_n , C\mi BC_n , C\mi B_n , F_4^I}$ and $\mathbf{G_2^I}$ the $F$-groups 
$\mc G$ and $\mc G'$ are quasi-split, so $\chi^e_{\mc G} = \mr{triv} = \chi^e_{\mc G'}$. 
Then the relevant enhancements of $\phi$ are $\Irr (\mc S_\phi / Z({G^\vee}_\Sc)) = 
\Irr \big( Z_{G^\vee} (\phi) / Z(G^\vee)^{\mb W_F} \big)$, and similarly for $\phi'$.
Clearly part (a) induces a group isomorphism $Z_{G'^\vee} (\phi) / Z(G'^\vee)^{\mb W_F} \to
Z_{G^\vee} (\phi) / Z(G^\vee)^{\mb W_F}$, which settles these cases.

For $\mathbf{{}^2 B\mi C_n , {}^2 C\mi B_{2n}}$ and $\mathbf{{}^2 C\mi B_{2n+1}}$, the $F$-groups 
$\mc G$ and $\mc G'$ are the unique non-quasi-split inner twists of a quasi-split group, so 
$\chi_{\mc G}$ and $\chi_{\mc G'}$ both equal the unique nontrivial character of 
\[
Z({G'^\vee}_\Sc)^{\mb W_F} \cong Z({G^\vee}_\Sc)^{\mb W_F} \cong \Z / 2 \Z . 
\]
In these cases, by inspection, $(\mc G_\ad )'$ equals ${\mc G'}_\ad$. 
From that and Lemma \ref{lem:4.1} for $\mc G_\ad$ we deduce
\begin{equation}\label{eq:4.16}
{G^{'\vee}}_\Sc = {G_\ad}^{'\vee} \cong ({G_\ad}^\vee)^{\mb I_F} = ({G^\vee}_\Sc )^{\mb I_F} .
\end{equation}
One checks directly that $\mb W_F$ acts trivially on $Z({G'^\vee}_\Sc))$, so that
$\mc S_{\phi'} = \pi_0 (Z_{{G'^\vee}_\Sc}(\phi'))$. The $G'$-relevant enhancements of $\phi'$
are the irreducible representations of $\mc S_\phi$ with nontrivial $Z({G'^\vee}_\Sc$-character.
By \eqref{eq:4.16} and part (a) for $G_\ad$, these are matched bijectively with the irreducible
representations of $\pi_0 (Z_{{G^\vee}_\Sc}(\phi))$ whose $Z({G^\vee}_\Sc)^{\mb W_F}$-character
is nontrivial. By Lemma \ref{lem:7.1}.c (the part that we need is elementary, it relies only
on \eqref{eq:7.6}) that set of representations is naturally in bijection with the set of 
irreducible representations of $\mc S_\phi$ whose $Z({G^\vee}_\Sc)$-character equals $\chi_{\mc G}^e$.
\end{proof}

Recall that the group of weakly unramified characters of $X_\Wr (G)$ is naturally isomorphic 
with $(Z (G^\vee)^{\mb I_F})_\Fr$. The latter group acts on $\Phi_e (G)$ by
\begin{equation}\label{eq:4.1}
z \cdot (\phi,\rho) = (z \phi, \rho) , \qquad (z \phi)(\Fr) = z (\phi (\Fr)) .
\end{equation}
where $z \phi = \phi$ on $\mb I_F \times SL_2 (\C)$. The constructions in Section 
\ref{sec:list} entail that 
\begin{equation}\label{eq:4.3}
(Z (G^\vee)^{\mb I_F})_\Fr \cong \Irr (\Omega_G) \cong 
\Irr (\Omega_{G'}) \cong (Z (G'^\vee)^{\mb I_F})_\Fr .
\end{equation}
We call this twisting a Langlands parameter by a weakly unramified character. 
Recall the definition of cuspidality for enhanced L-parameters from  \cite[\S 6]{AMS1}.
 
\begin{prop}\label{prop:4.4}
Lemma \ref{lem:4.1} induces a canonical bijection 
\[
\lambda^{\Phi_e}_{G' G} : \Phi_{\nr,e}(G') \to \Phi_{\nr,e}(G) .
\]
This map and its inverse preserve boundedness, discreteness, cuspidality and twists by 
$(Z (G'^\vee)^{\mb I_F})_\Fr \cong (Z (G^\vee)^{\mb I_F})_\Fr$ as in \eqref{eq:4.1}.
\end{prop}
\begin{proof}
The map $\lambda^{\Phi_e}_{G' G}$, its bijectivity and its canonicity come from
Lemmas \ref{lem:4.1}, \ref{lem:4.2} and \ref{lem:4.3}. It is clear from the construction
that this bijection preserves weakly unramified twists by the group \eqref{eq:4.3}.

We consider an arbitrary $(\phi' ,\rho' ) \in \Phi_{\nr,e}(G')$ and we write 
$(\phi,\rho) = \lambda^{\Phi_e}_{G' G} (\phi' ,\rho' )$. 
Boundedness of $\phi$ depends only on $s = \phi (\Fr) \Fr^{-1} \in G^\vee$. We saw in
the proof of Lemma \ref{lem:4.2} that $s \in (G^\vee)^{\mb I_F}$, while by construction
\[
\lambda_{G' G}^{\vee\, -1}(s) = \phi' (\Fr) \Fr^{-1} \in G'^\vee .
\]
Hence $\phi$ is bounded if and only if $\phi'$ is bounded.

As observed in \cite[\S 3.2]{GrRe}, $\phi$ is discrete if and only if 
$Z_{G^\vee}(\phi) / Z(G^\vee)^{\mb W_F}$ is finite. By Lemma \ref{lem:4.3}.a this is 
equivalent to finiteness of $Z_{G'^\vee}(\phi) / Z(G'^\vee)^{\mb W_F}$.

The construction in Lemma \ref{lem:4.2} entails that 
\[
Z_{G^\vee}(\phi (\mb W_F)) = \lambda^\vee_{G' G} \big( Z_{G'^\vee}(\phi' (\mb W_F)) \big) 
\quad \text{and} \quad 
\phi |_{SL_2 (\C)} = \lambda^\vee_{G' G} \circ \phi' |_{SL_2 (\C)}.
\] 
Similarly, in Lemma \ref{lem:4.3}.b we defined $\rho = \rho' \circ \lambda_{G' G}^{\vee\, -1}$.
Hence cuspidality of $(\phi,\rho)$ depends only on the pair $(u_\phi,\rho)$ for the
group $Z_{G^\vee}(\phi (\mb W_F))$. The situation for $(\phi',\rho')$ is entirely
analogous, with objects isomorphic to those for $(\phi,\rho)$. Hence one of these 
enhanced L-parameters if cuspidal if and only if the other is so.
\end{proof}

Let $\mc P = \mc M \mc U$ be a standard parabolic $F$-subgroup of $\mc G$ and let
$\mc P' = \mc M' \mc U'$ be the associated standard parabolic $F$-subgroup of $\mc G'$.
Then $\lambda^\vee_{G' G}$ restricts to an isomorphism $M'\vee \to (M^\vee)^{\mb I_F}$.
(When $\mc G' = \mc G'^\circ \times \{\pm 1\}$, we take $\mc M' = \mc M'^\circ \times \{\pm 1\}$
and $M'^\vee = M'^{\vee,\circ} \times \{ \pm 1 \}$.) Then $\mb W_F$ acts trivially on
$M'^\vee$ and the character $\chi_{\mc M'}$ of $Z(M'^\vee)$ can be deduced from $\chi_{\mc G'}$
via \cite[Lemma 6.6]{AMS1}. As in Lemmas \ref{lem:4.1}--\ref{lem:4.3} and Proposition 
\ref{prop:4.4} we obtain a canonical bijection
\begin{equation}\label{eq:4.13}
\lambda^{\Phi_e}_{M' M} : \Phi_{\nr,e}(M') \to \Phi_{\nr,e}(M) . 
\end{equation}
The cuspidal support of an element of $\Phi_{\nr,e}(G)$ \cite[\S 7]{AMS1} can be realized
as an element of $\Phi_{\nr,\cusp}(M)$ for a standard $F$-Levi subgroup $\mc M$ of $\mc G$.

\begin{lem}\label{lem:4.7}
The system of bijections $\lambda^{\Phi_e}_{M' M}$ (running over all standard $F$-Levi
subgroups) commutes with the cuspidal support map for enhanced L-parameters.
\end{lem}
\begin{proof}
Recall from \cite[Proposition 7.3 and Definition 7.7]{AMS1} that the cuspidal support 
$\mb{Sc}(\phi,\rho)$ of $(\phi,\rho) \in \Phi_{\nr,e}(G)$ 
\begin{itemize}
\item has the same $\phi |_{\mb I_F}$,
\item is a cuspidal L-parameter for a (standard) $F$-Levi subgroup $M$ of $G$,
\item is determined entirely by a construction in the complex reductive group
$Z^1_{{G_\ad}^\vee} (\phi |_{\mb I_F}) = {G_\ad}^{\vee,\mb I_F}$, with $\phi (\Fr), u_\phi$
and $\rho$ as input.
\end{itemize}
Via Lemmas \ref{lem:4.1}--\ref{lem:4.3} and Proposition \ref{prop:4.4} all this is canonically
transferred to analogous objects with primes. It follows that
\begin{equation}\label{eq:4.14}
\mb{Sc} \big( \lambda^{\Phi_e}_{G' G}(\phi',\rho') \big) = 
\lambda^{\Phi_e}_{M' M} ( \mb{Sc}(\phi',\rho') ) . \qedhere
\end{equation}
\end{proof}

We denote the adjoint action of ${}^L G$ on $\mr{Lie}(G^\vee) / 
\mr{Lie}(Z(G^\vee)^{\mb W_F})$ by $\mr{Ad}_{G^\vee}$. The adjoint $\gamma$-factor of 
$\phi \in \Phi (G)$ is related to $\epsilon$-factors and L-functions as
\begin{equation}\label{eq:4.4}
\gamma (s,\mr{Ad}_{G^\vee} \circ \phi,\psi) = \epsilon (s,\mr{Ad}_{G^\vee} \circ \phi,\psi)
L(1-s,\mr{Ad}_{G^\vee} \circ \phi) L(s,\mr{Ad}_{G^\vee} \circ \phi)^{-1} .
\end{equation}
Here $s \in \C$ and $\psi : F \to \C^\times$ is an additive character, which by our 
conventions from \eqref{eq:2.11} must have order 0. 
For the definitions of the local factors in \eqref{eq:4.4} we refer to \cite{Tat,GrRe}. 
Let $\mr{Ad}_{(G^\vee)^{\mb I_F}}$ denote the 
adjoint action of $(G^\vee)^{\mb I_F} \rtimes \mb W_F$ on $\mr{Lie}((G^\vee)^{\mb I_F}) / 
\mr{Lie}(Z(G^\vee)^{\mb W_F})$. For an unramified L-parameter $\phi$, $\mr{Ad}_{G^\vee} \circ 
\phi$ can be restricted to $\mr{Ad}_{G^{\vee,\mb I_F}} \circ \phi$.

\begin{lem}\label{lem:4.5}
Let $\mc G$ be any connected reductive $F$-group and let $\phi \in \Phi_\nr (G)$. There exists 
$\varepsilon \in \{\pm 1, \pm \sqrt{-1}\}$, with $\varepsilon^2$ depending only on 
$\mr{Lie}(G^\vee)$, such that
\[
\gamma (0,\mr{Ad}_{G^\vee} \circ \phi,\psi) = \varepsilon 
q_F^{\mb a (\mr{Lie}(G^\vee)) /2} \gamma (0,\mr{Ad}_{(G^\vee)^{\mb I_F}} \circ \phi,\psi) .
\]
\end{lem}
\begin{proof}
Let Lie$(G^\vee)^{ram}$ be the "ramified part" of Lie$(G^\vee)$, that is, the sum of
the nontrivial irreducible $\mb I_F$-subrepresentations. Since $\mb I_F$ is normal in
$\mb W_F$, this gives a decomposition of $\mb W_F$-representations
\begin{equation}\label{eq:4.9}
\mr{Lie}(G^\vee) = \mr{Lie}(G^\vee)^{\mb I_F} \, \oplus \, \mr{Lie}(G^\vee)^{ram} 
\end{equation}
As unramified L-parameters act via $\mb W_F$ and $(G^\vee)^{\mb I_F}$, \eqref{eq:4.9} 
can also be considered as a decomposition of $\mb W_F \times SL_2 (\C)$-representations, 
which we write as
\[
\mr{Ad}_{G^\vee} \circ \phi = \mr{Ad}_{\mr{Lie}(G^\vee)^{\mb I_F}} \circ \phi 
\, \oplus \, \mr{Ad}_{\mr{Lie}(G^\vee)^{ram}} \circ \phi .
\]
By the additivity of $\gamma$-factors
\begin{equation}\label{eq:4.5}
\gamma (s, \mr{Ad}_{G^\vee} \circ \phi ,\psi) = 
\gamma (s, \mr{Ad}_{\mr{Lie}(G^\vee)^{\mb I_F}} \circ \phi ,\psi)
\gamma (s, \mr{Ad}_{\mr{Lie}(G^\vee)^{ram}} \circ \phi ,\psi) \quad \forall s \in \C.
\end{equation}
Further $(\mr{Lie}(G^\vee)^{ram})^{\mb I_F} = 0$, so 
\[
L(s,\mr{Lie}(G^\vee)^{ram}) = 1 \quad \text{and} \quad 
\epsilon (s, \mr{Ad}_{\mr{Lie}(G^\vee)^{ram}} \circ \phi ,\psi) =
\epsilon (s, \mr{Ad}_{\mr{Lie}(G^\vee)^{ram}} \circ \phi |_{\mb W_F} ,\psi) .
\]
With that \eqref{eq:4.5} becomes
\begin{equation}\label{eq:4.6}
\gamma (s, \mr{Ad}_{G^\vee} \circ \phi ,\psi) = \gamma (s, \mr{Ad}_{\mr{Lie}
(G^\vee)^{\mb I_F}} \circ \phi ,\psi) \epsilon (s, \mr{Ad}_{\mr{Lie}(G^\vee)^{ram}} 
\circ \phi |_{\mb W_F} ,\psi) \quad \forall s \in \C.
\end{equation}
It was observed in \cite[\S 3.2]{GrRe} that $\mr{Ad}_{G^\vee} \circ \phi$ and
$\mr{Ad}_{G^\vee} |_{\mf g^{\vee,ram}} \circ \phi$ are self-dual with respect 
to the Killing form. By \cite[(15)]{GrRe} this implies
\[
\epsilon (1/2, \mr{Ad}_{G^\vee} |_{\mf g^{\vee,ram}} \circ \phi |_{\mb W_F} ,\psi) 
= \varepsilon
\]
for some $\varepsilon \in \{\pm 1, \pm \sqrt{-1}\}$ with $\varepsilon^2$ depending only on 
$\mr{Lie}(G^\vee)$. Then \cite[(3.4.5)]{Tat} says
\begin{equation}\label{eq:4.7}
\epsilon (0, \mr{Ad}_{\mr{Lie}(G^\vee)^{ram}} \circ \phi |_{\mb W_F} ,\psi) = 
\varepsilon q_F^{\mb a (\mr{Lie}(G^\vee)^{ram} \circ \phi |_{\mb W_F}) / 2} .
\end{equation}
By definition the Artin conductor of a $\mb W_F$-representation $V$ depends only on
the restriction to $\mb I_F$, and $\mb a (V) = 0$ if $V^{\mb I_F} = V$. As $\phi$ is
unramified and by \eqref{eq:4.9}
\begin{equation}\label{eq:4.8}
\mb a (\mr{Lie}(G^\vee)^{ram} \circ \phi |_{\mb W_F} ) = 
\mb a (\mr{Lie}(G^\vee)^{ram}) = \mb a (\mr{Lie}(G^\vee)) ,
\end{equation}
where Lie$(G^\vee)$ and Lie$(G^\vee)^{ram}$ are endowed with the $\mb W_F$-action 
derived from conjugation inside ${}^L G$. Now combine \eqref{eq:4.6}, \eqref{eq:4.7} 
and \eqref{eq:4.8}.
\end{proof}

For discrete L-parameters, the conjectures from \cite{HII} assert that 
$|\gamma (0,\mr{Ad}_{G^\vee} \circ \phi,\psi)|$ is related to the formal degree of any member
of the L-packet $\Pi_\phi (G)$. 

We return to ramified simple groups. We know from Lemma \ref{lem:3.5} how formal degrees 
behave under the transfer from $\Rep (G)_\unip$ to $\Rep (G')_\unip$. It turns out that 
adjoint $\gamma$-factors behave in the same way. 
Let $\phi' \in \Phi_\nr (G')$ and write $\phi = \lambda^\Phi_{G' G}(\phi)$.
By Lemmas \ref{lem:4.1} and \ref{lem:4.2}
\[
\gamma (s, \mr{Ad}_{\mr{Lie} (G^\vee)^{\mb I_F}} \circ \phi ,\psi) =
\gamma (s, \mr{Ad}_{G'^\vee} \circ \phi', \psi) \qquad \forall s \in \C. 
\]
Then Lemma \ref{lem:4.5} says
\begin{equation}\label{eq:4.15}
\gamma (0,\mr{Ad}_{G^\vee} \circ \phi,\psi) = 
\varepsilon q_F^{\mb a (\mr{Lie}(G^\vee)) /2} \gamma (0,\mr{Ad}_{G'^\vee} \circ \phi',\psi) .
\end{equation}
With the material from Sections \ref{sec:Hecke} and \ref{sec:Lpar} we obtain a good
candidate for a local Langlands correspondence for unipotent representations of
ramified simple groups. Recall that a LLC 
\begin{equation}\label{eq:4.10}
\Irr (G')_\unip \to \Phi_{\nr,e}(G') : \pi' \mapsto (\phi_{\pi'}, \rho_{\pi'})
\end{equation}
with many nice properties was constructed in \cite{SolLLC} and \cite[Theorem 2.1]{FOS2}.
For supercuspidal representations of adjoint groups this agrees with \cite{LusUni1,Mor2}, 
for other unipotent representations it differs from the earlier constructions of Lusztig.
For supercuspidal representations of ramified simple groups there are some arbitrary
choices in \cite{SolLLC}, which stem from \cite{FOS1}. From \cite[\S 4.5.1]{Opd18} we 
get some additional requirements related to suitable spectral transfer morphisms. 
As in \cite{FOS2}, we use those requirements to fix some of the choices in \cite{FOS1}. 

Consider the composition of Corollary \ref{cor:3.2}, \eqref{eq:4.10} and Proposition \ref{prop:4.4}:
\begin{equation}\label{eq:4.11}
\begin{array}{ccccccc}
\Rep (G)_\unip & \to & \Rep (G')_\unip & \to & \Phi_{\nr,e}(G') & \to & \Phi_{\nr,e}(G) \\
\pi & \mapsto & \pi' & \mapsto & (\phi_{\pi'},\rho_{\pi'}) & \mapsto &
\lambda^{\Phi_e}_{G' G} (\phi_{\pi'},\rho_{\pi'}) .
\end{array}.
\end{equation}
All involved maps are bijective, so we obtain a bijection
\begin{equation}\label{eq:4.12}
\begin{array}{ccc}
\Rep (G)_\unip & \longrightarrow & \Phi_{\nr,e}(G) \\
\pi & \mapsto & (\phi_\pi, \rho_\pi)
\end{array}.
\end{equation}

\begin{thm}\label{thm:4.6}
Let $\mc G$ be a ramified simple group. The bijection \eqref{eq:4.12} satisfies:
\enuma{
\item $\pi$ is tempered if and only if $\phi_\pi$ is bounded.
\item $\pi$ is essentially square-integrable if and only if $\phi_\pi$ is discrete.
\item $\pi$ is supercuspidal if and only if $(\phi_\pi,\rho_\pi)$ is cuspidal. 
\item Let $\chi \in X_\Wr (G)$ correspond to $\hat \chi \in (Z(G^\vee)^{\mb I_F})_\Fr$.
Then $(\phi_{\chi \otimes \pi},\rho_{\chi \otimes \pi}) = (\hat \chi \phi_\pi,\rho_\pi)$.
\item The HII conjectures hold for tempered $\pi \in \Irr (G)_\unip$.
\item Equivariance with respect to $\mb W_F$-automorphisms of the Dynkin diagram of $\mc G$.
}
\end{thm}
\begin{proof}
(a) and (b) These follow from Lemma \ref{lem:3.3}, Proposition \ref{prop:4.4}
and \cite[Lemmas 5.6 and 5.7]{SolLLC}. \\
(c) Follows from Corollary \ref{cor:2.4}, \cite[Lemma 5.4]{SolLLC} and Lemma \ref{lem:4.7}.\\
(d) This is a consequence of Lemma \ref{lem:3.7}, \cite[Lemma 5.3]{SolLLC} 
and Proposition \ref{prop:4.4}.\\
(e) For $G'$ this was shown in \cite[Theorem 5.3]{FOS2}. When we pass from $G$ to $G'$, Lemma 
\ref{lem:3.5} says that the formal degree a square-integrable modulo centre representation
picks up a factor $q_F^{\mb{a}(\mf g^\vee)/2}$. By \eqref{eq:4.15} the absolute value of the 
involved adjoint $\gamma$-factor picks up the same factor. The other ingredients of the 
HHI-formula for formal degrees (the size of a component group and the dimension of an irreducible 
representation thereof) do not change when passing from $G$ to $G'$, by Lemma \ref{lem:4.3}. \\
(f) From page \pageref{lem:2.5} we know that it suffices to consider the single diagram
automorphism $\tau$. In Lemma \ref{lem:3.6} we showed that it fixes every $\pi \in \Irr_\unip (G)$.
It remains to show that $\tau$ acts trivially on $\Phi_{\nr,e}(G)$. From the classification
on page \pageref{lem:2.5} we also see that, on $G^\vee$, $\tau$ coincides with an element of
$\mb I_F$. By Lemma \ref{lem:4.2} im$(\phi_\pi) \subset (G^\vee)^{\mb I_F} \rtimes \mb W_F$
and by Lemma \ref{lem:4.3} $Z_{G^\vee}(\phi_\pi) \subset (G^\vee)^{\mb I_F}$. Hence $\tau$
indeed fixes $(\phi_\pi, \rho_\pi)$.
\end{proof}

\section{Supercuspidal unipotent representations}
\label{sec:supercusp}

In \cite{FOS1} it was assumed that all reductive $F$-groups under consideration split over 
an unramified extension of $F$. In this section we will lift that condition, and we generalize
all the results from that paper to arbitrary connected reductive $F$-groups.
Let us formulate the generalization of the main results of \cite{FOS1} that we are after.

\begin{thm}\label{thm:5.1}
Let $\mc G$ be a connected reductive $F$-group. There exists a bijection
\[
\begin{array}{ccc}
\Irr (G)_{\cusp,\unip} & \longrightarrow & \Phi_\nr (G)_\cusp \\
\pi & \mapsto & (\phi_\pi, \rho_\pi)
\end{array}
\]
with the following properties:
\enuma{
\item Equivariance with respect to twisting by weakly unramified characters.
\item Equivariance with respect to $\mb W_F$-automorphisms of the root datum of $\mc G$.
\item Compatibility with almost direct products of reductive groups. 
\item Suppose that $\pi \in \Irr (G)_{\cusp,\unip}$ is a constituent of the pullback of \\
$\pi_\ad \in \Irr (G_\ad)_{\cusp,\unip}$ to $G$. Then the canonical map 
\[
{G_\ad}^\vee \rtimes \mb W_F \to G^\vee \rtimes \mb W_F
\]
sends $\phi_{\pi_\ad}$ to $\phi_\pi$.
\item Let $Z(\mc G)_s$ be the maximal $F$-split central torus of $\mc G$. When $\pi$ is unitary:
}
\[
\mr{fdeg}(\pi, d\mu_{G,\psi}) = \dim (\rho_\pi) \big| \pi_0 \big( Z_{(\mc G / Z(\mc G)_s)^\vee}
(\phi) \big) \big|^{-1} |\gamma (0,\mr{Ad}_{G^\vee} \circ \phi_\pi, \psi)| .
\]
For a given $\pi$ the properties (a), (c), (d) and (e) determine $\phi_\pi$ uniquely,
up to twists by weakly unramified characters of $(\mc G / Z(\mc G)_s)(F)$.
\end{thm}

Most of the time we will assume that the centre of $\mc G$ is $F$-anisotropic. 
For such groups we recall the definitions of a few relevant numbers from \cite[\S 2]{FOS1}.
Let $\phi \in \Phi_\nr (G)$ and $\sigma \in \Irr_{\cusp,\unip} (P_{\mf f})$. 
\begin{itemize}
\item $\sfa$ is the number of $\tilde \phi \in \Phi_\nr^2 (G)$ which admit a $G$-relevant 
cuspidal enhancement and for each $F$-simple factor $\mc G_i$ of $\mc G$ satisfy
\[
\gamma (0,\mr{Ad}_{G_i^\vee} \circ \tilde \phi, \psi) = 
c_i \gamma (0,\mr{Ad}_{G_i^\vee} \circ \phi,\psi)
\] 
for some $c_i \in \Q^\times$ (as rational functions of $q_F$).
\item $\sfb$ is the number of $G$-relevant cuspidal enhancements of $\lambda$.
\item $\sfa'$ is defined as $|\Omega_{G,\mf f}|$ times the number of $G$-conjugacy
classes of maximal parahoric subgroups $P_{\tilde{\mf f}} \subset G$ for which there exists a 
$\tilde \sigma \in \Irr_{\cusp,\unip} (P_{\tilde{\mf f}})$ such that the components 
$\sigma_i, \tilde{\sigma}_i$ corresponding to any $F$-simple factor $\mc G_i$ of $\mc G$ satisfy
\[
\mr{fdeg} (\mr{ind}_{P_{\tilde{\mf f'},i}^G}(\tilde{\sigma}_i) = 
\tilde{c}_i \mr{fdeg} (\mr{ind}_{P_{\mf f,i}^G}(\sigma_i)
\] 
for some $\tilde{c}_i \in \Q^\times$ (as rational functions of $q_F$).
\item $\sfb'$ is the number of $\tilde \sigma \in \Irr_{\cusp,\unip}(P_{\mf f})$ with
$\dim (\tilde \sigma) = \dim (\sigma)$.
\end{itemize}

Suppose now that $\mc G$ is semisimple. From \cite[Theorem 2]{FOS2}, Theorem \ref{thm:4.6}
and compatibility with direct products of simple groups we get a map
\begin{equation}\label{eq:5.1}
\Irr (G)_{\cusp,\unip} \to \Phi_\nr^2 (G) / X_\Wr (G) ,
\end{equation}
where $X_\Wr (G)$ is identified with $(Z(G^\vee)^{\mb I_F})_\Fr$.
Fix an irreducible direct summand $\pi$ of $\mr{ind}_{P_{\mf f}}^G (\sigma)$ and
let $X_\Wr (G) \phi_\pi$ be its image under \eqref{eq:5.1}.

\begin{lem}\label{lem:5.2}
Let $\mc G$ be a ramified simple $F$-group and let $\mc G'$ be its $F_\nr$-split 
companion group. Let $\mf f, \sigma, \pi, \phi_\pi$ be as above and let $\mf f',
\sigma', \pi', \phi_{\pi'}$ be their images under the maps from 
Proposition \ref{prop:2.1}, Theorem \ref{thm:2.3}, Corollary \ref{cor:2.4} and
Lemma \ref{lem:4.2}. Then the numbers $\sfa, \sfb, \sfa', \sfb'$ for $G,\mf f, \sigma,
\lambda_\pi$ are the same as their counterparts for $G', \mf f', \sigma', \phi_{\pi'}$.
\end{lem}
\begin{proof}
For $\sfa$ and $\sfb$ this follows from Proposition \ref{prop:4.4} and Lemma \ref{lem:4.5}.
For $\sfa'$ and $\sfb'$ it is a consequence of Proposition \ref{prop:2.1}, Theorem
\ref{thm:2.3} and Corollary \ref{cor:2.4}.
\end{proof}

\emph{Proof of Theorem \ref{thm:5.1} for adjoint groups}\\
For simple adjoint groups Theorem \ref{thm:5.1} is established case-by-case, as explained
in \cite[\S 12]{FOS1}. The ramified simple adjoint groups are not considered in
\cite{FOS1}, for those we use Theorem \ref{thm:4.6} to associate enhanced L-parameters
to $\Irr (G)_{\cusp,\unip}$. By Corollary \ref{cor:2.4}, Lemma \ref{lem:4.5} and 
\cite[Theorem 1]{FOS1}, these are essentially (in a sense specified in that paper) the 
only L-parameters that make the HII conjectures true for $\Irr (G)_{\cusp,\unip}$.

With Sections \ref{sec:unirep} and \ref{sec:Lpar} and Lemma \ref{lem:5.2}, we transfer all
further issues in the proof of \cite[Proposition 12.1]{FOS1} to the group $G'$, which is
treated in \cite{FOS1}. The generalization from simple adjoint groups to all adjoint
groups in \cite[Proposition 12.2 and Lemma 16.1]{FOS1} works equally well for ramified groups.
We note that restriction of scalars is dealt with in \cite[Appendix]{FOS1}, which is 
already written in the generality of reductive groups. \qed \\

For non-adjoint reductive $F$-groups we have to be more careful. It appears that for 
semisimple $F$-groups the proof of Theorem \ref{thm:5.1} in \cite[\S 13--14]{FOS1} can be 
modified without too much trouble. However, the arguments for reductive $F$-groups with 
anisotropic centre in \cite[\S 15]{FOS1} do not easily carry over to ramified groups. The main
difference is that in the $F_\nr$-split case the inclusion $G_\der \to G$ induces a bijection 
$\Irr (G)_{\cusp,\unip} \to \Irr (G_\der)_{\cusp,\unip}$ \cite[Lemma 15.3]{FOS1}. For ramified 
groups this is just false, firstly because ramified anisotropic tori can admit 
nontrivial weakly unramified characters, secondly because the pullback map 
$\Irr (G)_{\cusp,\unip} \to \Rep (G_\der)$ need not preserve irreducibility, and thirdly because
not all elements of $\Irr (G_\der)_{\cusp,\unip}$ are contained in a representation pulled
back from $G$.
In view of this, we rather aim to extend the arguments from \cite[\S 13--14]{FOS1} to all
(possibly ramified) reductive $F$-groups $\mc G$ with anisotropic centre. 

Thus, we want to
reduce Theorem \ref{thm:5.1} for $\mc G$ to Theorem \ref{thm:5.1} for its adjoint group
$\mc G_\ad$. One problem is that, in contrast with the $F_\nr$-split case, the natural
map $\Omega_G \to \Omega_{G_\ad}$ need not be injective. Equivalently, the natural map
\begin{equation}\label{eq:5.7}
X_\Wr (G_\ad) \cong (Z({G^\vee}_\Sc )^{\mb I_F})_\Fr \to (Z(G^\vee )^{\mb I_F} )_\Fr 
\cong X_\Wr (G) 
\end{equation}
need not be surjective. Clearly, the image of \eqref{eq:5.7} is contained in the kernel
of the natural map to 
\[
X_\Wr (Z(\mc G)^\circ (F)) \cong \big( Z(G^\vee / {G^\vee}_\der )^{\mb I_F} \big)_\Fr .
\]
But even for ramified simple groups, \eqref{eq:5.7} fails to be surjective in two cases:
\begin{itemize}
\item $\mc G = SU_{2n} / \mu_{2n / d}$, where $\mu_k$ denotes the group scheme of $k$-th
roots of unity, $d \in 2 \N$ and $2n / d \in 2 \N$.
Then $G^\vee = SL_{2n}(\C) / \mu_d (\C)$, ${G^\vee}_\Sc = SL_{2n}(\C)$,
$(Z({G^\vee}_\Sc)^{\mb I_F})_\Fr = \{1,-1\}$ and $(Z(G^\vee)^{\mb I_F})_\Fr =
\{1,\exp (\pi i / d)\}$. 
\item $\mc G = SO_{2n}^*$. Then $G^\vee = SO_{2n}(\C)$, ${G^\vee}_\Sc = Spin_{2n}(\C)$,
$(Z({G^\vee}_\Sc)^{\mb I_F})_\Fr = \ker (Spin_{2n}(\C) \to SO_{2n}(\C))$ and
$(Z(G^\vee)^{\mb I_F})_\Fr = \{1,-1\}$.
\end{itemize}
We note that in the first case 
\[
(G^\vee)^{\mb I_F} = PSp_{2n}(\C) \times \langle \exp (\pi i / d) \rangle / 
\langle \exp (2\pi i / d) \rangle 
\]
and in the second case
\[
(G^\vee)^{\mb I_F} = O_{2n-1}(\C) = SO_{2n-1}(\C) \times \{1,-1\} .
\]
In both cases there are natural isomorphisms
\begin{equation}\label{eq:5.2}
(Z(G^\vee)^{\mb I_F})_\Fr = Z(G^\vee)^{\mb W_F} = Z(G^{\vee,\mb W_F}) \cong 
\pi_0 (G^{\vee,\mb I_F}) = \pi_0 (G^{\vee,\mb I_F})_\Fr.
\end{equation}
For all other simple groups $\mc G$, $G^{\vee,\mb I_F}$ is connected, see the proof
of Lemma \ref{lem:4.1}. With the list in Section \ref{sec:list}, one checks that for
any simple $F$-group $\mc G$, $\ker (\Omega_G \to \Omega_{G_\ad})$ is naturally 
isomorphic to $\pi_0 (\mc G')$ (and it is trivial when $\mc G$ is $F_\nr$-split).

For any homomorphism of connected reductive $F$-groups $\mc G \to \mc H$, we define
\[
\Omega_{G,H} := \mr{im}(\Omega_G \to \Omega_H) \quad \text{and} \quad
X_\Wr (H,G) := \mr{im}(X_\Wr (H) \to X_\Wr (G)) .
\]
For $\mc H = \mc G_\ad$ we obtain short exact sequences (dual to each other):
\begin{equation}\label{eq:5.3} 
\begin{array}{ccccccccc}
1 & \to & \ker (\Omega_G \to \Omega_{G_\ad}) & \to & \Omega_G & \to & \Omega_{G,G_\ad} & \to & 1 , \\
1 & \leftarrow & \Irr ( \ker (\Omega_G \to \Omega_{G_\ad}) ) & \leftarrow & X_\Wr (G) & 
\leftarrow & X_\Wr (G_\ad ,G) & \leftarrow & 1 .
\end{array}
\end{equation}
We note also that the image of $\ker ({G^\vee}_\Sc \to G^\vee)$ in $X_\Wr (G_\ad) \cong 
(Z({G^\vee}_\Sc)^{\mb I_F})_\Fr$ is
\begin{equation}\label{eq:5.5}
\Irr (\Omega_{G_\ad} / \Omega_{G,G_\ad}) = \ker (X_\Wr (G_\ad) \to X_\Wr (G) ) .
\end{equation}
For semisimple groups, the method from \eqref{eq:5.2} yields a group 
$\pi_0 (G^{\vee,\mb I_F})_\Fr$ isomorphic to $X_\Wr (G) / X_\Wr (G_\ad,G)$. 
It is naturally represented in $X_\Wr (G) \cong (Z(G^\vee)^{\mb I_F})_\Fr$, 
and forms a complement to $X_\Wr (G_\ad ,G)$. Thus both sequences \eqref{eq:5.3} split
for semisimple groups (but not necessarily for reductive groups).\\

For $\sigma_\ad \in \Irr_{\cusp,\unip}(P_{\mf f,\ad})$, let 
$\sigma \in \Irr_{\cusp,\unip}(P_{\mf f})$ be its pullback. Morris showed in 
\cite[Proposition 4.6]{Mor2} that $\sigma_\ad$ can be extended to a representation 
$\sigma_\ad^N$ of $N_{G_\ad} (P_{\mf f,\ad})$ (on the same vector space). Let $\sigma^N 
\in \Irr (N_G (P_{\mf f}))$ be the pullback of $\sigma_\ad^N$ along
$G \to G_\ad$. This construction shows that 
\begin{equation}
\sigma \text{ can be extended to } N_G (P_{\mf f}), \text{ via } N_{G_\ad}(P_{\mf f,\ad}) . 
\end{equation}
Another issue with Theorem 
\ref{thm:5.1} concerns the pullback of supercuspidal unipotent representations along the
canonical map $\sfq : G \to G_\ad$. Such a pullback is trivial on $Z(G)$ and it can never 
involve elements of $X_\Wr (G)$ outside $X_\Wr (G,G_\ad)$. Notice that in general $G / Z(G)$
is a proper subgroup of $G_\ad = \mc G_\ad (F)$. We consider $\Irr (G / Z(G))$ as a subset
of $\Irr (G)$, endowed all the relevant notions from Rep$(G)$. 

We note also that a $G_\ad$-orbit of facets of $\mc B (\mc G,F)$ can decompose into
several $G$-orbits. These are parametrized by $G_\ad / N_{G_\ad}(P_{\mf f,\ad}) G$ (in 
such a quotient $G$ is a shorthand for its image in $G_\ad$). This can cause the pullback 
of an irreducible $G_\ad$-representation to become reducible (as $G$-representation).

\begin{lem}\label{lem:5.3}
\enuma{
\item The $G$-constituents of $\sfq^* (\mr{ind}_{P_{\mf f,\ad}}^{G_\ad} \sigma_\ad)$
are the representations 
\[
\mr{Ad}(g)^* \mr{ind}_{N_G (P_{\mf f})}^G (\sigma^N) =
\mr{ind}_{N_G (P_{g \mf f})}^G (g \cdot \sigma^N) ,
\]
where $g \in G_\ad$ and $\sigma^N$ is
an extension of $\sigma$ to $N_G (P_{\mf f})$ such that $Z(G) \subset \ker (\sigma^N)$.
\item Tensoring with weakly unramified characters provides a bijection
\[
X_\Wr (G) \times_{X_\Wr (G_\ad,G)} \Irr (G / Z(G))_{(P_{\mf f},\sigma)} 
\to \Irr (G)_{(P_{\mf f},\sigma)} .
\]
In particular $\Irr ( \ker (\Omega_G \to \Omega_{G_\ad}) )$ acts freely on 
$\Irr (G)_{(P_{\mf f},\sigma)} / X_\Wr (G_\ad,G) $.
}
\end{lem}
\begin{proof}
(a) This follows from \eqref{eq:2.13} and \eqref{eq:2.6}. \\
(b) By part (a) tensoring with $\chi \in X_\Wr (G)$ maps any element of 
$\Irr (G / Z(G))_{(P_{ \mf f},\sigma)}$ to $\Irr (G / Z(G))_{(P_{\mf f},\sigma)}$  
if and only if $\chi \in X_\Wr (G_\ad,G)$, and then $\chi$ stabilizes that set entirely.
Combine that with \eqref{eq:2.13} and \eqref{eq:5.3}. 
\end{proof}

On the Galois side something similar happens. Not every unramified $\phi \in \Phi (G)$ 
can be lifted along $\sfq^\vee : {G_\ad}^\vee \to G^\vee$ to an element of $\Phi (G_\ad)$. 

\begin{lem}\label{lem:5.4}
\enuma{
\item $\phi \in \Phi_\nr (G)$ can be lifted to an element of $\Phi_\nr (G_\ad)$ if and only
if $\phi (\Fr) \Fr^{-1}$ lies in $({G^\vee}_\der)^{\mb I_F,\circ}(1 - \Fr)(G^{\vee,\mb I_F})$.
\item The action of any $\chi \in X_\Wr (G)$ maps any element of $\sfq^\vee (\Phi_\nr (G_\ad))$ 
to $\sfq^\vee (\Phi_\nr (G_\ad))$ if and only if $\chi \in X_\Wr (G_\ad,G)$. This provides
a bijection
\[
X_\Wr (G) \times_{X_\Wr (G_\ad, G)} \sfq^\vee (\Phi_\nr (G_\ad)) \to
\Phi_\nr (G) .
\]
}
\end{lem}
\begin{proof}
(a) By reduction to the absolutely simple case and classification one sees that 
${G_\ad}^{\vee,\mb I_F} = ({G^\vee}_\Sc)^{\mb I_F}$ is always connected. Hence its image in 
$G^\vee$ is precisely $({G^\vee}_\der)^{\mb I_F,\circ}$, and $\sfq^\vee (\phi_\ad) (\Fr) \Fr^{-1}$
always lies in $({G^\vee}_\der)^{\mb I_F,\circ}$. The equivalence relation in $\Phi_\nr (G)$ still
allows for conjugation by elements of $G^{\vee,\mb I_F}$. That can change $\sfq^\vee (\phi_\ad) 
(\Fr) \Fr^{-1}$ by elements of $(1 - \Fr) (G^{\vee,\mb I_F})$.\\
(b) Let $\phi \in \sfq^\vee (\Phi_\nr (G_\ad))$ and let $\chi \in X_\Wr (G)$ such that 
$\chi \phi \in \sfq^\vee (\Phi_\nr (G_\ad))$. From the proof of part (a) we see that $\chi$ can 
be represented by an element $z \in {G_\der}^{\vee,\mb I_F,\circ} \cap Z(G^\vee)^{\mb I_F}$.
Then $z$ can be lifted to an element of $Z({G^\vee}_\Sc)^{\mb I_F}$, so $\chi$ lies in the
image of $X_\Wr (G_\ad) \to X_\Wr (G)$.

From \eqref{eq:5.2} we see that $Z({G^\vee}_\der)^{\mb I_F} \to \pi_0 
\big( ({G^\vee}_\der)^{\mb I_F} \big)$ is surjective. Hence 
$Z(G^\vee)^{\mb I_F} \to G^{\vee,\mb I_F} / ({G^\vee}_\der)^{\mb I_F,\circ}$ is surjective as 
well. It follows that every $\phi \in \Phi_\nr (G)$ can be written as an element of $X_\Wr (G)$ 
times an element of $\sfq^\vee (\Phi_\nr (G_\ad))$. Combine that with part (a).
\end{proof}

\emph{Proof of Theorem \ref{thm:5.1} for reductive $F$-groups with anisotropic centre}\\
We analyse \cite[\S 13]{FOS1} in detail. Let $\pi_\ad$ be an irreducible constituent
of $\mr{ind}_{P_{\mf f,\ad}}^{G_\ad} (\sigma_\ad)$. Let $(\phi_\ad,\rho_\ad)$ be the enhanced
L-parameter of $\pi_\ad$, via Theorem \ref{thm:5.1} for $G_\ad$. Let $\pi$ be an irreducible
constituent of $\sfq^* (\pi_\ad)$ and put $\phi = \sfq^\vee (\phi_\ad)$.

Write $g' = [\Omega_{G_\ad} / \Omega_{G_\ad,\mf f} : \Omega_{G,G_\ad} / \Omega_{G,G_\ad,\mf f}]$. 
It is checked in \cite[p. 29]{FOS1} that $\sfa'_\ad = |\Omega_{G_\ad,\mf f}|$, 
$\sfb' = \sfb'_\ad$ and
\[
\sfa' = |\Omega_{G,\mf f}| g' = |\ker (\Omega_G \to \Omega_{G_\ad})| \, 
|\Omega_{G_\ad,\mf f} : \Omega_{G,G_\ad,\mf f}]^{-1} \sfa'_\ad .
\]
As in \cite[(13.3)]{FOS1}, let $N_{\phi_\ad} \subset \Omega_{G_\ad,\mf f}$ be such that
\[
X_\Wr (G_\ad )_{\phi_\ad} = \Irr (\Omega_{G_\ad})_{\phi_\ad} = \Irr (\Omega_{G_\ad} / N_{\phi_\ad}) .
\]
Then $\sfa_\ad = |N_{\phi_\ad}|$. Taking the above into account, \cite[Lemma 13.1]{FOS1} generalizes
with almost the same proof. It says:

\begin{lem}\label{lem:5.5}
Suppose that $\phi_\ad \in \Phi_\nr^2 (G_\ad)$.
\enuma{
\item $X_\Wr (G) = \Irr (\Omega_G)$ acts transitively on the collection of elements  
$\phi' \in \Phi_\nr^2 (G)$ which, for every $F$-simple factor $\mc G_i$ of $\mc G$, have the 
same $\gamma$-factor (at $s=0$)
$\gamma (0, \mr{Ad}_{G_i^\vee)} \circ \phi',\psi)$ as $\phi = \sfq^\vee (\phi_\ad)$.
\item The stabilizer of $\phi \in \Phi_\nr^2 (G)$ in $X_\Wr (G)$ equals 
\[
\Irr \big( \Omega_{G,G_\ad} / (\Omega_{G,G_\ad} \cap N_{\phi_\ad}) \big) ,
\]
and it contains $\Irr (\Omega_{G,G_\ad} / \Omega_{G,G_\ad,\mf f})$.
\item $\sfa = | N_{\phi_\ad} \cap \Omega_{G,G_\ad} | \, |\ker (\Omega_G \to \Omega_{G_\ad}|$.
}
\end{lem}

The other arguments from \cite[\S 13]{FOS1} also generalize, the main difference is that we often
have to replace $\Omega_G$ by $\Omega_{G,G_\ad}$. In particular \cite[Lemma 13.2]{FOS1} becomes
\[
\mc A_\phi / \mc A_{\phi_\ad} \cong \Irr (\Omega_{G_\ad} / \Omega_{G,G_\ad} N_{\phi_\ad}) 
\]
and \cite[Lemma 13.4]{FOS1} becomes
\[
\sfb = g' [\Omega_{G_\ad,\mf f} : \Omega_{G,G_\ad} N_{\phi_\ad}]^{-1} \sfb_\ad .  
\]
We turn to \cite[\S 14]{FOS1}. With the above modifications to \cite[\S 13]{FOS1}, the proof in
\cite[\S 14]{FOS1} extends directly to possibly ramified reductive $F$-groups with anisotropic
centre. The enhanced L-parameters are constructed first for $G$-representations contained in 
$\sfq^* (\pi_\ad)$ for some $\pi_\ad \in \Irr_{\cusp,\unip}(G_\ad)$, and then extended 
$X_\Wr (G)$-equivariantly to the whole of $\Irr_{\cusp,\unip}(G)$ by means of Lemmas 
\ref{lem:5.3}.b and \ref{lem:5.4}.b. This establishes Theorem \ref{thm:5.1} for $\mc G$,
except part (e) and (when $\mc G$ is not semisimple) parts (b) and (c).

We note that for an anisotropic $F$-torus $\mc T$ the above parametrization agrees with the 
natural isomorphism
\begin{equation}\label{eq:5.4}
\Irr_\unip (T) = X_\Wr (T) \cong (T^{\mb I_F})_\Fr \cong \Phi_\nr (T),
\end{equation}
which is a special case of the LLC for tori. Since \eqref{eq:5.4} is natural, we obtain 
property (b) for all reductive $F$-groups with anisotropic centre. For property (c), the 
compatibility with almost direct products, we refer to the proof of \cite[Proposition 15.6]{FOS1}
in combination with \eqref{eq:5.4}.

Next we consider the proof of the HII conjecture for $\Irr_{\cusp,\unip}(G)$ in 
\cite[Lemmas 16.2 and 16.3]{FOS1}. This also goes by reduction to adjoint groups. 
Both formal degrees of $G$-representations and adjoint $\gamma$-factors of L-parameters for
$G$ are invariant under the action of $X_\Wr (G)$. In view of Lemmas \ref{lem:5.3}.b and 
\ref{lem:5.4}.b, this means that it suffices to check the HII conjectures for 
$\Irr (G / Z(G))_{\cusp,\unip}$ and $\sfq^\vee (\Phi_\nr (G_\ad))$. 

Let $Z(G)^\circ_1$ be the unique parahoric subgroup of $Z(\mc G)^\circ (F)$ and let 
$\overline{Z(\mc G)^\circ}(k_F)$ be its finite reductive quotient. We note that
$\overline{\mc G_{\mf f}^\circ}$ is isogenous to $\overline{\mc G_{\ad,\mf f}^\circ} \times
\overline{Z(\mc G)^\circ}$. By \cite[Proposition 1.4.12.c]{GeMa} these two groups have the
same number of $k_F$-points. That and \eqref{eq:2.10} lead to
\begin{multline}\label{eq:5.8}
\mr{vol}(P_{\mf f}) = |\overline{\mc G_{\mf f}^\circ}(k_F)|
q_F^{-(\mb a (\mf g^\vee) + \dim \overline{\mc G_{\mf f}^\circ} + \dim G^{\vee,\mb I_F}) / 2} = \\
|\overline{\mc G_{\ad,\mf f}^\circ}(k_F)| q_F^{-(\mb a (\mf g_\ad^\vee) + \dim \overline{\mc 
G_{\ad,\mf f}^\circ} + \dim G_\ad^{\vee,\mb I_F}) / 2} |\overline{Z(\mc G)^\circ}(k_F)|
q_F^{-(\mb a (Z(\mf g^\vee)) + \dim \overline{Z(\mc G)^\circ} + \dim Z(\mc G)^\circ) / 2} \\
= \mr{vol}(P_{\mf f,\ad}) \mr{vol}(Z(G)^\circ_1) .
\end{multline}
Write $\pi_\ad = \mr{ind}_{N_{G_\ad}(P_{\mf f,\ad})}^{G_\ad}(\sigma_\ad^N)$ and let
$\pi \in \Irr (G)_{\cusp,\unip}$ be a direct summand of $\sfq^* (\pi_\ad)$. Then \eqref{eq:2.8}
and \eqref{eq:5.8} yield
\begin{multline}\label{eq:5.9}
\frac{\mr{fdeg}(\pi, \mu_{G,\psi})}{\mr{fdeg}(\pi_\ad, \mu_{G_\ad,\psi})} = 
\frac{\dim (\sigma) |\overline{\mc G_{\ad,\mf f}^\circ}(k_F)| q_F^{-(\mb a (\mf g_\ad^\vee) + 
\dim \overline{\mc G_{\ad,\mf f}^\circ} + \dim G_\ad^{\vee,\mb I_F}) / 2} |\Omega_{G_\ad,\mf f}|}{
\dim (\sigma_\ad) |\overline{\mc G_{\mf f}^\circ}(k_F)| q_F^{-(\mb a (\mf g^\vee) + \dim 
\overline{\mc G_{\mf f}^\circ} + \dim G^{\vee,\mb I_F}) / 2} |\Omega_{G,\mf f}|} \\
= \frac{q_F^{(\mb a (Z(\mf g^\vee)) + \dim \overline{Z(\mc G)^\circ} + \dim Z(\mc G)^\circ) / 2} 
|\Omega_{G_\ad,\mf f}|}{|\Omega_{G,\mf f}| \, |\overline{Z(\mc G)^\circ}(k_F)|} =
\frac{|\Omega_{G_\ad,\mf f}|}{|\Omega_{G,\mf f}| \mr{vol}(Z(G)^\circ_1) } . 
\end{multline}
By construction $\sfq^\vee (\phi_{\pi_\ad}) = \phi_\pi$ and $\mc A_{\phi_{\pi_\ad}} \subset
\mc A_{\phi_\pi}$. Equations \cite[(16.7) and (16.8)]{FOS1} must be modified to
\begin{equation}\label{eq:5.10}
\frac{\dim (\rho_\pi)}{\dim (\rho_{\pi_\ad})} = \frac{|\Omega_{G_\ad, f}| \, 
| \Omega_{G,\mf f}^\circ \cap N_{\phi_{\pi_\ad}}|}{|\Omega_{G, \mf f}^\circ | \, 
| N_{\phi_{\pi_\ad}} |}
\end{equation}
while \cite[(16.9) and (16.10)]{FOS1} become
\begin{equation}\label{eq:5.6}
\frac{|S_{\phi_\pi}|}{|S_{\phi_\pi}^\sharp|} = 
\frac{|Z({G^\vee}_\Sc)^{\mb W_F}|}{|Z(G^\vee)^{\mb W_F}|} = 
\frac{|\Omega_{G_\ad}|}{|\Omega_G|} .
\end{equation}
Following \eqref{eq:5.6}, one obtains
\begin{equation}\label{eq:5.11}
\frac{|S_{\phi_\pi}^\sharp|}{|S^\sharp_{\rho_{\pi_\ad}}|} =
\frac{|\Omega_G|}{|\Omega_{G_\ad}|} [\Omega_{G_\ad} : \Omega_G^\circ N_{\phi_{\pi_\ad}}] =
|\ker (\Omega_G \to \Omega_{G_\ad})| \frac{|\Omega_G^\circ \cap N_{\phi_{\pi_\ad}}|}{
|N_{\phi_{\pi_\ad}}|} .
\end{equation}
By \cite[(16.14)]{FOS1}
\begin{equation}\label{eq:5.12}
\begin{aligned}
\gamma (s, \mr{Ad}_{G^\vee} \circ \phi_\pi, \psi) & = \gamma (s, \mr{Ad}_{G^\vee_\der} \circ 
\phi_\pi, \psi) \gamma (s, \mr{Ad}_{Z(G^\vee)^\circ} \circ \phi_\pi, \psi) \\
& = \gamma (s, \mr{Ad}_{{G_\ad}^\vee} \circ \phi_{\pi_\ad}, \psi) 
\gamma (s, \mr{Ad}_{G^\vee / {G^\vee}_\der} \circ \mr{id}_{\mb W_F}, \psi) .
\end{aligned} 
\end{equation}
We note that the formal degree of a unitary character of $Z(\mc G)^\circ (F)$ is
\[
\mr{vol}(Z(\mc G)^\circ (F))^{-1} = |\Omega_{Z(\mc G)^\circ (F)}|^{-1} 
\mr{vol}(Z(G)^\circ_1)^{-1} . 
\]
It was shown in \cite[Lemma 3.5 and Correction]{HII} that
\begin{equation}\label{eq:5.13}
|\gamma (0, \mr{Ad}_{G^\vee / {G^\vee}_\der} \circ \mr{id}_{\mb W_F}, \psi) | =
|\Omega_{Z(\mc G)^\circ (F)}| \, \mr{fdeg}(\mr{triv}_{Z(\mc G)^\circ (F)}) = 
\mr{vol}(Z(G)^\circ_1 )^{-1} .
\end{equation}
From \eqref{eq:5.10}--\eqref{eq:5.13} we deduce
\begin{multline}\label{eq:5.14}
\frac{\dim (\rho_\pi) | S^\sharp_{\phi_{\pi_\ad}} | \, |\gamma (0, \mr{Ad}_{G^\vee} \circ 
\phi_\pi, \psi)|}{\dim (\rho_{\pi_\ad}) | S^\sharp_{\phi_\pi} | \, |\gamma (0, 
\mr{Ad}_{{G_\ad}^\vee} \circ \phi_{\pi_\ad}, \psi)|} = \\
\frac{|\Omega_{G_\ad,\mf f}| \, |\Omega_{G,G_\ad,\mf f} \cap N_{\phi_{\pi_\ad}}|}{
|\Omega_{G,G_\ad,\mf f}| \, |\Omega_{G,G_\ad} \cap N_{\phi_{\pi_\ad}}| \,
|\ker (\Omega_G \to \Omega_{G_\ad})| \mr{vol}(Z(G)^\circ_1 )} .
\end{multline}
As $N_{\phi_{\pi_\ad}} \subset \Omega_{G_\ad,\mf f}$, this simplifies to
\begin{equation}\label{eq:5.15}
\frac{|\Omega_{G_\ad,\mf f}|}{|\Omega_{G,G_\ad,\mf f}| \, 
|\ker (\Omega_G \to \Omega_{G_\ad})| \mr{vol}(Z(G)^\circ_1 )} =
\frac{|\Omega_{G_\ad,\mf f}|}{|\Omega_{G,\mf f}| \mr{vol}(Z(G)^\circ_1 )} . 
\end{equation}
By \eqref{eq:5.9}, the expressions \eqref{eq:5.15} and \eqref{eq:5.14} also equal
fdeg$(\pi,\mu_{G,\psi}) / \mr{fdeg}(\pi_\ad, \mu_{G_\ad,\psi})$. From the already established
HII conjectures for $\Irr (G_\ad)_{\unip,\cusp}$ we know that
\[
\mr{fdeg}(\pi_\ad, \mu_{G_\ad,\psi}) = \dim (\rho_{\pi_\ad}) | S^\sharp_{\phi_{\pi_\ad}} |^{-1}  
|\gamma (0, \mr{Ad}_{{G_\ad}^\vee} \circ \phi_{\pi_\ad}, \psi)| .
\]
With \eqref{eq:5.9}, \eqref{eq:5.14} and \eqref{eq:5.15} we conclude the analogous
equality for $\Irr (G)_{\cusp,\unip}$ holds. \qed \\

In fact the above shows more, namely that \cite[Theorem 2.2]{FOS1} holds for all reductive
$F$-groups with anisotropic centre. This concerns precise statements about the numbers
$\sfa, \sfb, \sfa', \sfb'$, in terms of subquotients of $\Omega_G$. In contrast with
\cite[\S 13--14]{FOS1}, the formulation of \cite[Theorem 2.2]{FOS1} does not have to be
adjusted to accomodate for ramified groups, it generalizes exactly as written.\\

\emph{Proof of Theorem \ref{thm:5.1} for reductive groups}\\
This can be derived from the case of reductive $F$-groups with anisotropic centre,
see \cite[p. 38--41 and p. 44]{FOS1}. For these arguments it does not matter whether $\mc G$
is ramified or not. The only small difference is that in one of the steps on \cite[p. 41]{FOS1} 
we should not restrict from $(\mc G_\der Z(\mc G)_a)(F)$ to $G_\der$, with our proof for
reductive $F$-groups with anisotropic centre that step already works with  
$(\mc G_\der Z(\mc G)_a)(F)$.

For later use we recall the main idea of the proof. Let $Z(\mc G)_s$ be the maximal 
$F$-split torus in $Z(\mc G)$. Then $\mc G / Z(\mc G)_s$ has $F$-anisotropic centre and 
$(\mc G / Z(\mc G)_s)(F) = G / Z(G)_s$ \cite[(15.6)]{FOS1}. Tensoring with weakly
unramified characters yields a natural bijection \cite[(15.8)]{FOS1}
\begin{equation}\label{eq:5.16}
X_\Wr (G) \underset{X_\Wr (G / Z(G)_s)}{\times} \Irr (G / Z(G)_s )_{\cusp,\unip} 
\to \Irr (G)_{\cusp,\unip} .
\end{equation}
Similarly twisting by $X_\Wr (G) \cong (Z(G^\vee)^{\mb I_F})_\Fr$ provides a natural
bijection \cite[(15.12)]{FOS1}
\begin{equation}\label{eq:5.17}
X_\Wr (G) \underset{X_\Wr (G / Z(G)_s)}{\times} \Phi_{\nr,e}(G / Z(G)_s) \to
\Phi_{\nr,e}(G) .
\end{equation}
Combining \eqref{eq:5.16} and \eqref{eq:5.17} with Theorem \ref{thm:5.1} for $G/ Z(G)_s$,
one obtains the desired $X_\Wr (G)$-equivariant bijection 
$\Irr (G)_{\cusp,\unip} \to \Phi_{\nr,e}(G). \qed$

\section{A local Langlands correspondence}
\label{sec:LLC}

We want to generalize the results leading to a local Langlands correspondence for 
unipotent representations in \cite{SolLLC} from $F_\nr$-split to arbitrary connected 
reductive $F$-groups $\mc G$. In the non-supercuspidal case these results rely mainly on
\cite{AMS1,AMS2,AMS3}, in which no restriction on $\mc G$ is placed. Sections 1, 2 and 3 of 
\cite{SolLLC} were also written in that generality. 

In \cite[\S 4]{SolLLC} it is assumed that the reductive 
groups split over $F_\nr$, but that is only to apply the main result of \cite{FOS1}.
If we replace the input for \cite[Theorem 4.1 and Proposition 4.2]{SolLLC} by
Theorem \ref{thm:5.1}, they apply to ramified connected reductive $F$-groups as well.

In \cite[Lemma 4.4]{SolLLC} a Hecke algebra $\mc H (G, \hat P_{\mf f}, \hat \sigma)$
as in \eqref{eq:3.9} is compared with Hecke algebra $\mc H (\mf s^\vee, \vec{v})$
constructed from enhanced L-parameters for an absolutely simple adjoint group $G$.
When $\mc G$ is moreover ramified and $\mc G'$ is its $F_\nr$-split companion group,
we showed in Theorem \ref{thm:3.1} that Theorem \ref{thm:2.3}.b induces an algebra
isomorphism 
\[
\mc H (G, \hat P_{\mf f}, \hat \sigma) \cong \mc H (G', \hat P_{\mf f'}, \hat \sigma') .
\]
The algebra $\mc H (\mf s^\vee, \vec{v})$ (see \cite[\S 3.3]{AMS3} and \cite[\S 2]{SolLLC})
is constructed from the group $Z^1_{{G^\vee}_\Sc} (\phi (\mb I_F))$ with the data
$\phi (\Fr), u_\phi$ and $\rho$. Here $(\phi,\rho)$ comes from the Bernstein component in
$\Phi_e (G)$ associated to $\Irr (G)_{(\hat P_{\mf f},\hat \sigma)}$ by 
\cite[Proposition 4.2]{SolLLC}, so $\phi$ is unramified. As $\mc G$ is adjoint,
$Z^1_{{G^\vee}_\Sc} (\phi (\mb I_F)) = (G^\vee)^{\mb I_F}$. In view of the comparison
results Lemma \ref{lem:4.1} and Proposition \ref{prop:4.4}, the data underlying 
$\mc H (\mf s^{'\vee}, \vec{z})$ are canonically isomorphic to those for 
$\mc H (\mf s^\vee,\vec{z})$. Hence $\mc H (\mf s^{\vee},\vec{v})$ is canonically isomorphic 
to the Hecke algebra $\mc H(\mf{s'}^\vee, \vec{v})$ constructed in the same way for $G'$. 

Recall from \eqref{eq:4.11} that the transfer between enhanced unramified L-parameters for 
$G$ and $G'$ reflects the transfer between cuspidal unipotent representations in Theorem 
\ref{thm:2.3}. The group $G'$ was already treated in \cite[Lemma 4.5]{SolLLC} and
\cite{LusUni1,LusUni2}. In this way we obtain algebra isomorphisms
\begin{equation}\label{eq:6.1}
\mc H (G, \hat P_{\mf f}, \hat \sigma) \cong \mc H (G', \hat P_{\mf f'}, \hat \sigma') \cong
\mc H(\mf{s'}^\vee, \vec{v}) \cong \mc H(\mf s^\vee, \vec{v}) .
\end{equation}
This means that \cite[Lemma 4.5]{SolLLC} holds for ramified $F$-groups. With that the entire
Section 4 of \cite{SolLLC} works for arbitrary connected reductive $F$-groups.
Now \cite[Theorem 5.1]{SolLLC} gives:

\begin{thm}\label{thm:6.4}
There exists a bijective local Langlands correspondence
\begin{equation}
\begin{array}{ccc}
\Irr (G)_\unip & \longleftrightarrow & \Phi_{\nr,e}(G) \\
\pi & \mapsto & (\phi_\pi, \rho_\pi) \\
\pi (\phi,\rho) & \text{\reflectbox{$\mapsto$}} & (\phi,\rho)
\end{array}
\end{equation}
\end{thm}

In \cite[\S 5]{SolLLC} several properties of Theorem \ref{thm:6.4} were checked. These arguments 
generalize readily to possibly ramified connected reductive $F$-groups, if we take the following 
into account for the cases with $\mc G$ simple:
\begin{itemize}
\item For the $X_\Wr (G)$-equivariance from \cite[Lemma 5.3]{SolLLC} we use Lemma \ref{lem:3.7}
and Proposition \ref{prop:4.4}.
\item For the cuspidality and the compatibility with cuspidal supports from \cite[Lemmas 5.4
and 5.5]{SolLLC} we use Lemmas \ref{lem:3.4} and \ref{lem:4.7}.
\item For temperedness and boundedness in \cite[Lemma 5.6]{SolLLC} we use Lemma \ref{lem:3.3}.a
and Proposition \ref{prop:4.4}.
\item For square-integrability and discreteness in \cite[Lemma 5.7]{SolLLC} we use Lemma 
\ref{lem:3.3}.b and Proposition \ref{prop:4.4}.
\item For the considerations with parabolic induction in \cite[Lemmas 5.9 and 5.10]{SolLLC}
we use Lemma \ref{lem:3.4} and \eqref{eq:3.8}.
\end{itemize}
The central characters associated to both sides of Theorem \ref{thm:6.4}, as discussed in 
\cite[Lemma 5.8]{SolLLC}, need more attention.
Recall from \cite[p. 20--23]{Lan} and \cite[\S 10.1]{Bor} that every $\phi \in \Phi (G)$
determines a character $\chi_\phi$ of $Z(G)$. For the construction, one first
embeds $\mc G$ in a connected reductive $F$-group $\overline{\mc G}$ with
$\mc G_\der = \overline{\mc G}_\der$, such that $Z(\overline{\mc G})$
is connected. Then one lifts $\phi$ to a L-parameter $\overline \phi$ for
$\overline G = \overline{\mc G}(F)$. The natural projection ${}^L \overline G \to
{}^L Z(\overline G)$ produces an L-parameter $\overline{\phi}_z$ for 
$Z(\overline G) = Z(\overline{\mc G})(F)$, and via the local Langlands correspondence
for tori $\overline{\phi}_z$ determines a character $\chi_{\overline \phi}$ of 
$Z (\overline G)$. Then $\chi_\phi$ is given by restricting $\chi_{\overline \phi}$
to $Z(G)$. Langlands \cite[p. 23]{Lan} checked that $\chi_\phi$ does not depend on
the choices made above.

\begin{lem}\label{lem:6.1}
In Theorem \ref{thm:6.4} the central character of $\pi$ equals $\chi_{\phi_\pi}$. 
\end{lem}
\begin{proof}
By construction $G^\vee$ is the quotient of $\overline{G}^\vee$ by a central subgroup. 
Then ${\overline{G}^\vee}_\der$ projects onto ${G^\vee}_\der$.

In the cuspidal support $(M,\phi_M,\rho_M) := \mb{Sc}(\phi_\pi,\phi_\pi)$, the difference
between $\phi_\pi$ and $\phi_M$ lies entirely in ${G^\vee}_\der$. Hence $\phi_\pi$ and
$\phi_M$ give the same map $\mb W_F \to (G^\vee / {G^\vee}_\der) \rtimes \mb W_F$.
Then their lifts $\overline{\phi_\pi}$ and $\overline{\phi_c}$ project to the same map
\[
\mb W_F \to (\overline{G}^\vee / {\overline{G}^\vee}_\der ) \rtimes \mb W_F = 
{}^L Z(\overline G) .
\]
Consequently $\phi_\pi$ and $\phi_c$ determine the same character of $Z(\overline G)$,
and $\chi_{\phi_\pi} = \chi_{\phi_c} |_{Z(G)}$.

Similarly, the central character of $\pi$ equals that of its supercuspidal support
(restricted to $Z(G)$). Together with \cite[Lemma 5.5]{SolLLC} (generalized above to
possibly ramified $F$-groups), this means that it suffices to consider the case where 
$\pi$ is supercuspidal and $(\phi_\pi,\rho_\pi)$ is cuspidal.

We specialize further to the case where $Z(\mc G)^\circ$ is $F$-anisotropic. 
When $\pi$ is contained in $\sfq^* (\pi_\ad)$ for some $\pi_\ad \in 
\Irr (G_\ad)_{\cusp,\unip}$, its central character is obviously trivial. By the 
construction in Section \ref{sec:supercusp}, $\phi_\pi = \sfq^\vee (\pi_\ad)$. Hence 
$\phi_\pi (\mb W_F) \subset {G^\vee}_\der \rtimes \mb W_F$ and $\overline{\phi_\pi}_z$
is the trivial parameter $\mr{id}_{\mb W_F}$ for $Z(\overline G)$. Then $\chi_{\phi_\pi} 
= \mr{triv}_{Z(G)}$, as required.

Other $\tilde \pi \in \Irr (G)_{\cusp,\unip}$ are obtained from such a $\pi$ by tensoring
with a suitable $\chi \in X_\Wr (G)$, see Lemma \ref{lem:5.3}. This is mimicked in 
Lemma \ref{lem:5.4}, and $\phi_{\tilde \pi} = \phi_{\chi \pi} = \chi \phi_\pi$. Then the 
central character of $\tilde \pi$ is $\chi$ and $\overline{\phi_{\tilde \pi}}_z = 
\overline{\chi}_z \overline{\phi_\pi}_z = \overline{\chi}_z$,
so $\chi_{\phi_{\tilde \pi}}$ is also $\chi$.

Finally we consider the case where $\pi$ is supercuspidal and $Z(\mc G)^\circ$ is 
$F$-isotropic. Since $\mc G / Z(\mc G)_s$ has $F$-anisotropic centre, we already know
the claim for $G / Z(G)_s$. But with \eqref{eq:5.16} and \eqref{eq:5.17} the LLC for 
$\Irr (G)_{\cusp,\unip}$ is deduced from its analogue for $G / Z(G)_s$ by twisting with
$X_\Wr (G)$ on both sides of the correspondence. Explicitly, every $\pi \in \Irr 
(G)_{\cusp,\unip}$ can be written as $\chi \otimes \tilde \pi$ with $\tilde \pi \in \Irr
(G / Z(G)_s)_{\cusp,\unip}$, and then $\phi_\pi = \chi \phi_{\tilde \pi}$. By the lemma for 
$G / Z(G)_s$, the central character of $\pi$ equals $\chi \otimes \chi_{\phi_{\tilde \pi}}$.
On the other hand 
\[
\overline{\phi_\pi}_z = \overline{\chi \phi_{\tilde \pi}}_z = \overline{\chi}_z 
\overline{\phi_{\tilde \pi}}_z ,
\]
so $\chi_{\phi_\pi} = \chi \otimes \chi_{\phi_{\tilde \pi}}$ as well.
\end{proof}

Summarising: we generalized the entire paper \cite{SolLLC} from $F_\nr$-split to arbitrary
connected reductive $F$-groups. In particular we may now use its main result 
\cite[Theorem 1]{SolLLC} in that generality.

Next we investigate the functoriality of Theorem \ref{thm:6.4}, as in \cite{SolFunct}.
The larger part of that paper (namely Sections 1--5) is written in complete generality,
for all connected reductive groups. Only \cite[\S 7]{SolFunct} deals exclusively with
unipotent representations. There it is assumed that the groups are $F_\nr$-split,
following \cite{FOS1,SolLLC}. 

Fortunately all the arguments from \cite[\S 7]{SolFunct} are also valid for ramified
groups. There are only two small points to note:
\begin{itemize}
\item In the proof of Lemma \cite[Lemma 7.1]{SolFunct} for possibly ramified connected
reductive $F$-groups, we must omit the reduction step from $G$ (with $F$-anisotropic 
centre) to $G_\der$. With our proof of Theorem \ref{thm:5.1} for reductive groups with
anisotropic centre, the arguments for \cite[Lemma 7.1]{SolFunct} apply directly.
\item In \cite[(7.21)]{SolFunct} it is claimed that 
\[
\sfq : \hat P_{\mf f} / P_{\mf f} \to \hat P_{\mf f,\ad} / P_{\mf f,\ad}
\]
is injective, which need not be true when $\mc G$ is ramified. To overcome that, we can
take $\hat \sigma \in \Irr (\hat P_{\mf f})_\cusp$ of the form
\[
\chi \otimes \sfq^* (\hat \sigma_\ad) \text{ with } \hat \sigma_\ad \in \Irr (P_{\mf f,\ad})_\cusp
\text{ and } \chi \in \Irr (\Omega_{G,\mf f}) \cong \Irr (\hat P_{\mf f} / P_{\mf f} ),
\]
as in Lemma \ref{lem:5.3}. For $g \in \hat P_{\mf f,\ad}$ and $p \in \hat P_{\mf f}$ we 
have Ad$(g)^* \chi = \chi$ because $\Omega_{G,\mf f}$ and $\Omega_{G_\ad,\mf f}$ are abelian.
The equation following \cite[(7.21)]{SolFunct} becomes
\begin{multline*}
\hspace{2cm} \mr{Ad}(g)^* (\hat \sigma)(p) = \chi (p) \big( \mr{Ad}(g)^* 
\sfq^* (\sigma_\ad) \big)(p) = \\
\chi (p) \hat{\sigma_\ad} (g) \hat{\sigma_\ad} (\sfq (p)) \hat{\sigma_\ad} (g^{-1}) =
\hat{\sigma_\ad} (g) \hat \sigma (p) \hat{\sigma_\ad} (g^{-1}) \qquad p \in \hat P_{\mf f} .
\end{multline*}
With that, the proof of \cite[Lemma 7.5.b]{SolFunct} works fine.
\end{itemize}

This means that the results of \cite{SolFunct} hold for unipotent representations of any
connected reductive $F$-group. To formulate this precisely, let $\eta : \tilde{\mc G} \to
\mc G$ be a homomorphism of connected reductive $F$-groups such that
\begin{itemize}
\item the kernel of d$\eta : \mr{Lie}(\tilde{\mc G}) \to \mr{Lie}(\mc G)$ is central,
\item the cokernel of $\eta$ is a commutative $F$-group.
\end{itemize}
Let ${}^L = \eta^\vee \rtimes \mr{id} : G^\vee \rtimes \mb W_F \to \tilde G^\vee \rtimes \mb W_F$
be a L-homomorphism dual to $\eta$. For $\eta \in \Phi (G)$ we get ${}^L \eta \circ \phi 
\in \Phi (\tilde G)$. Then $\eta$ gives rise to an injective algebra homomorphism
\begin{equation}\label{eq:6.3}
{}^S \eta : \C [\mc S_\phi] \to \C [\mc S_{{}^L \eta \circ \phi}] , 
\end{equation}
which under mild assumptions is canonical. It is a twist of the injection 
${}^L \eta : \mc S_\phi \to \mc S_{{}^L \eta \circ \phi}$
by a character of $\mc S_\phi$, see \cite[Proposition 5.4]{SolFunct}.

Decomposing $\eta$ as in \cite[(5.2)]{SolFunct}, we see that the pullback $\eta^*$ sends 
unipotent $G$-representations to unipotent $\tilde G$-representations and that ${}^L \eta$
maps $\Phi_\nr (G)$ to $\Phi_\nr (\tilde G)$. Then \cite[Conjecture 2 and Theorem 3]{SolFunct},
applied with the LLC from Theorem \ref{thm:6.4} say: 

\begin{thm}\label{thm:6.2}
For any $(\phi,\rho) \in \Phi_{\nr,e}(G)$: 
\[
\eta^* (\pi (\phi,\rho)) = \bigoplus\nolimits_{\tilde \rho \in \Irr (\mc S_{{}^L \eta \circ \phi})} 
\mr{Hom}_{\mc S_\phi} \big( \rho, {}^S \eta^* (\tilde \rho) \big) 
\otimes \pi ({}^L \eta \circ \phi, \tilde \rho) .
\] 
\end{thm}

Finally we come to Conjecture \ref{conj:HII} by Hiraga, Ichino and Ikeda \cite{HII}. To prove it
we will generalize the arguments from \cite{FOS2}, which was designed for $F_\nr$-split groups.

\begin{prop}\label{prop:6.3}
\enuma{
\item The LLC from Theorem \ref{thm:6.4} satisfies the HII conjecture \ref{conj:HII}, 
up to some rational constants that depend only on an orbit $\mc O$.
\item Part (a), Lemma \ref{lem:6.1}, Theorem \ref{thm:6.2} and compatibility with direct products 
of reductive groups determine this LLC uniquely, up to twists by $X_\Wr (G_\ad,G)$.
}
\end{prop}
\begin{proof}
Part (a) is shown in \cite[Theorem 4.5.1]{Opd18}, for the "Langlands parametrization" from
that paper. We proved it for ramified simple groups in Theorem \ref{thm:4.6}.e, which in
combination with \cite[\S 4.5]{Opd18} gives part (a) for all adjoint $F$-groups. The proof
in the case of $F_\nr$-split groups with anisotropic centre in \cite{Opd18} proceeds via
reduction to adjoint groups. It relies on spectral transfer morphisms for affine Hecke algebras
\cite{Opd2}.
We showed in Theorem \ref{thm:3.1} that the Hecke algebras for ramified adjoint $F$-groups
have exactly the same shape and the same parameters as those for suitable $F_\nr$-split
adjoint $F$-groups, so that Opdam's arguments with spectral transfer morphisms remain valid.
This means that the Langlands parametrization from \cite{Opd18} can be constructed for all
connected reductive $F$-groups, and that it satisfies Conjecture \ref{conj:HII} up to constants.

Our LLC from Theorem \ref{thm:6.4} extends \cite{SolLLC} to possibly ramified groups. 
In \cite[Theorem 2.1]{FOS2} it is checked that the LLC from \cite{SolLLC} agrees with the 
Langlands parametrization from \cite{Opd18} (in the sense that the latter can be obtained
from the former by forgetting the enhancements of L-parameters). We need to extend this
compatibility to Theorem \ref{thm:6.4} and the above generalization of Opdam's Langlands
parametrization. 

By Lemmas \ref{lem:5.3} and \ref{lem:5.4}, it suffices to do so for 
L-parameters $\phi \in \sfq^\vee (\Phi_\nr (G_\ad))$ and for unipotent $G$-representations
with trivial central character. For those we saw in the proof of Lemma \ref{lem:5.4} that
$\phi (\Fr) \in G^{\vee,\mb I_F,\circ}$. By  \cite[Lemma 6.4]{Bor} that element corresponds
to a unique $W(G^{\vee,\mb I_F,\circ}, T^{\vee,\mb I_F,\circ})^\Fr$-orbit in 
$(T^{\vee,\mb I_F,\circ})_\Fr$. With this modification in mind, the proof of 
\cite[Theorem 2.1]{FOS2} works for such L-parameters and $G$-representations. The first
part of that proof establishes part (a) of the current proposition, while the last part deals
with the essential uniqueness asserted in part (b).
\end{proof}

With Proposition \ref{prop:6.3}, everything in \cite[Sections 1--4]{FOS2} works equally
well for ramified $F$-groups. Recall that we proved the HII conjecture for 
square-integrable representations of ramified simple $F$-groups in Theorem \ref{thm:4.6}.e.
Together with \cite[\S 5.1]{FOS2} that establishes Conjecture \ref{conj:HII} for 
square-integrable representations of adjoint $F$-groups. 

\begin{lem}\label{lem:6.5}
Let $\mc G$ be semisimple and let $\delta \in \Irr_\unip (G)$ be square-integrable. 
Then there exists a
$\chi \in X_\Wr (G)$ and a square-integrable $\delta_\ad \in \Irr_\unip (G_\ad)$ such
that $\chi \otimes \delta$ is a constituent of the pullback $\eta^* (\delta_\ad)$.
\end{lem}
\begin{proof}
Let $\mc M$ be a $F$-Levi subgroup of $\mc G$ and let $\mc M_{AD} = \mc M / Z(\mc G)$ 
be the image of $\mc M$ in $\mc G_\ad$. In Lemma \ref{lem:5.3} we showed that for 
every $\pi_M \in \Irr_{\cusp,\unip}(M)$ there exist
$\chi_M \in X_\Wr (M)$ and $\pi_{M_{AD}} \in \Irr_{\cusp,\unip}(M_{AD})$ such that 
$\chi_M \otimes \pi_M$ is a constituent of the pullback $\eta_M^* (\pi_{M_{AD}})$.
We recall from \cite[(37)]{FOS2} that parabolic induction is compatible with pullback 
from $G_\ad$ (resp. $M_{AD}$). This implies that every $\pi \in \Irr_\unip (G)$ is,
up to twisting by a $\chi \in X_\Wr (G)$, contained in the pullback
of a $\pi_\ad \in \Irr (G_\ad)_\unip$. 

Since $G$ is semisimple, $\chi$ is automatically unitary and tensoring by $\chi$
preserves square-integrability. We can regard $\pi \otimes \chi$ as a representation
of the cocompact subgroup $G/Z(G)$ of $G_\ad$. Then a small variation on 
\cite[Proposition 2.7]{Tad} says that $\pi \otimes \chi$ is square-integrable 
if and only if $\pi_\ad$ is square-integrable.
\end{proof}

We note that the tensoring with the unitary character $\chi$ in Lemma \ref{lem:6.5} 
does not change the Plancherel densities. Therefore $\chi$ may be ignored in the subsequent
computations of formal degrees.
With Lemma \ref{lem:6.5} at hand, the proofs in \cite[\S 5.2]{FOS2} apply to all connected
reductive $F$-groups with anisotropic centre, if we make the following modifications:
\begin{itemize}
\item By \cite[(34)]{FOS2} we must multiply the right hand side of \cite[(47)]{FOS2} with 
vol$(Z(G)^\circ_1)^{-1}$. (This factor is invisible in the semisimple setting of 
\cite[\S 5.2]{FOS2}.) For consistency, we must multiply \cite[(48)]{FOS2} and the right 
hand sides of \cite[Theorem 5.4.a and (51)]{FOS2} with the same factor. 
\item Using \eqref{eq:5.13} we replace the equality   
\[
\gamma (s,\mr{Ad}_{G^\vee} \circ \phi_\delta, \psi) =
\gamma (s,\mr{Ad}_{{G^\vee}_\Sc} \circ \phi_{\delta_\ad}, \psi)
\]
in the proof of \cite[Theorem 5.4.b]{FOS2} by
\[
\left| \frac{\gamma (0,\mr{Ad}_{G^\vee} \circ \phi_\delta, \psi)}{\gamma 
(0,\mr{Ad}_{{G^\vee}_\Sc} \circ \phi_{\delta_\ad}, \psi)} \right| =
| \gamma (0,\mr{Ad}_{Z(G^\vee)^\circ} \circ \mr{id}_{\mb W_F}, \psi)| = 
\mr{vol}(Z(G)^\circ_1)^{-1} .
\]
Then the next line of that proof becomes
\[
\frac{\mr{fdeg} (\delta)}{\mr{fdeg} (\delta_\ad)} = \frac{\dim (\rho_\delta) \, 
\big| S_{\phi_{\delta_\ad}}^\sharp \big| \,|\gamma (0,\mr{Ad}_{G^\vee} \circ \phi_\delta, \psi)|}{
\dim (\rho_{\delta_\ad}) \, \big| S_{\phi_\delta}^\sharp \big| \,|\gamma 
(0,\mr{Ad}_{{G^\vee}_\Sc} \circ \phi_{\delta_\ad}, \psi)|}  
\]
Combining that with the adjoint case yields the HII 
conjecture for formal degrees of square-integrable representations of connected reductive 
$F$-groups with anisotropic centre.
\end{itemize}
That renders most of \cite[\S 5.3]{FOS2} superfluous, except for the last part of the 
proof of \cite[Theorem 5.6]{FOS2}. That achieves the generalization (from $F$-anisotropic 
centre) to square-integrable modulo centre representations of arbitrary connected 
reductive $F$-groups.

We move on to Plancherel densities for tempered unipotent representation of possibly 
ramified $F$-groups. Some statements in \cite[\S 6.2]{FOS2} need to be modified:
\begin{itemize}
\item The adjoint $\gamma$-factors no longer need to be real-valued, as in 
\cite[Lemma A.5]{FOS2}, because of $\epsilon$-factors of ramified $\mb W_F$-subrepresentations
of Lie$(G^\vee)$. To compensate for that, one can include fourth roots of unity as in
Lemma \ref{lem:4.5}, or one can replace $\pm \gamma (0,\mr{Ad} \circ \phi,\psi)$
everywhere by $|\gamma (0,\mr{Ad} \circ \phi,\psi)|$. 
\item In view of \eqref{eq:2.9}, \cite[(67)]{FOS2} becomes
\[
\frac{\tau (N_e)}{\tau_{\mc H^M} (N_e)} = \frac{\mr{vol}(\hat P_{\mf f,M})}{\mr{vol}
(\hat P_{\mf f})} = \frac{|\overline{\mc M_{\mf f}}(k_F)| \, q_F^{(\dim \overline{\mc G_{\mf f}^\circ} 
+ \dim \mc G + \mb a (\mr{Lie} \, G^\vee ))/2}}{|\overline{\mc G_{\mf f}}(k_F)| \, 
q_F^{(\dim \overline{\mc M_{\mf f}^\circ} + \dim \mc M + \mb a (\mr{Lie} \, M^\vee ))/2}} . 
\]
This entails that in \cite[Lemma 6.4]{FOS2} one also gets an extra factor\\
$q_F^{(\mb a (\mr{Lie} \, G^\vee) - \mb a (\mr{Lie} \, M^\vee )) / 2}$.
\item The computation of adjoint $\gamma$-factors in \cite[Appendix A.2]{FOS2} applies only
to the $\mb I_F$-fixed points in the involved complex Lie algebras. With Lemma \ref{lem:4.3}
we can obtain similar formulas based on Lie$(G^\vee)$ and Lie$(M^\vee)$. It follows that
\cite[(69)]{FOS2} must be replaced by
\[
\gamma (0, \mr{Ad}_{G^\vee,M^\vee} \circ t \phi_M, \psi) = \epsilon \gamma (0,\mr{Ad}_{M^\vee} 
\circ t \phi_M,\psi) m^{M^*} (t r_M) \frac{q_F^{(\dim \mc G + \mb a (\mr{Lie} \, G^\vee ))/2}}{
q_F^{(\dim \mc M + \mb a (\mr{Lie} \, M^\vee ))/2}} ,
\]
where $\epsilon^2 \in \{\pm 1\}$ depends only on the $\mb W_F$-representation
$\mr{Lie}(G^\vee) / \mr{Lie} (M^\vee)$.
\end{itemize}
With these adjustments \cite[\S 6]{FOS2} becomes valid for all connected reductive $F$-groups.
In particular \cite[Theorem 6.5]{FOS2} then establishes the HII conjecture for all tempered
irreducible unipotent $G$-representations.

\section{Rigid inner twists}
\label{sec:rigid}

So far we adhered to the conventions of Arthur \cite{Art,ABPS} for the setup with inner forms, 
components groups of L-parameters and relevance of enhancements. In this paragraph we take
a different point of view, that of rigid inner twists. This notion was developed for reductive
groups over local fields of characteristic zero by Kaletha \cite{Kal1}. Recently Dillery
extended it to reductive groups over local function fields \cite{Dil}.

The main point is to replace $H^1 ( \mb W_F, \mc G_\ad)$, which parametrizes inner twists of
$\mc G$, by a new cohomology set $H^1 (\mc E, \mc Z \to \mc G)$. Here $\mc Z$ is a fixed finite
central $F$-subgroup of $\mc G$. When char$(F) = 0$, this is based on a canonical extension
of topological groups
\[
1 \to u \to W \to \mr{Gal}(F_s / F) \to 1 .
\]
Let $Z^1(\mc E,\mc Z \to \mc G)$ be the set of those continuous cocycles of $W$ in $\mc G (F_s)$,
whose restriction to $u$ is a homomorphism $u \to \mc Z (F_s)$. In \cite[\S 3.2]{Kal1}, 
$H^1 (\mc E, \mc Z \to \mc G)$ is defined as a subset of $H^1 (W,\mc G)$, namely the image
of $Z^1 (\mc E,\mc Z \to \mc G)$. The construction of $H^1 (\mc E,\mc Z \to \mc G)$ is similar, 
but considerably more involved, when char$(F) > 0$ \cite[\S 3.2]{Dil}. 

For our purposes it is best to take $\mc Z = Z (\mc G_\der)$, as suggested in \cite{Kal1}.
We write
\[
\overline{\mc G} = \mc G / Z(\mc G_\der) = \mc G_\der / Z(\mc G_\der) \times
Z (\mc G ) / Z(\mc G_\der) = \mc G_\ad \times Z(\overline{\mc G}) .
\]
The canonical homomorphisms $\mc G \to \overline{\mc G} \to \mc G_\ad$ induce natural maps
\begin{equation}\label{eq:7.1}
H^1 (\mc E, Z(\mc G_\der) \to \mc G) \to H^1 (\mb W_F, \overline{\mc G}) \to H^1 (\mb W_F, \mc G_\ad) 
\end{equation}
which are surjective \cite[Proposition 5.12]{Dil}. 

Recall that an inner twist $(\mc G',\psi)$ of $\mc G$ is determined by an isomorphism of
$F_s$-groups $\psi : \mc G \to \mc G'$, such that for every $\gamma \in \mr{Gal}(F_s /F)$ the 
$F$-automorphism $\psi^{-1} \circ \gamma \circ \psi \circ \gamma^{-1}$ of $\mc G$ is inner. 
Via \eqref{eq:7.1} every $z \in H^1 (\mc E, Z(\mc G_\der) \to \mc G)$ determines an element of
$H^1 (\mb W_F, \mc G_\ad)$ and hence a unique equivalence class of inner twists $\mc G$.
By the surjectivity of \eqref{eq:7.1} every inner twist of $\mc G$ (up to equivalence) can be 
obtained in this way. 

By definition \cite[\S 5.1]{Kal1} a rigid inner twist of $\mc G$ is a triple $(\mc G',\psi,z)$
where $(\mc G',\psi)$ is an inner twist of $\mc G$ and $z \in Z^1 (\mc E,Z(\mc G_\der) \to \mc G)$
such that
\[
\psi^{-1} \circ \gamma \circ \psi \circ \gamma^{-1} = \mr{Ad}(z(\gamma)) 
\text{ for all } \gamma \in \mr{Gal}(F_s / F) .
\]
This applies to local fields of characteristic zero, and it probably it works well for most 
connected reductive groups over local functions field also. Nevertheless, the actual definition 
of rigid inner twistsis more complicated when $F$ has positive characteristic \cite[\S 7.1]{Dil}. 
In any case, $H^1 (\mc E, Z(\mc G_\der) \to \mc G)$ can be regarded as the set of equivalence 
classes of rigid inner twists of $\mc G$. The big advantage of this setup is that it allows 
canonical transfer factors \cite[\S 6]{Dil}.

Let $\overline{\mc G}^\vee = {\mc G^\vee}_\Sc \times Z (\mc G / Z(\mc G_\der) )^\vee$ be the dual 
group of $\overline{\mc G}$ and write
\begin{equation}\label{eq:7.5}
\overline{G}^\vee = \overline{\mc G}^\vee (\C) = {G^\vee}_\Sc \times Z (\overline{G}^\vee )^\circ ,
\end{equation}
a cover of $G^\vee$. The $\mb W_F$-stable pinning of $G^\vee$ can be lifted to a pinning of 
$\overline{G}^\vee$ and we use the latter to define an action of $\mb W_F$ on $\overline{G}^\vee$. 
That furnishes a surjection ${}^L \overline{G} \to {}^L G$. Let $Z (\overline{G}^\vee )^+$ 
be the preimage of $Z(G^\vee )^{\mb W_F}$ in $Z( \overline{G}^\vee)$.
In \cite[Corollary 7.11]{Dil} a natural isomorphism
\begin{equation}\label{eq:7.2}
H^1 (\mc E, Z(\mc G_\der) \to \mc G) \cong \Irr \big( \pi_0 \big( Z(\overline{G}^\vee)^+ \big) \big) 
\end{equation}
was established. Via \eqref{eq:7.1} it extends the Kottwitz isomorphism
\begin{equation}\label{eq:7.3}
H^1 (\mb W_F, \mc G_\ad) \cong \Irr \big( Z ({G^\vee}_\Sc)^{\mb W_F} \big)
\end{equation}
from \cite{Tha}. This shows that the fibers of \eqref{eq:7.1} carry simply transitive actions of
\begin{equation}\label{eq:7.8}
\Irr \big( \pi_0 \big( Z(\overline{G}^\vee )^+ \big) / Z( {G^\vee}_\Sc )^{\mb W_F} \big) .
\end{equation}
In particular the number of rigid inner twists lying over a given inner twist (both considered up
to equivalence) is finite and equals 
$\big[ \pi_0 \big( Z(\overline{G}^\vee )^+ \big) : Z( {G^\vee}_\Sc )^{\mb W_F} \big]$. 
From \eqref{eq:7.5} we see that the group \eqref{eq:7.8} is trivial when $\mc G$
is an inner form of a split group. However, when $\mc G$ is an outer form of a split group (e.g.
a non-split torus), \eqref{eq:7.8} can very well be nontrivial. In such cases the quasi-split
rigid inner twist of a reductive group is not unique anymore -- 
apparently a price one has to pay for canonical transfer factors. 

For issues involving parabolic induction we need to understand how \eqref{eq:7.2} is related to
its versions for Levi subgroups. To this end we assume that $G$ is quasi-split. Let $(\mc G^z,\psi,z)$
be a rigid inner twist of $\mc G$ and let $\mc L^z$ be a Levi $F$-subgroup of $\mc G^z$. Any 1-cocycle
used to construct $\mc L^z$ as inner twist of a Levi $F$-subgroup $\mc L$ of $\mc G$ can also be
used for $\mc G^z$. Therefore we may assume (upon replacing $z$ by an equivalent element) that 
$z \in Z^1 (\mc E,Z(\mc G_\der) \to \mc L)$ and that $\psi$ restricts to an $F_s$-isomorphism 
$\mc L \to \mc L^z$. 

Consider the Levi subgroup $\tilde{\mc L} = \mc L / Z (\mc G_\der)$ of $\overline{\mc G}$ and the
inclusion
\begin{equation}\label{eq:7.11}
\mc L / Z(\mc G_\der) = \tilde{\mc L} \to \overline{\mc G} .
\end{equation}
We note that in general there no canonical homomorphism between $\tilde{\mc L}$ and $\overline{\mc L}$.

\begin{lem}\label{lem:7.4}
\enuma{
\item The map \eqref{eq:7.11} induces a surjection 
$\pi_0 \big( Z(\overline{G}^\vee)^+ \big) \to \pi_0 \big( Z(\tilde{L}^\vee)^+ \big)$.
\item Let $\chi_z \in \Irr ( Z(\overline{G}^\vee)^+ )$ be associated to $z \in H^1 (\mc E,Z(\mc G_\der)
\to \mc G)$ via \eqref{eq:7.2} and let $\chi_{z,\mc L}$ be the character of $Z(\tilde{L}^\vee )^+$ 
associated to $z \in H^1 (\mc E,Z(\mc G_\der) \to \mc L)$ via \eqref{eq:7.2}. Then $\chi_z$ is the
pullback of $\chi_{z,\mc L}$ along the map from part (a). 
}
\end{lem}
\begin{proof}
(a) Since $\tilde{\mc L}$ is a Levi subgroup of $\overline{\mc G} = \mc G_\ad \times Z(\overline{\mc G})$,
we can express it as $\mc L_{AD} \times Z(\overline{\mc G})$, where $\mc L_{AD}$ is a Levi subgroup of
$\mc G_\ad$. We see that $\tilde{L}^\vee = L^\vee_c \times Z(\overline{G}^\vee )^\circ$, where $L^\vee_c$
is the preimage of $L^\vee$ in ${G^\vee}_\Sc$. By \cite[Lemma 1.1]{Art1}
\[
Z(G^\vee )^{\mb W_F} Z(L^\vee )^{\mb W_F,\circ} = Z(L^\vee )^{\mb W_F}.
\]
Taking preimages in $\overline{G}^\vee$, we find 
\[
Z(\overline{G}^\vee )^+ Z(\tilde{L}^\vee )^{\mb W_F,\circ} = Z(\tilde{L}^\vee )^+.
\]
Now it is clear that $\pi_0 \big( Z(\overline{G}^\vee)^+ \big)$ surjects onto
$\pi_0 \big( Z(\tilde{L}^\vee)^+ \big)$.\\
(b) The map \eqref{eq:7.11} and the naturality of \eqref{eq:7.2} give rise to a commutative diagram
\begin{equation}\label{eq:7.13}
\begin{array}{ccc}
H^1 (\mc E,Z(\mc G_\der) \to \mc G) & \leftarrow & H^1 (\mc E,Z(\mc G_\der) \to \mc L) \\
\downarrow & & \downarrow \\
\Irr ( Z(\overline{G}^\vee)^+ ) & \leftarrow & \Irr ( Z(\tilde{L}^\vee)^+ )
\end{array}.
\end{equation}
Use that the same data $(\psi,z)$ realizes $\mc G^z$ as inner twist of $\mc G$ and $\mc L^z$ as an
inner twist of $\mc L$.
\end{proof}

Consider a Langlands parameter $\phi : \mb W_F \times SL_2 (\C) \to {}^L G$. Let
$Z_{\overline{G}^\vee} (\phi)$ be the preimage of $Z_{G^\vee}(\phi)$ under ${}^L \overline{G}
\to {}^L G$. Notice that
\[
Z_{\overline{G}^\vee} (\phi) \cap Z(\overline{G}^\vee ) = Z(\overline{G}^\vee )^+ .
\]
In this context the appropriate component group of $\phi$ is 
\[
\mc S_\phi^+ := \pi_0 (Z_{\overline{G}^\vee} (\phi)) .
\]
Now an enhancement of $\phi$ is defined to be an irreducible representation $\rho^+$ of 
$\mc S_\phi^+$. The groups $\overline{G}^\vee$ and $G^\vee$ act naturally on the set of such
enhanced L-parameters $(\phi,\rho^+ )$, with the same orbits. The set of $G^\vee$-association
classes of such enhanced L-parameters depends only on ${}^L G$, so we denote it by $\Phi^+ ({}^L G)$.

Recall from \cite[\S 3]{Bor} and \cite[Definition 1.3]{ABPS} that $G$-relevance of a Langlands
parameter $\phi$ can be formulated in terms of parabolic subgroups of ${}^L G$ that contain
the image of $\phi$. A nice aspect of rigid inner twists is that they enable a canonical 
definition of $G$-relevance of $\phi$ in terms of enhancements. 

In view of the surjectivity of \eqref{eq:7.1}, we may assume without loss of generality that
$G$ is quasi-split. We abbreviate a rigid inner twist $(\mc G^z,\psi,z)$ of $\mc G$ to 
$(\mc G^z,z)$. By the compatibility of \eqref{eq:7.2} and \eqref{eq:7.3}, 
$\chi_z \big|_{Z ({G^\vee}_\Sc)^{\mb W_F}} = \chi_{\mc G^z}$.
We say that $(\phi,\rho^+) \in \Phi^+ ({}^L G)$, or just the enhancement $\rho^+$, is relevant
for $(\mc G^z,z)$ if the character of $Z(\overline{G}^\vee )^+$ determined by $\rho^+$ via
the natural map $Z(\overline{G}^\vee )^+ \to \mc S_\phi^+$ equals $\chi_z$. 

We denote the subset of $\Phi^+ ({}^L G)$ that is relevant for $(\mc G^z,z)$ by $\Phi^+ (G^z, z)$. 
By design every enhancement $\rho^+$ is relevant for exactly one rigid inner twist of $\mc G$ 
(up to equivalence). That yields a natural decomposition
\begin{equation}\label{eq:7.4}
\Phi^+ ({}^L G) = \bigsqcup\nolimits_{z \in H^1 (\mc E,Z(\mc G_\der) \to \mc G)} \Phi^+ (G^z, z) .
\end{equation}
For comparison, in Section \ref{sec:Lpar} we fixed an inner twist $\mc G$ of a quasi-split group
and we let $\chi_{\mc G}$ be the associated character of $Z( {G^\vee}_\Sc )^{\mb W_F}$. To define
relevance of enhancements there, we picked some extension $\chi_{\mc G}^e$ of $\chi_{\mc G}$ to
$Z ({G^\vee}_\Sc)$ and used that to pin down central characters of $\mc S_\phi$-representations.
That works fine, but the freedom in the choice of $\chi_{\mc G}^e$ means that it is not
entirely canonical when $Z({G^\vee}_\Sc )^{\mb W_F} \neq Z({G^\vee}_\Sc)$. In that case, if one
picks a set of representatives for the inner forms of $\mc G$, not all characters 
of $Z({G^\vee}_\Sc)$ will occur in the union of the $\Phi_e (G')$ over the selected $G'$. 
That prevents a decomposition like \eqref{eq:7.4} with $\mc S_\phi$ instead of $\mc S_\phi^+$,
in the most general case at least.

Let us compare the two kinds of relevant enhancements of $\phi$.  
For a group $S$ with a central subgroup $Z$ and $\chi \in \Irr (Z)$ we write
\[
\Irr (S,\chi) = \{ (\pi,V) \in \Irr (S) : \pi (z) = \chi (z) \, \mr{id}_V \; \forall z \in Z \} .
\]

\begin{lem} \textup{\cite[(4.6) and (4.7)]{Kal2}}
\label{lem:7.1} \\
Let $(\mc G^z,z)$ be a rigid inner twist of a quasi-split connected reductive $F$-group $\mc G$. 
Choose an extension $\chi_{\mc G^z}^e \in \Irr (Z({G^\vee}_\Sc))$ of $\chi_{\mc G^z} \in \Irr 
\big( Z({G^\vee}_\Sc )^{\mb W_F} \big)$ which agrees with $\chi_z$ on $Z({G^\vee}_\Sc)^+$.
\enuma{
\item There is a canonical group isomorphism
\[
\mc S_\phi \to \mc S_\phi^+ \underset{Z ({G^\vee}_\Sc)^+}{\times} Z({G^\vee}_\Sc) 
\]
\item When $\phi$ is $G^z$-relevant, part (a) induces a bijection
\[
\Irr \big( \mc S_\phi,\chi_{\mc G^z}^e \big) \longleftrightarrow \Irr \big( \mc S_\phi^+, \chi_z \big) .
\]
Here $\chi_{\mc G^z}^e$ is considered as a character on the image of $Z({G^\vee}_\Sc)$ in $\mc S_\phi$, 
and similarly for $\chi_z$. When $\phi$ is not $G^z$-relevant, these two sets are empty.
}
\end{lem}

As a consequence of Lemma \ref{lem:7.1}.b, there is a natural bijection
\begin{equation}\label{eq:7.9}
\Phi^+ (G^z,z) \to \Phi_e (G^z) ,
\end{equation}
whenever $\mc G^z, z$ and $\chi_{\mc G^z}^e$ are as in Lemma \ref{lem:7.1}. By Lemma \ref{lem:7.4}
the relevance condition is compatible with passage to Levi subgroups. Thus all aspects of 
the local Langlands correspondence discussed in this paper can be viewed equally well in terms
of rigid inner twists $(\mc G^z,z)$ and enhanced L-parameters $\Phi^+ (G^z,z)$.
The only topic that needs some further clarification is the cuspidal support map for enhanced 
L-parameters figuring in Theorem \ref{thm:1}.g.

\begin{lem}\label{lem:7.2}
The cuspidal support map for $\Phi_e (G^z)$ can also be defined for $\Phi^+ (G^z,z)$, retaining
all its properties.
\end{lem}
\begin{proof}
The construction of this map in \cite[\S 7]{AMS1} and Lemma \ref{lem:7.1} show us how 
to proceed:
\begin{itemize}
\item translate $(\phi,\rho^+) \in \Phi^+ (G^z,z)$ to $( \phi, \rho ) \in \Phi_e (G_z)$ with
Lemma \ref{lem:7.1},
\item compute its cuspidal support $(L, \phi_v, q\epsilon)$ as in \cite[Definition 7.7]{AMS1},
\item translate $q\epsilon$ to $q\epsilon^+ \in \Irr (\mc S_{\phi_v}^+)$, using Lemma \ref{lem:7.1} 
and $\chi_{\mc L,z}$,
\item define the cuspidal support of $(\phi,\rho^+)$ to be $(L,\phi_v,q\epsilon^+)$.
\end{itemize}
The composition of these steps preserves the $Z(\overline{G}^\vee)^+$-character $\chi_z$. Together 
with Lemma \ref{lem:7.4} that means that $(\phi_v,q\epsilon^+)$ is relevant for the Levi subgroup 
$L$ of $G^z$.
\end{proof}

We write $\Phi_\nr^+ (G^z,z) = \{ (\phi,\rho^+) \in \Phi^+ (G^z,z) : \phi \in \Phi_\nr (G^z) \}$
and define $\Phi^+_\nr ({}^L G) \subset \Phi^+ ({}^L G)$ similarly.

\begin{thm}\label{thm:7.3}
Let $\mc G$ be a quasi-split connected reductive $F$-group and let $(\mc G^z,z)$ be a rigid inner
twist of $\mc G$. There exists a bijection
\[
\begin{array}{ccc}
\Irr_\unip (G^z) & \longleftrightarrow & \Phi^+_\nr (G^z,z) \\
\pi & \mapsto & (\phi_\pi, \rho^+_\pi)
\end{array}
\]
with all the properties from Theorem \ref{thm:1}.

When we let $(\mc G^z,z)$ run over all rigid inner twists of $\mc G$ up to equivalence, 
we obtain a bijection
\[
\bigsqcup\nolimits_{z \in H^1 (\mc E,Z(\mc G_\der) \to \mc G)} \Irr_\unip (G^z) 
\longleftrightarrow \Phi_\nr^+ ({}^L G) .
\]
\end{thm}
\begin{proof}
The first bijection is the composition of Theorem \ref{thm:1} and \eqref{eq:7.9}. By Lemmas \ref{lem:7.4},
\ref{lem:7.1} and \ref{lem:7.2} it enjoys the same properties as in Theorem \ref{thm:1}.

Then the second bijection follows from \eqref{eq:7.4}.
\end{proof}

\end{document}